\documentclass[10pt,reqno,a4paper]{mcom-l}
\usepackage{amsmath,amsthm,verbatim,amscd,amssymb,hyperref,graphicx}
\usepackage{exscale,color}
\usepackage{esint}
\usepackage{cases}
\usepackage{subfig}
\usepackage{epstopdf}
\usepackage{comment}
\usepackage{scalerel,stackengine}
\usepackage{appendix}
\usepackage{bm}% for bold omega
\newcommand{\nc}{\newcommand}
\nc{\MBB}{\mathbb}
\nc{\MBF}{\mathbf}
\nc{\MBM}[1]{\mbox{\boldmath $#1$}}
\nc{\MCAL}{\mathcal}
\nc{\RR}{\mathbb{R}}
\nc{\RA}{\rightarrow}
\nc{\OVR}{\overrightarrow}
\nc{\OVL}{\overline}
\nc{\UDL}{\underline}
\nc{\eps}{\varepsilon}
\nc{\OMG}{\omega}
\nc{\Lap}{\Delta}
\nc{\DPS}{\displaystyle}
\nc{\MSGN}{\mbox{sgn}}
\nc{\D}{\partial}
\nc{\FD}[1]{\frac{\D}{\D#1}}
\nc{\FDO}[2]{\frac{\D^{#2}}{\D^{#2}#1}}
\nc{\PAD}[2]{\frac{\D#2}{\D#1}}
\nc{\Grad}{\nabla}
\nc{\abs}[1]{\lvert#1\rvert}
\nc{\bigabs}[1]{\big\lvert#1\big\rvert}
\nc{\biggabs}[1]{\bigg\lvert#1\bigg\rvert}
\nc{\Bigabs}[1]{\Big\lvert#1\Big\rvert}
\nc{\Biggabs}[1]{\Bigg\lvert#1\Bigg\rvert}
\nc{\ip}[1]{\langle#1\rangle}
\nc{\bigip}[1]{\big\langle#1\big\rangle}
\nc{\biggip}[1]{\bigg\langle#1\bigg\rangle}
\nc{\Bigip}[1]{\Big\langle#1\Big\rangle}
\nc{\Biggip}[1]{\Bigg\langle#1\Bigg\rangle}
\nc{\bigpa}[1]{\big(#1\big)}
\nc{\biggpa}[1]{\bigg(#1\bigg)}
\nc{\Bigpa}[1]{\Big(#1\Big)}
\nc{\Biggpa}[1]{\Bigg(#1\Bigg)}
\nc{\bigmpa}[1]{\big[#1\big]}
\nc{\biggmpa}[1]{\bigg[#1\bigg]}
\nc{\Bigmpa}[1]{\Big[#1\Big]}
\nc{\Biggmpa}[1]{\Bigg[#1\Bigg]}
\nc{\bigbra}[1]{\big\{#1\big\}}
\nc{\biggbra}[1]{\bigg\{#1\bigg\}}
\nc{\Bigbra}[1]{\Big\{#1\Big\}}
\nc{\Biggbra}[1]{\Bigg\{#1\Bigg\}}
\nc{\bigparv}[2]{\bigpa{#1}\big\rvert_{#2}}
\nc{\biggparv}[2]{\biggpa{#1}\bigg\rvert_{#2}}
\nc{\Bigparv}[2]{\Bigpa{#1}\Big\rvert_{#2}}
\nc{\Biggparv}[2]{\Biggpa{#1}\Bigg\rvert_{#2}}
\nc{\bigparvt}[3]{\bigpa{#1}\big\rvert_{#2}^{#3}}
\nc{\biggparvt}[3]{\biggpa{#1}\bigg\rvert_{#2}^{#3}}
\nc{\Bigparvt}[3]{\Bigpa{#1}\Big\rvert_{#2}^{#3}}
\nc{\norm}[1]{\lVert#1\rVert}
\nc{\bignorm}[1]{\big\lVert#1\big\rVert}
\nc{\biggnorm}[1]{\bigg\lVert#1\bigg\rVert}
\nc{\Bignorm}[1]{\Big\lVert#1\Big\rVert}
\nc{\Biggnorm}[1]{\Bigg\lVert#1\Bigg\rVert}
\nc{\intdef}[2]{\int_{#1}^{#2}}
\makeatletter
\nc{\leqnomode}{\tagsleft@true}
\nc{\reqnomode}{\tagsleft@false}
\makeatother
\makeatletter
\newcount\my@repeat@count
\nc{\myrepeat}[2]{%
  \begingroup
  \my@repeat@count=\z@
  \@whilenum\my@repeat@count<#1\do{#2\advance\my@repeat@count\@ne}%
  \endgroup
}
\makeatother
\newtheorem{lemma}{Lemma}[section]%
\newtheorem{remark}[lemma]{Remark}%
\newtheorem{thm}[lemma]{Theorem}%
\stackMath
\nc\reallywidehat[1]{%
\savestack{\tmpbox}{\stretchto{%
  \scaleto{%
    \scalerel*[\widthof{\ensuremath{#1}}]{\kern-.6pt\bigwedge\kern-.6pt}%
    {\rule[-\textheight/2]{1ex}{\textheight}}%WIDTH-LIMITED BIG WEDGE
  }{\textheight}% 
}{0.5ex}}%
\stackon[1pt]{#1}{\tmpbox}%
}
\begin{document}

\title{Convergence of the boundary integral method for interfacial Stokes flow}

\author{David M. Ambrose}
\address{Department of Mathematics, Drexel University, Philadelphia, PA 09104} 
\email{ambrose@math.drexel.edu}

\author{Michael Siegel}
\address{Department of Mathematical Sciences and Center for Applied Mathematics and Statistics, New Jersey Institute of Technology, Newark, NJ 07102}
\email{misieg@njit.edu}

\author{Keyang Zhang}
\address{Department of Mathematical Sciences and Center for Applied Mathematics and Statistics, New Jersey Institute of Technology, Newark, NJ 07102}
\email{kz78@njit.edu}

\thanks{The authors acknowledge support from NSF grants DMS-1907684 (DMA) and DMS-1909407 (MS)}

\date{%Printed
\today}
\maketitle
\begin{abstract}
Boundary integral numerical methods are among the most accurate methods for interfacial Stokes flow, and are widely applied. They have the advantage that only the boundary of the domain must be discretized, which reduces the number of discretization points and allows the treatment of complicated interfaces. Despite their popularity, there is no analysis of the convergence of these methods for interfacial Stokes flow. In practice, the stability of discretizations of the boundary integral formulation can depend sensitively on details of the discretization and on the application of numerical filters. We present a convergence analysis of the boundary integral method for Stokes flow, focusing on a rather general method for computing the evolution of an elastic capsule, viscous drop, or inviscid bubble in 2D strain and shear flows.  
%variant of the method of \cite{HSB:2012} for computing the evolution of an elastic capsule in 2D strain and shear flows. 
The analysis clarifies the role of numerical filters in practical computations.
\end{abstract}

\section{Introduction}

Boundary integral (BI) methods are among the most popular methods for computing interfacial fluid flow.
They have been widely applied to compute the evolution of interfaces in potential flow, including the Kelvin-Helmholtz
and Rayleigh-Taylor instabilities \cite{Baker:1993}, \cite{Krasny:1986},
%, \cite{DIP:1982},  
Hele-Shaw flow \cite{HLS:1994}, \cite{Li:2007}, and water waves \cite{BMOS2:1980}, \cite{BHL:1996}.
% and crystal growth \cite{JSL},\cite{JStrain}.
They have also been extensively applied in Stokes flow to simulate the evolution of drops, bubbles, elastic capsules and vesicles  \cite{HSB:2012}, \cite{CPoz:1992},  \cite{STLVL:2010}, \cite{Veerapaneni:2009}.
%BI methods apply in problems that have a Green's function formulation. In fluid dynamics, this includes the important special cases of potential flow and Stokes flow (although see \cite{Askham:2020} for a recent extension of the BI method to the Navier-Stokes equations).
%They also apply to boundary value problems outside of fluid dynamics, including linear elasticity, electrostatics, electromagnetic wave propagation (i.e., the Helmholtz equation) and related areas \cite{ColtonKress:2013}, \cite{GreengardHelsing:1998}.
Boundary integral methods have been particularly important in micro- and bio- fluidic applications in which viscous forces are dominant over inertial ones. Such flows can therefore be accurately modeled by the Stokes equations. 
%This is a field with numerous applications, including drug delivery, lab-on-a-chip diagnostic devices, and microreactors \cite{Teh:2008}. Boundary integral simulations of Stokes equations have played an important role in the development of such microfluidic applications. 
%which motivates the analysis of this paper.  
Overviews of the BI method applied to Stokes flow in micro- and bio- fluidic applications are given in \cite{Kim:2013}, \cite{Pozrikidis2010}.

The main advantage of boundary integral methods is that they only involve  surface quantities, thereby reducing the dimension of the problem. 
%\cite{CPoz}, 
This simplifies the handling of complex geometries and reduces the number of discretization points.
Another significant advantage is that they can be made to have high accuracy. Boundary integral methods use a sharp interface formulation, which allows accurate treatment of the discontinuity in normal stress due to surface tension forces or elastic membrane stress at the interface. 
%In contrast, accurate treatment of this stress discontinuity remains a challenge for so-called interface capturing methods in which the interface does not necessarily  align with the computational grid.  Such methods, e.g., the  immersed boundary method \cite{Peskin:1977},  will typically regularize the interfacial forces. This regularization limits the accuracy to first order near the interface.  By imposing the jump conditions in the normal stress directly, the immersed interface method \cite{Leveque:1994} and cut finite element method \cite{Hansbo:2014} are able to achieve second order accuracy, but adapting these methods to achieve higher-order accuracy is difficult.
Spectrally accurate discretizations of boundary integral formulations are now routinely implemented for 2D interfacial Stokes flow of drops and bubbles, see,
e.g.,  \cite{Crowdy:2005}, \cite{MCAK1:2001}, \cite{MCAK2:2002}, \cite{Ojala:2015}, \cite{PST:2019}, \cite{XBS:2013} and references therein. Spectral or high-order boundary integral methods for inextensible vesicles or elastic capsules are  provided by  \cite{HSB:2012},  \cite{Marple:2016}, \cite{Quaife:2016}, and  \cite{Veerapaneni:2009}.
High order accurate discretizations of axisymmetric and 3D flow problems, although still a subject of current research, are increasingly common  \cite{Dodson:2009},  
 \cite{SorgentoneTornberg:2018},   \cite{Veerapaneni:2009a}, \cite{Veerapaneni:2011}, \cite{YoungHao:2012}.
%TODO: add references SorgentoneTornberg. 
As a result, boundary integral methods are a good choice in problems that demand high accuracy.

One of the greatest challenges in the practical implementation of boundary integral methods for time-evolution problems is that they are sensitive to numerical instabilities.
If left uncontrolled, these instabilities will dominate and adversely affect the accuracy of computations.  Numerical instabilities have been commonly observed in boundary integral computations for inviscid interfacial flow \cite{BMOS:1982}, \cite{BHL:1996},  \cite{JWD:1992}, \cite{LHC:1976}, \cite{DIP:1982}, and in spectrally accurate computations for Stokes flow  \cite{Crowdy:2005}, \cite{MCAK1:2001}, \cite{MCAK2:2002}, \cite{Ojala:2015}, \cite{PST:2019},  \cite{XBS:2013}.   They are typically controlled 
%(although not always eliminated) 
by application of numerical filtering or by deliasing through spectral padding.  Computations of interfacial flow with surface or elastic membrane tension can be even  more sensitive to numerical instabilities due to the presence of nonlinear terms with high-order spatial derivatives.   A major aim of the analysis presented here is to clarify the role of spatial discretization and filtering in controlling these aliasing-type instabilities for interfacial Stokes flow.

There are relatively few analyses of the stability and  convergence of BI methods for multi-phase flow simulations. This is in part due to the difficulty of the analysis involving nonstandard (nonlocal) governing equations. Hou, Lowengrub and Krasny \cite{Hou:1991} prove the convergence of a BI method for vortex sheets in inviscid flow without surface tension.  
Baker and Nachbin \cite{GBAN:2006} 
%apply normal mode analysis to the linearized discrete equations 
%about equilibrium to study the evolution of small  perturbations to  a flat vortex sheet in inviscid flow, and 
%to 
identify common reasons for numerical instability in the same problem as \cite{Hou:1991} when surface tension is present.
In fundamental work, Beale, Hou and Lowengrub \cite{BHL:1996} prove convergence in the fully nonlinear regime of a BI method for water waves in two dimensions both with and without surface tension. 
They 
%adapted  a framework developed in \cite{BHL2:1993} to study the linearized motion of perturbations about an arbitrary smooth solution, and 
discovered that delicate balances must be sustained among terms in singular integrals and derivatives at the discrete level in order to preserve numerical stability.
They also noticed that numerical filtering is necessary at certain places to prevent the discretization from producing new instabilities in the high modes.
%This filtering depends on the particular approach for approximating spatial derivatives and quadrature rules for singular integrals.
% Begin insert 1
Ceniceros and Hou \cite{CH:1998} extend the analysis of \cite{BHL:1996} to include two-phase flow and surface tension, and
%, using the $\theta-s_{\alpha}$ formulation of \cite{HLS}.
% End insert 1
Hou and Zhang \cite{HZ:2002} generalize the analysis of \cite{BHL:1996} to 3D.
Other convergence analyses have been  performed for Darcy-law flow problems.
Hao et al. \cite{Hao2018} show convergence of a boundary integral method for a generalized Darcy-law model of 2D tumor growth.
Ambrose, Liu and Siegel \cite{ALS2017:2017} prove convergence of a boundary integral method for 3D Darcy-law flow with surface tension \cite{Ambrose2013}. 
% TODO: cite [4]?

Despite the significance of the above-mentioned  convergence studies for BI methods in inviscid and Darcy-law interfacial flow, there is no convergence analysis that we are aware of for the important case of interfacial Stokes flow.
In this paper, we provide such an analysis.
The main difficulty of this analysis, compared to previous convergence studies for water waves, is a more complicated boundary integral formulation for the Stokes problem, and the presence of high derivatives in the boundary condition for an elastic membrane.

In the analysis of the stability of our method, we make significant use of the stabilizing effects of the highest derivative or leading-order terms (so-called parabolic smoothing) to control lower-order terms.
As in the water wave problem of \cite{BHL:1996}, 
%it was found that strategically placed numerical filtering is necessary to prevent instabilities due to aliasing error from growing and destroying the computation.
%We similarly 
we find that a targeted application of numerical filtering is necessary to prove stability in the Stokes-interface problem.
This is consistent with numerical implementations of spectrally accurate methods for the evolution of drops, bubbles and elastic capsules in Stokes flow, e.g.,  \cite{HSB:2012}, \cite{MCAK1:2001}, \cite{Ojala:2015}, \cite{PST:2019}, \cite{XBS:2013}, which also find the need for some form of numerical filtering or dealiasing for stability.
However, to make use of the parabolic smoothing in the elastic capsule problem and to minimize the amount of numerical filtering, we find it important that filtering \textit{not} be applied to the leading-order or highest derivative terms. 
Based on our analysis, we present a numerical scheme that utilizes a minimal amount of filtering yet is provably stable, even in the fully nonlinear regime. Note that the specific  filtering applied in our method is not unique, and other filtering techniques (e.g., zero padding) may give stable schemes.

In our convergence analysis, we consider a rather general boundary integral formulation which governs the time-dependent evolution 
of a Hookean elastic capsule in 2D Stokes flow for an externally imposed straining or shearing flow, but also encompasses the deformation of a drop or bubble with constant surface tension. An elastic capsule is a drop or bubble that is enclosed by a thin, elastic membrane and suspended in an external fluid.
It serves as a simple mechanical model of a cell or vesicle that is deformed by a fluid flow.  Numerical studies of capsules in fluid flow performed with various membrane constitutive laws include \cite{BealeStrain:2008}, \cite{Dodson2009:2009}, \cite{HSB:2012},  \cite{Veerapaneni:2009}, \cite{Walter2010:2010}. If the parameter governing membrane bending stress is set to zero, one recovers the governing equations and numerical method for two-fluid flow about a drop or bubble with constant interfacial tension.  This case is therefore also included in our convergence analysis.

The algorithm we analyze is closely based on a spectrally accurate numerical method for the evolution of a drop or bubble in an extensional flow that was developed by Kropinski \cite{MCAK1:2001}, \cite{MCAK2:2002} and is extensively used, see, e.g.,  
\cite{Ojala:2015},  \cite{PST:2019}, \cite{XBS:2013}.
The method of Kropinski makes use of a complex-variable description of the problem known as the Sherman-Lauricella formulation, as well as a construction due to Hou, Lowengrub and Shelley \cite{HLS:1994} in which the interface is prescribed by its tangent angle $\theta(\alpha,t)$ and an equal-arclength parameter $\alpha$, so that $\PAD{\alpha}{s}$ is constant in $\alpha$ (here $s(\alpha,t)$ measures arclength from a reference point at $\alpha=0$).
%This BI formulation is called the $\theta,~s_{\alpha}$ formulation.
This so-called arclength-angle formulation was originally developed to overcome numerical stiffness in the time-discretization, but is also convenient for analysis \cite{Ambrose2003}, \cite{Ambrose2017}.  We further adapt this formulation to Stokes flow with elastic surfaces. A complication of this approach is that the location of material points on the interface must be tracked to determine the elastic or ``stretching" tension in the membrane. This is not readily available from the equal arclength parameterization. Following \cite{HSB:2012}, we introduce a backwards map $\alpha_0(\alpha,t)$ which gives the location of the material parameter $\alpha_0$ in terms of the equal arclength parameter $\alpha$, and derive a time evolution equation for $\alpha_0(\alpha,t)$. The discrete version of this time-evolution equation is then incorporated into the  energy estimates to show stability of our method. The algorithm  presented here also generalizes a spectrally accurate method for the evolution of an elastic capsule in an extensional flow, which was developed in \cite{HSB:2012}, to include a viscous interior fluid and nonzero membrane bending stress. 

Following the earlier convergence studies for inviscid flow, our analysis is discrete in space and continuous in time.  The main result is contained in Theorem 1, which proves that the numerical method with filtering converges to the exact solution with spectral accuracy. The convergence proof follows the general framework of \cite{BHL:1996} and uses energy estimates in discrete Sobolev spaces. It relies on the smoothness of the underlying solution to the continuous evolution problem, see \cite{Lin:2019}, \cite{Mori:2019} for relevant results on the existence and regularity of the continuous problem.
%Existence and regularity results for the (continuous) time-dependent interfacial Stokes flow problem are rather limited.
%Recently there have been a number of existence and regularity results for the so-called Peskin problem, which describes the motion of a closed 1D elastic string with bending and stretching energy in a 2D Stokes flow \cite{Lin:2019}, \cite{Mori:2019}.
%This is closely related to the problem considered here. 
When there is a jump in viscosity between the fluid in the  interior of the capsule or drop and the exterior fluid, an additional system of Fredholm integral equations must be solved to obtain
the density in the boundary integrals of the Sherman-Lauricella formulation. Analysis of the discrete version  of this system of integral equations presents an additional complication of the proof. However, we are able to show that the discrete system is invertible and that the inverse operator is bounded for sufficiently small viscosity contrast. This provides a convergence proof of the full evolution problem, including viscosity contrast, as long as that contrast is sufficiently small.

The governing equations for our problem are presented in Section 2, and 
the BI formulation is given in Section 3.
For our BI formulation, we present in Section 4 a spectrally accurate numerical discretization. 
Several preliminary lemmas are presented in Section 5 which provide error estimates on numerical differentiation, integration, and filtering operators.
We then prove consistency of our numerical method in Section 6.
The statement of the main convergence theorem, Theorem \ref{MainThm}, is given in Section 7.
Some preliminary estimates used in the proof of stability are given Sections 8 and 9.
Evolution equations for the errors are presented in Section 10, and the  proof of stability (and hence Theorem \ref{MainThm}) by energy estimates in the special case of viscosity matched fluids and nonzero bending stress is given in Section 11. The case of unequal viscosities is discussed in Section 12. Modifications to the convergence analysis for the drop problem, with zero membrane bending stress and constant interfacial tension, are given in Section 13.
Concluding remarks are provided in Section 14.  Proofs of critical lemmas and estimates of nonlinear terms in the variation of velocity are given in the Appendix.

\section{Problem formulation}
We present the governing equations for a single elastic capsule in 2D Stokes flow.
The exterior fluid domain is denoted by $\Omega$, and we use a superscript $i$ for variables and parameters in the inner fluid.  The membrane surface is given by $\D\Omega=\gamma$.

The drop and exterior fluid are assumed to have the same density, so gravitational effects are absent.
On $\gamma$ the unit normal vector $\MBF{n}$ points toward the exterior fluid.
The unit tangent $\MBF{t}$ points in the direction such that the interior fluid is to the right as $\gamma$ is traversed clockwise.
We define an angle $\theta$ measured counterclockwise positive from the positive $x-$axis to $\MBF{t}$.
The geometry is illustrated in Figure 1.
\begin{figure}[ht!] 
\begin{center}  
\includegraphics[scale=0.5]{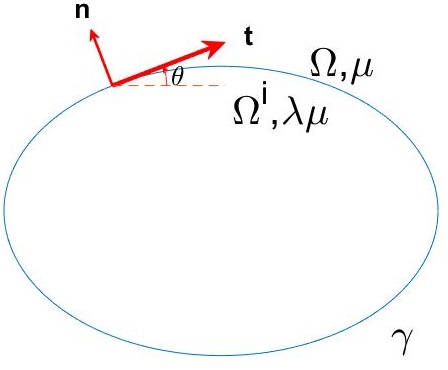}
\caption{\label{fig-1}Fluid drop with viscosity $\lambda\mu$ occupying region $\Omega^i$ is immersed in a fluid with viscosity $\mu$ occupying region $\Omega$}
\end{center}
\end{figure}
The local curvature of the interface is $\kappa=-\PAD{s}{\theta}$ and is positive when the shape is convex.
Here, $s$ is an arclength parameter that increases as $\gamma$ is traversed clockwise.

In dimensionless form, the Stokes equations governing fluid flow are
\begin{numcases}{ }% add a spacebar here to avoid the Laplacian symbol come in front of the brackets!!!
\Lap \MBF{u}=\Grad p\mbox{, }\Grad\cdot\MBF{u}=0\mbox{,}\quad \MBF{x}\in\Omega\mbox{,} \label{Outer_stokes}
\\
\lambda\Lap \MBF{u}^i=\Grad p^i\mbox{, }\Grad\cdot\MBF{u}^i=0\mbox{,}\quad\MBF{x}\in\Omega^i\mbox{,} \label{Inner_stokes}
\end{numcases}
where $p(\MBF{x})$ and $\MBF{u}(\MBF{x})$ are the pressure and velocity fields and $\lambda=\frac{\mu^i}{\mu}$ is the viscosity ratio.
The fluid velocity is taken to be continuous across the interface,
i.e., $\MBF{u}(\MBF{x})=\MBF{u}^i(\MBF{x})$ for $\MBF{x}\in\gamma$.

The area enclosed by the capsule is conserved, and lengths are nondimensionalized by the radius $R$ of the circular capsule with the same area.
Velocities are nondimensionalized by $U$, where $U$ will be specified below.
Time is nondimensionalized by $\frac{R}{U}$, and pressure by $\frac{U\mu}{R}$.
At $t=0$ the capsule can have arbitrary shape and membrane tension.

The no slip condition on the capsule surface is given by
\begin{equation}\label{interfacevel}
\frac{d\MBF{x}}{dt}=\MBF{u}(\MBF{x},t) \quad \mbox{for } \MBF{x}\in\gamma\mbox{,}
\end{equation}
Equation (\ref{interfacevel}) satisfies the kinematic condition that $\frac{d\MBF{x}}{dt}\cdot \MBF{n}=\MBF{u}\cdot\MBF{n}$ on $\gamma$. 
The far-field boundary condition is taken to be a general incompressible linear flow
\begin{equation} \label{ff}
\DPS\lim_{\abs{\MBF{x}}\RA\infty} \MBF{\MBF{u}(\MBF{x})}=\MBF{u}_{\infty}(\MBF{x})
=
\begin{pmatrix}
Q & B+\frac{G}{2} \\
B-\frac{G}{2} & -Q
\end{pmatrix}
\MBF{x} + O(|\mathbf{x}|^{-2})~~\mbox{as}~\mathbf{x} \rightarrow \infty \mbox{,}
\end{equation}
% Begin insert 2
where the dimensionless parameters $(Q,B,G)$ are equal to their dimensional counterparts $Q_{\infty}$, etc., times the time scale $\frac{R}{U}$; i.e. $(Q,B,G)=\frac{R}{U}(Q_{\infty},B_{\infty},G_{\infty})$.
% End insert 2
The far-field flow is a pure strain if $B=G=0$, and a linear shear flow if $Q=0$ and $G=2B$.
At the elastic membrane interface, we have the additional boundary condition that the total interfacial stress $\MBF{f}$ is balanced by the jump in fluid stress across the interface,
\begin{equation}\label{jumpeqn}
\bigmpa{T\cdot \MBF{n}}=\MBF{f}\mbox{,}
\end{equation}
where $T=-p+2 E_{ij}$ and $T^i=-p^i+2\lambda E_{ij}^i$, and where
\begin{equation}
E_{ij}=\frac{1}{2}\Bigpa{\PAD{x_j}{u_i}+\PAD{x_i}{u_j}}\mbox{,}
\end{equation}
is the stress tensor.
Here $\bigmpa{\cdot}$ denotes the jump
\begin{equation}
\bigmpa{\MBF{g}}=\MBF{g}-\MBF{g}^i\quad \mbox{for } \MBF{x}\in\gamma\mbox{.}
\end{equation}

% Begin insert 3
An expression for the interfacial stress $\MBF{f}$ on the right hand side of (\ref{jumpeqn}) is obtained in \cite{Pozrikidis2001:2001} by an analysis of interfacial forces and torques.
%or more precisely, interfacial tension and bending moment.
The result is given in equation (3.16) of \cite{Pozrikidis2001:2001}, which in our notation is
\begin{align}\label{interface_force_eqn_5b_1}
\MBF{f}=-\FD{s}\bigpa{
\MCAL{S}\MBF{t}+q_B\MBF{n}
}
\end{align}
where $\MCAL{S}=\MCAL{S}(s)$ is the surface tension in terms of an arclength parameter $s$, and $q_B(s)=\frac{d m_B}{ds}$ with $m_B=m_B(s)$ the bending moment.
The constitutive equation for the bending moment $m$ is assumed to be the simple linear relation
\begin{align}\label{m_B_eqn_5c_1}
m_B(s)=\kappa_B \kappa(s)
\end{align}
where $\kappa_B$ is the (dimensionless) bending modulus, and $\kappa(s)$ is the interfacial curvature.
% End insert 3
For the sake of simplicity, we consider a membrane with a Hookean or linear elastic response, for which the dimensional tension is given by \cite{CPoz:1992}
\begin{align} \label{tension_dimensional}
\tilde{\MCAL{S}}=E(\eta-1)\mbox{, }\eta=\PAD{s_R}{s}\mbox{.}
\end{align}
Here $\eta$ is the stretch ratio between arclength $s$ of the membrane at time $t$ and arclength $s_R$ in a reference configuration in which there is no tension in the membrane, and $E$ is the modulus of elasticity.
The tension is nondimensionalized by $E$, so that in dimensionless form
\begin{align}\label{tension_const_defn}
\MCAL{S}=\eta-1\mbox{.}
\end{align}
% Begin insert 4
We also now define the characteristic velocity which is used for nondimensionalization as $U=\frac{E}{\mu}$.
% End insert 4

\section{Boundary Integral Formulation}
%A boundary integral formulation for an elastic capsule or vesicle in 2D Stokes flow with an inextensible membrane is given by Veerapaneni et al. \cite{Veerapaneni}.
%Their formulation uses a single layer potential $S[\MBF{f}](\MBF{x})$ to represent the velocity, where:
%\begin{equation}\label{SLP}
%S[\MBF{f}]=\int_{\gamma} %G_s(\MBF{x},\MBF{y})\MBF{f}(\MBF{y})ds_{\MBF{y}}\mbox{,}
%\end{equation}
%and where the 2D Stokes free space kernel $G_s$ is given by:
%\begin{equation}
%G_s(\MBF{x},\MBF{y})=\frac{1}{4\pi} \Bigpa{-\ln \rho %\MBF{I}+\frac{\MBF{r}\otimes\MBF{r}}{\rho^2} }\mbox{,}
%\quad \MBF{r}=\MBF{x}-\MBF{y}\mbox{,}
%\quad \rho=\norm{\MBF{r}}_2\mbox{.}
%\end{equation}
%Because of the log singularity in $G_s$, Veerapaneni et al. \cite{Veerapaneni} employ a special form of Gauss$-$Legendre quadrature due to Alpert \cite{AL} to discretize the integral.
%While accurate, this quadrature is rather complicated and for this and other reasons an analysis of the discrete equations for their method is difficult.
For our boundary integral method, we adapt the Sherman-Lauricella formulation \cite{GKM:1996}, \cite{MCAK1:2001}, \cite{SGM:1964} to the capsule-membrane problem.
% modified references, check
This is a complex variable formulation for which the primitive variables are expressed in terms of an integral over a complex density that is defined on the drop interface and satisfies a second kind Fredholm  equation.
It has been extensively used to solve Stokes equations for multi-phase  fluid flow, see e.g.,  \cite{MCAK1:2001}, \cite{MCAK2:2002}, \cite{KL:2011}, \cite{Ojala:2015}, \cite{PST:2019}, \cite{XBS:2013}.

In the Sherman-Lauricella formulation, the complex fluid velocity $(u_1+i u_2)(\tau,t)$ for $\tau=x_1+i x_2$ on the time evolving  interface $\gamma$ is  written in terms of Cauchy-type integrals that contain a single complex density $\omega(\zeta,t)$ \cite{MCAK1:2001}, \cite{XBS:2013}: 
\begin{align}\label{velinteqn}
u\rvert_{\gamma}=(u_1+iu_2)\rvert_{\gamma}&=-\frac{1}{2\pi}\mbox{P.V.}\int_{\gamma} \omega(\zeta,t)\Bigpa{\frac{d\zeta}{\zeta-\tau}+\frac{d\OVL{\zeta}}{\OVL{\zeta}-\OVL{\tau}}}
\nonumber \\
&+\frac{1}{2\pi}\int_{\gamma} \OVL{\omega(\zeta,t)} d\frac{\zeta-\tau}{\OVL{\zeta}-\OVL{\tau}}
+(Q+iB)\OVL{\tau}-\frac{iG}{2}\tau.
\end{align}
The apparent singularity for $\zeta$ near $\tau$ in the second integral is removable, but in the first integral the P.V. indicates that it is to be interpreted as a Cauchy principal value integral.

The complex density $\omega$ satisfies an integral equation which is obtained by modifying the derivation of \cite{MCAK1:2001}, \cite{XBS:2013} for drops to incorporate the more general elastic membrane stress in (\ref{interface_force_eqn_5b_1}). If we denote the interface $\gamma$ by $\tau(s,t)$, then the final form that the equation takes can be written as
\begin{align}\label{finalforminteqn}
\omega(\tau,t)+&\frac{\beta}{2\pi i}\int_{\gamma} \omega(\zeta,t)d\ln\Bigpa{\frac{\zeta-\tau}{\OVL{\zeta}-\OVL{\tau}}}
+\frac{\beta}{2\pi i}\int_{\gamma} \OVL{\omega(\zeta,t)}d\frac{\zeta-\tau}{\OVL{\zeta}-\OVL{\tau}}
\nonumber \\
&=-\frac{\chi}{2}\left[ (\MCAL{S}+\kappa_B\kappa^2)\tau_s-\kappa_B\tau_{sss} \right]
-\beta(B-iQ)\OVL{\tau}-2\beta H(t)\mbox{,}
\end{align}
where $\beta=\frac{1-\lambda}{1+\lambda}$ and $\chi=\frac{1}{1+\lambda}$, and a subscript $s$ denotes derivative. The expression in brackets on the right hand side of (\ref{finalforminteqn}) is the integral with respect to $s$ of the interfacial stress in (\ref{interface_force_eqn_5b_1}), represented using complex variables. 
%This comes from integrating the stress balance equation (\ref{jumpeqn}) during the derivation of (\ref{finalforminteqn}).  
When $\kappa_B=0$ and $\MCAL{S}=constant$, (\ref{finalforminteqn}) reduces to the corresponding equation for a drop interface  \cite{XBS:2013}.  
The apparent singularity at $\zeta=\tau$ in the two integrals on the left$-$hand side is removable. 
We set
\begin{equation}
H(t)=\frac{1}{2}\int_{\gamma} \omega(\zeta,t)ds;
\end{equation}
as demonstrated in \cite{MCAK1:2001}, this choice
removes a rank deficiency of the integral equation (\ref{finalforminteqn}) in the limit $\lambda=0$ of an inviscid drop and is consistent with $H(t)\equiv 0$, which is a consequence of the constant area of the interior region $\Omega^i$.
% TODO: changed references, check!

The fluid velocity on the interface, in terms of its normal and tangential components $u_n$ and $u_s$, is $\MBM{u}=u_n\MBM{n}+u_s\MBM{t}$, where the complex counterparts of the unit vectors $\MBM{n}$ and $\MBM{t}$ are $n$ and $s_T$ with $s_T=-in=\D_s\tau$.
It follows that
\begin{equation}\label{UComponentEqn}
u_n=\mbox{Re}\bigbra{u\rvert_{\gamma}\OVL{n}}
\quad \mbox{and} \quad
u_s=-\mbox{Im}\bigbra{u\rvert_{\gamma}\OVL{n}}
\mbox{,}
\end{equation}
on the interface $\gamma$.

For the numerical discretization of (\ref{finalforminteqn}), we introduce an equal arclength parametrization of the interface $\gamma$.
This is constructed following Hou, Lowengrub and Shelley \cite{HLS:1994}.
%and is an essential component of the method for removing the stiffness.
The spatial parametrization of the interface is given by $\alpha\in[-\pi,\pi]$, and a point $\tau$ on the interface has Cartesian coordinates $(x_1,x_2)$, so that $\tau(\alpha,t)=x_1(\alpha,t)+i x_2(\alpha,t)$.
The unit tangent vector $s_T$ and normal $n$ in complex form are $s_T=\PAD{s}{\tau}=\frac{\tau_{\alpha}}{s_{\alpha}}=\exp(i\theta)$
and $n = i s_T = i\exp(i\theta)$.
Differentiation of $\tau_\alpha=s_\alpha e^{i \theta}$  with respect to time implies that
\begin{align}\label{tau_alpha_t_eqn_9_1}
\tau_{\alpha t}=s_{\alpha t}e^{i\theta}+s_{\alpha}\theta_tie^{i\theta}\mbox{.}
\end{align}
When $\tau=\tau_m$ is a material point on the interface its velocity is equal to the local fluid velocity, per (\ref{interfacevel}), so that differentiation with respect to time implies that
\begin{equation}\label{complexevolveqn}
\frac{d\tau_m}{dt}=u_n i e^{i\theta}+u_s e^{i\theta}\mbox{,}
\end{equation}
where the subscript $m$ is used to denote material point.

However, the shape of the evolving interface is determined by the normal velocity component $u_n$ alone.
Although $u_s$ has physical meaning as the tangential component of the fluid velocity, if $u_s$ is replaced by any other smooth function $\phi_s(\alpha,t)$ in (\ref{complexevolveqn}), then $\tau$ still lies on the interface but is no longer a material point, and the role of $\phi_s$ is simply to implement a specific choice of  the interface parametrization via $\alpha$, without changing the interface shape or evolution.
% begin insert 6
The interfacial velocity generated by using $\phi_s$ instead of $u_s$ is denoted by $v$, and has complex form given by
\begin{align}\label{vel_eqn_9a_1}
v=\frac{d\tau}{dt}
=u_n i e^{i\theta}+\phi_s e^{i\theta}\mbox{.}
\end{align}
% end insert 6
Differentiation of (\ref{vel_eqn_9a_1}) with respect to $\alpha$ gives a second relation for $\tau_{\alpha t}$,
\begin{equation}\label{tau_alpha_t_eqn_30}
\tau_{\alpha t}=\bigpa{(\phi_s)_{\alpha}-u_n\theta_{\alpha}}e^{i\theta}
+\bigpa{(u_n)_{\alpha}+\phi_s\theta_{\alpha}}ie^{i\theta}\mbox{.}
\end{equation}
Equating (\ref{tau_alpha_t_eqn_9_1}) and (\ref{tau_alpha_t_eqn_30}), we have
\begin{align}
s_{\alpha t}&=(\phi_s)_{\alpha}-u_n\theta_{\alpha}\mbox{,}\label{S_ALPHA_EVOLVEQN} \\
\theta_t&=\frac{1}{s_{\alpha}}\bigpa{(u_n)_{\alpha}-\phi_s\theta_{\alpha}}\mbox{,} \label{THETA_EVOLVEQN}
\end{align}
where $\gamma$ is now described parametrically by $s=s(\alpha,t)$ and $\theta=\theta(\alpha,t)$ instead of $x_1=x_1(\alpha,t)$ and $x_2=x_2(\alpha,t)$.

The equal arclength frame is chosen by setting $s_{\alpha}=s_{\alpha}(t)$ to be spatially constant along the interface, so that it varies in time only.
Then since $s_{\alpha}$ is always equal to its mean around $\gamma$, it follows from (\ref{S_ALPHA_EVOLVEQN}) that
\begin{equation}\label{S_ALPHA_EVOLVINTEQN_ORG}
s_{\alpha t}=(\phi_s)_{\alpha}-u_n\theta_{\alpha}=-\frac{1}{2\pi}\intdef{-\pi}{\pi}u_n\theta_{\alpha'}d\alpha'\mbox{.}
\end{equation}
Integration of the second of these equations with respect to $\alpha$ implies that
\begin{equation}\label{PHI_S_ALPHA_T}
\phi_s(\alpha,t)=\D_{\alpha}^{-1}\bigpa{
u_n\theta_{\alpha}-\ip{u_n\theta_{\alpha}}
}
\mbox{,}
\end{equation}
where
\begin{align}
\ip{f}=\frac{1}{2\pi}\intdef{-\pi}{\pi} f(\alpha')d\alpha'
\end{align}
is the mean of $f$, $\D_{\alpha}^{-1}$ is defined for a function $f$
with zero mean as
\begin{align} \label{antideriv}
\D_{\alpha}^{-1} f=
\DPS\sum_{\substack{k=-\infty\\k\neq 0}}^{\infty}
\frac{\hat{f}_k}{ik} e^{ik\alpha}\mbox{,}
\end{align}
and $\hat{f}_k$ are the Fourier coefficients of $f$.
In (\ref{PHI_S_ALPHA_T}) an arbitrary function of time has been set so that $\phi_s(\alpha,t)$ has zero mean.
%This gives the required tangential velocity $\phi_s$ of the equal arc length frame.

When (\ref{PHI_S_ALPHA_T}) is substituted into (\ref{S_ALPHA_EVOLVEQN}) and (\ref{THETA_EVOLVEQN}), the system by which the dynamics of the interface is tracked becomes
\begin{align}
s_{\alpha t}=-\frac{1}{2\pi}\intdef{-\pi}{\pi}u_n\theta_{\alpha'}d\alpha'\mbox{,}\label{S_ALPHA_EVOLVINTEQN} \\
\theta_t =\frac{1}{\sigma}\Bigmpa{ \theta_{\alpha}\D_{\alpha}^{-1}\bigpa{
u_n\theta_{\alpha}-\ip{u_n\theta_{\alpha}}
} +(u_n)_{\alpha} }\mbox{.}\label{THETA_EVOLVINTEQN}
\end{align}

At each time step (\ref{S_ALPHA_EVOLVINTEQN}) and (\ref{THETA_EVOLVINTEQN}) are integrated forward in time, and $(s_{\alpha},\theta)$ are mapped to the Cartesian coordinates $(x_1,x_2)$ of points on $\gamma$.
The map is given by integration of $\tau_{\alpha}=s_{\alpha} e^{i\theta}$ with respect to $\alpha$ and is
\begin{align}
x_1(\alpha,t)=x_{1c}(t)+s_{\alpha}(t)\D_{\alpha}^{-1}(\cos(\theta(\alpha',t)))\mbox{,}\label{X1_POS_INT_EQN} \\
x_2(\alpha,t)=x_{2c}(t)+s_{\alpha}(t)\D_{\alpha}^{-1}(\sin(\theta(\alpha',t)))\mbox{,}\label{X2_POS_INT_EQN}
\end{align}
where $(x_{1c}(t),x_{2c}(t))$ is the constant Fourier mode of $(x_1(\alpha,t),x_2(\alpha,t))$, which is evolved from (\ref{vel_eqn_9a_1}) as
\begin{align}\label{vel_0_eqn_41}
\frac{d}{dt}(x_{1c}(t)+ix_{2c}(t))=\hat{v}_0(t)=\ip{v}\mbox{,}
\end{align}
where $\hat{v}_0(t)$ is the $k=0$ Fourier mode of interface velocity $v$.
\newline
\newline
\noindent\underline{\textit{Membrane tension}}.
\newline
\newline
A formula for the membrane tension $\MCAL{S}(\alpha,t)$ in terms of interface shape $\tau(\alpha,t)$ and the initial tension $\MCAL{S}(\alpha,0)$ is required to close the system of equations.
%(\ref{finalforminteqn})-(\ref{velinteqn}), %(\ref{PHI_S_ALPHA_T}), 
%(\ref{S_ALPHA_EVOLVINTEQN})-(\ref{vel_0_eqn_41}).
We obtain this formula by adapting the construction in \cite{HSB:2012}.

Recall that $\tau(\alpha,t)$ is a general nonmaterial parameterization of the interface at time $t$.
Introduce a parameterization $\tau(\alpha_p,0)$ of the initial profile in terms of a Lagrangian or material coordinate $\alpha_p$, and denote the location of the same material point at time $t>0$ by $\tau(\alpha_m(\alpha_p,t),t)$; this serves as a definition of a `forward' map $\alpha_m(\alpha_p,t)$.
We also define the `backward' map $\alpha_0(\alpha,t)$ such that $\tau(\alpha_0(\alpha,t),0)$ is the location at $t=0$ of the material point that at time $t$ is located at $\tau(\alpha,t)$.
It follows that $\alpha_m$ and $\alpha_0$ are one-to-one and inverses.

A formula for $\MCAL{S}(\alpha,t)$ that gives the tension in terms of the initial state of the membrane and the backwards map $\alpha_0(\alpha, t)$ is given by (4.14) in  \cite{HSB:2012}. In our notation, this formula takes the form
%Write $s_0$ for arclength at time $t=0$ and note from (\ref{tension_const_defn}) that for a Hooke's law membrane,
%\begin{align}
%\MCAL{S}(s_0,0)=\PAD{s_R}{s_0}-1\mbox{.}
%\end{align}
%Then for $t>0$,
%\begin{align}\label{tension_modulus_s_eqn_1}
%\MCAL{S}(s,t)=\PAD{s_0}{s}\PAD{s_R}{s_0}-1
%=\PAD{s_0}{s}\bigpa{
%1+\MCAL{S}(s_0,0)
%}
%-1\mbox{.}
%\end{align}
%We determine an equation for the time evolution of $\PAD{s_0}{s}$.
%The arclength at time $t$ is `
%\begin{align}
%s(\alpha,t)=\intdef{0}{\alpha}
%s_{\alpha}(\alpha',t)
%d\alpha'\mbox{,}
%\end{align}
%and the length of the same material arc at $t=0$ is:
%\begin{align}
%s_0(\alpha,t)=\intdef{\alpha_0(0,t)}{\alpha_0(\alpha,t)}
%s_{\alpha}(\alpha',0)
%d\alpha'\mbox{.}
%\end{align}
%Hence, in terms of $\alpha$,
%\begin{align}
%\PAD{s_0}{s}
%=\frac{s_{\alpha}(\alpha,t)}{s_{\alpha}(\alpha_0(\alpha,t),0)\alpha_0'(\alpha,t)}
%\mbox{,}
%\end{align}
%where $\alpha_0'(\alpha,t)=\PAD{\alpha}{\alpha_0}(\alpha,t)$, so (\ref{tension_modulus_s_eqn_1}) becomes
\begin{align}\label{tension_modulus_eqn_alpha_t}
\MCAL{S}(\alpha,t)
=\frac{s_{\alpha}(t)}{s_{\alpha}(0)\alpha_0'(\alpha,t)}
\bigpa{
1+\MCAL{S}(\alpha_0(\alpha,0),0)
}
-1\mbox{,}
\end{align}
where $\alpha_0'(\alpha,t)=\PAD{\alpha}{\alpha_0}(\alpha,t)$ and 
we have made use of the fact that $s_{\alpha}(\alpha,t)=s_{\alpha}(t)$ is spatially independent.

The formula for the membrane tension therefore requires an equation for the backward map $\alpha_0(\alpha,t)$.
First, note that, by definition of $\alpha_m$ and $\alpha_p$, the condition for the motion of a material particle becomes
\begin{align}
\frac{d}{dt}\tau(\alpha_m(\alpha_p,t),t)=u \mbox{,}
\end{align}
that is,
\begin{align}\label{complex_vel_eqn_tension_mod_1}
\PAD{t}{\tau}\Big\rvert_{\alpha}
+\PAD{\alpha}{\tau}\PAD{t}{\alpha_m}\Big\rvert_{\alpha_p}
=u \mbox{,}
\end{align}
at $\alpha=\alpha_m(\alpha_p,t)$.
An expression for $\PAD{t}{\tau}\Big\rvert_{\alpha}$ is given by (\ref{vel_eqn_9a_1}),
%i.e.,
%which describes the motion of $\tau(\alpha,t)$ at a fixed $\alpha$ with normal velocity $u_n$ and tangential velocity $\phi_s$.
%Substitution of (\ref{vel_eqn_9a_1}) 
and substituting this into (\ref{complex_vel_eqn_tension_mod_1}) yields the evolution equation for the forward map $\alpha_m(\alpha_p,t)$
\begin{align}\label{complex_vel_eqn_alpha_m_mod_1}
\PAD{t}{\alpha_m}\Big\rvert_{\alpha_p}
=\frac{1}{\tau_{\alpha}}\bigmpa{
u -(u_n i e^{i\theta}+\phi_s e^{i\theta})
}
\end{align}
at $\alpha=\alpha_m(\alpha_p,t)$.

The evolution of the backward map $\alpha_0(\alpha,t)$ is obtained by noting that $\alpha_m$ and $\alpha_0$ are inverses, so that differentiation of the identity $\alpha=\alpha_m(\alpha_0(\alpha,t),t)$ with respect to time keeping $\alpha$ fixed implies
\begin{align}\label{complex_vel_eqn_alpha_m_mod_2}
\PAD{t}{\alpha_m}\Big\rvert_{\alpha_p}
+\PAD{\alpha_m}{\alpha_p}\PAD{t}{\alpha_0}\Big\rvert_{\alpha}
=0\mbox{,}
\end{align}
where we have set $\alpha_p=\alpha_0(\alpha,t)$ in the first two derivatives.
Differentiation of the same identity with respect to $\alpha$ keeping $t$ fixed gives
\begin{align}\label{alpha_m_alpha_p_rel_eqn}
\PAD{\alpha_p}{\alpha_m}
=\Bigpa{
\PAD{\alpha}{\alpha_0}
}^{-1}\mbox{.}
\end{align}
Eliminating $\alpha_m$ in favor of $\alpha_0$ in (\ref{complex_vel_eqn_alpha_m_mod_1}), (\ref{complex_vel_eqn_alpha_m_mod_2}) and (\ref{alpha_m_alpha_p_rel_eqn}) gives the initial value problem for the backward map,
\begin{align}\label{d_alpha_0_dt_eqn_12a_1}
\PAD{t}{\alpha_0}\Big\rvert_{\alpha}
&=\PAD{\alpha}{\alpha_0}\frac{1}{\tau_{\alpha}}\bigmpa{
u_n i e^{i\theta}+\phi_s e^{i\theta}
- u
}
\nonumber\\
&=\PAD{\alpha}{\alpha_0}\frac{1}{\tau_{\alpha}}
\bigmpa{
(\phi_s-u_s)e^{i\theta}
}
\mbox{,} \mbox{~~with~~} \alpha_0(\alpha,0)=\alpha,
\end{align}
which together with (\ref{tension_modulus_eqn_alpha_t}) is the main result of this subsection.

In summary, the main equations that govern capsule evolution are given by (\ref{velinteqn})-(\ref{finalforminteqn}), (\ref{S_ALPHA_EVOLVINTEQN})-(\ref{vel_0_eqn_41}), (\ref{tension_modulus_eqn_alpha_t}), and (\ref{d_alpha_0_dt_eqn_12a_1}).
A spectrally accurate numerical method for solving this system of equations is presented in the next section.

\section{Numerical Method} \label{sec:numerical_method}
We construct a continuous in time, discrete in space numerical scheme for the evolution equations by providing rules to approximate the spatial derivatives and singular integrals.

The spatial variable $\alpha$ is discretized  by $\alpha_j=j h$, where $j=-\frac{N}{2}+1,\cdots,\frac{N}{2}$ with $N$ assumed to be even, so that $\alpha$ is defined on a uniform grid of mesh size $h=\frac{2\pi}{N}$.
Define a discrete Fourier transform of a periodic function $f$ whose values are known at $\alpha_j$ by
\begin{equation}\label{APPFT}
\hat{f}_k=\frac{1}{N} \displaystyle\sum_{j=-\frac{N}{2}+1}^{\frac{N}{2}}
f(\alpha_j)e^{-i k\alpha_j}\mbox{, for } k=-\frac{N}{2}+1,\cdots,\frac{N}{2}\mbox{,}
\end{equation}
with the inverse transform given by
\begin{equation}\label{APPIFT}
f(\alpha_j)=\displaystyle\sum_{k=-\frac{N}{2}+1}^{\frac{N}{2}}
\hat{f}_k e^{ik\alpha_j}\mbox{, for } j=-\frac{N}{2}+1,\cdots,\frac{N}{2}\mbox{.}
\end{equation}

Spatial derivatives of $f$ are computed using a pseudo-spectral approximation, which is denoted by $S_h f$ and defined by
\begin{eqnarray}\label{S_h_defn}
\widehat{(S_h f)}_k&=&ik\hat{f}_k\mbox{, for } k=-\frac{N}{2}+1,\cdots,\frac{N}{2}-1\mbox{,}\\
&=& 0 \mbox{, for } k=\frac{N}{2}.
\end{eqnarray}
Due to the asymmetry of the discrete Fourier transform, we zero out the $k=\frac{N}{2}$ mode of $(\widehat{S_h f})_k$.
This will be important for stability.

% Filtering paragraph moved here
Sometimes we need to apply numerical filtering to the discrete solution.
Indeed, this will be critical for the stability of our method.
Numerical filtering is defined in Fourier space following \cite{BHL:1996} as
\begin{equation}\label{Num_Filtering_defn}
\widehat{(f^p)}_k=\rho(kh)\hat{f}_k\mbox{,}
\end{equation}
where $\rho$ is a cutoff function with the following properties:
% \leqnomode
% note that we use the right label of number in this paper!!!
\begin{align}
\rho(-x)&=\rho(x)\mbox{ ; }\rho(x)\geq 0\mbox{,}\tag{i} \\
\rho(x)&\in C^r\mbox{ ; }r>2\mbox{,}\tag{ii} \\
\rho(\pm \pi)&=\rho'(\pm \pi)=0\mbox{,}\tag{iii} \\
\rho(x)&=1 \mbox{ for }\abs{x}\leq \mu \pi\tag{iv}\mbox{, }
0<\mu <1\mbox{.}
\end{align}
Condition (iv) ensures the spectral accuracy of the filtering.
We also define a filtered derivative operator $D_h$ by
\begin{align}
(\widehat{D_h f})_k = ik\rho(kh)\hat{f}_k
\mbox{, for } k=-\frac{N}{2}+1,\cdots,\frac{N}{2}
\mbox{.}
\end{align}
%%%%%%%%%%%%%%%%%%%%%%%%%%%%%%%%

We denote by $\theta(\alpha_j)$, $\omega(\alpha_j)$, $\zeta(\alpha_j)$, etc. the exact continuous solution evaluated at grid points $\alpha_j$, and by $\theta_j$, $\omega_j$, $\zeta_j$, etc. the discrete approximation.
Also, we use $\sigma(t)$ to denote the numerical approximation of $s_{\alpha}(t)$.\\

\noindent\underline{\textit{Discrete equations for velocity $u$ and density $\omega$}}.\\

The interface contour $\gamma$ is parameterized by $\tau(\alpha,t)$.
If we set $\tau=\tau(\alpha)$ (omitting the time dependence) and
$\zeta=\tau(\alpha')$, then the integral equation (\ref{finalforminteqn}) for $\omega$ becomes
\begin{align}\label{SLFORM2}
\omega(\alpha)&+ \beta \intdef{-\pi}{\pi} F(\alpha, \alpha') \ d\alpha'
%+\beta(B-iQ)\OVL{\tau(\alpha)}+2\beta H(t)
=g(\alpha)
\mbox{,}
\end{align}
where
\begin{align}\label{SLFORM3}
%\frac{\tau_{\alpha}(\alpha')}{\tau(\alpha')-\tau(\alpha)}
%-\frac{\OVL{\tau_{\alpha}(\alpha')}}{\OVL{\tau(\alpha')}-\OVL{\tau(\alpha)}}
F(\alpha, \alpha') &= \frac{\omega(\alpha')}{2 \pi i} \left(2 i \  \mbox{Im} \left[ \frac{\tau_\alpha(\alpha')}{\tau(\alpha')-\tau(\alpha)} \right]
 \right)
\nonumber \\
&+ \frac{\OVL{\omega(\alpha')}}{2 \pi i} \left(
\frac{\tau_{\alpha}(\alpha')}{\OVL{\tau(\alpha')}-\OVL{\tau(\alpha)}}
-\frac{(\tau(\alpha')-\tau(\alpha))}{(\OVL{\tau(\alpha')}-\OVL{\tau(\alpha)})^2}
\OVL{\tau_{\alpha}(\alpha')} \right)
\mbox{,}
\end{align}
and where we have written $\omega(\alpha)$ for $\omega(\tau,t)$ and $\omega(\alpha')$ for $\omega(\zeta,t)$.
%and made use of the fact that $s_{\alpha}=s_{\alpha}(t)$ depends on time alone.
The function $g(\alpha)$ represents the right hand side of  (\ref{finalforminteqn}) and can be written in terms of $\theta$, $s_{\alpha}$ and $\tau$ as
\begin{align} \label{g_defn}
    g(\alpha)=-\frac{\chi}{2} \Bigpa{\MCAL{S}(\alpha)e^{i\theta(\alpha)} -\frac{\kappa_B\theta_{\alpha\alpha}(\alpha)}{s_{\alpha}^2}
ie^{i\theta(\alpha)}} -\beta(B-iQ)\OVL{\tau}-2\beta H(t).
\end{align}
Although the apparent singularity $\tau(\alpha')=\tau(\alpha)$ is removable, we shall nonetheless discretize (\ref{SLFORM2}) using alternate point trapezoidal rule \cite{HLK:1991},
\begin{equation}\label{alternate_trape_summation_eqn_62}
\intdef{-\pi}{\pi}
f(\alpha,\alpha')d\alpha'
\approx
\DPS\sum_{\substack{j=-\frac{N}{2}+1\\(j-i)\mbox{ odd}}}^{\frac{N}{2}}
f(\alpha_i,\alpha_j) (2h)\mbox{.}
\end{equation}
This quadrature rule is normally used for singular integrals, but for convenience  we shall also apply it here for smooth kernels since it precludes the need for separate, analytical kernel evaluations at $\alpha=\alpha'$. 
%but more importantly, it allows the use of several important quadrature estimates from \cite{BHL:1996}.

\begin{comment}
We thus discretize (\ref{SLFORM2}) as:
\begin{align}\label{Disc_omega_eqn}
\omega_i&+\frac{\beta h}{\pi i} \DPS\sum_{\substack{j=-\frac{N}{2}+1\\(j-i)\mbox{ odd}}}^{\frac{N}{2}} \biggbra{\omega_j^p\biggpa{
\frac{S_h \zeta_j}{\zeta_j-\tau_i}
-\frac{\OVL{S_h\zeta_j}}{\OVL{\zeta_j}-\OVL{\tau_i}}
}+\OVL{\omega_j^p}\biggpa{
\frac{S_h \zeta_j}{\OVL{\zeta_j}-\OVL{\tau_i}}
-\frac{\zeta_j-\tau_i}{(\OVL{\zeta_j}-\OVL{\tau_i})^2}\OVL{S_h\zeta_j}
}
}
\nonumber \\
&+\beta (B-iQ)\OVL{\tau_i}
+2\beta H_d(t)
=-\frac{\chi}{2}
\Bigpa{ \MCAL{S}_i
e^{i\theta_i}
-\frac{\kappa_B S_h^2\theta_i}{\sigma^2}
ie^{i\theta_i}
}\mbox{,}
\end{align}
where
\begin{align}
H_d(t)=h\DPS\sum_{\substack{j=-\frac{N}{2}+1\\(j-i)\mbox{ odd}}}^{\frac{N}{2}} \omega_j \sigma\mbox{.}
\end{align}

\end{comment}

% Begin insert 1

The real and imaginary parts of  (\ref{SLFORM2}) form a system of Fredholm integral equations for $\omega_1=\mbox{Re}(\omega)$ and $\omega_2=\mbox{Im}(\omega)$.
To write the corresponding discrete system, first decompose
\begin{align} \label{omega_til}
\bm{\omega}_i=\bm{\tilde{\omega}}_i + \MBF{g}_i,
\end{align}
where
\begin{align} \label{g_omega_omegatil_def}
\bm{\omega}_i
=\begin{bmatrix}
\omega_1\\
\omega_2
\end{bmatrix}_i
\mbox{,     }
\bm{\tilde{\omega}}_i
=\begin{bmatrix}
\tilde{\omega}_1\\
\tilde{\omega}_2
\end{bmatrix}_i
\mbox{,     }
\bm{g}_i
=\begin{bmatrix}
g_1 \\
g_2
\end{bmatrix}_i
\mbox{,}
\end{align}
with  $g_1=\mbox{Re}(g)$ and $g_2=\mbox{Im}(g)$  (cf. (\ref{g_defn})).
Then form the   discrete system as
\begin{align}\label{I_plus_beta_K_matrix_eqn_2_1}
(\MBF{I}+\beta\MBF{K})\bm{\tilde{\omega}}_i
=-\beta \MBF{K} \MBF{g}_i^p \mbox{.}
\end{align}
Here
$\MBF{K}$ is the discrete operator
\begin{align}\label{K_omega_matrix_eqn_3_0}
\MBF{K}\bm{\omega}_i
&=\DPS\sum_{\substack{j=-\frac{N}{2}+1\\(j-i)\mbox{ odd}}}^{\frac{N}{2}}
\begin{pmatrix}
\bigpa{K_R^{(1)}}_{i,j}+ \bigpa{ K_R^{(2)}}_{i,j} & \bigpa{K_I^{(2)}}_{i,j} \\
\bigpa{K_I^{(2)}}_{i,j} & \bigpa{K_R^{(1)}}_{i,j}-\bigpa{K_R^{(2)}}_{i,j}
\end{pmatrix}
\begin{bmatrix}
\omega_1 \\
\omega_2
\end{bmatrix}_j
(2h)
\\
&=\DPS\sum_{\substack{j=-\frac{N}{2}+1\\(j-i)\mbox{ odd}}}^{\frac{N}{2}}
\bigpa{\MBF{K}_M}_{i,j}\bm{\omega}_j (2h) \label{system}
\mbox{,}
\end{align}
where
\begin{align}\label{K_omega_disc_kernel_eqn_3_1}
\bigpa{K_R^{(1)}}_{i,j}
&=\frac{1}{\pi}\mbox{Im}\Bigpa{
\frac{S_h\tau_j}{\tau_j-\tau_i}}
+\sigma
\mbox{,}
\\
\bigpa{K_R^{(2)}}_{i,j}
&=\mbox{Re}\Bigbra{
\frac{1}{2\pi i}
\Bigpa{
\frac{S_h\tau_j}{\OVL{\tau_j}-\OVL{\tau_i}}
-\frac{\tau_j-\tau_i}{(\OVL{\tau_j}-\OVL{\tau_i})^2}
\OVL{S_h\tau_j}
}
}\mbox{,}
\\
\bigpa{K_I^{(2)}}_{i,j}
&=\mbox{Im}\Bigbra{
\frac{1}{2\pi i}
\Bigpa{
\frac{S_h\tau_j}{\OVL{\tau_j}-\OVL{\tau_i}}
-\frac{\tau_j-\tau_i}{(\OVL{\tau_j}-\OVL{\tau_i})^2}
\OVL{S_h\tau_j}
}
}\mbox{,}
\end{align}
and $\bigpa{\MBF{K}_M}_{i,j}$ is the matrix kernel in (\ref{K_omega_matrix_eqn_3_0}). The discrete function
$g_i=g_{1i} + i g_{2i}$ is the discretization of (\ref{g_defn}):
\begin{align} \label{g_i_def}
  g_i=-\frac{\chi}{2} \Bigpa{\MCAL{S}_i e^{i\theta_i} -\frac{\kappa_B S_h^2\theta_i}{s_{\alpha}^2}
ie^{i\theta_i}}
-\beta(B-iQ)\OVL{\tau_i}-2\beta H(t).
\end{align}
We sometimes use a filtered $g^p_i$ in which the second derivative operator $S_h^2$ in 
(\ref{g_i_def}) is replaced by its filtered version $D_h^2$; see, for example, (\ref{I_plus_beta_K_matrix_eqn_2_1}). In  a slight abuse of notation, this (partially) filtered discrete function is denoted with a superscipt $p$.

The invertibility of (\ref{I_plus_beta_K_matrix_eqn_2_1}) is a consequence of Lemma \ref{I_plus_beta_K_matrix_lemma_6_1}.
% TODO: figure out the reference lemma and section numbers
There, it is shown that for sufficiently small $\beta$ and spatial step size $h$, (\ref{I_plus_beta_K_matrix_eqn_2_1}) is uniquely solvable for $\bm{\tilde{\omega}}_i$ by the method of successive approximations.
Details are deferred to Section 12. Note that $\bm{\tilde{\omega}}=0$ if $\beta=0$.
% TODO: figure out the reference numbers and section numbers
% End insert 1

% Begin move 1 to page 15

% End move 1 to page 15

We next consider the velocity equation (\ref{velinteqn}).
To obtain a stable scheme, a careful treatment of the principal value integral is required.
We parameterize the contour by $\tau(\alpha)$ then isolate the most singular part by adding and subtracting the periodic Hilbert transform 
\begin{align}
\MCAL{H} \omega(\alpha)
=\frac{1}{2\pi}\mbox{P.V.}\intdef{-\pi}{\pi}
\omega(\alpha')
\cot\Bigpa{\frac{\alpha-\alpha'}{2}}
d\alpha'\mbox{,}
\end{align}
to obtain
\begin{align}\label{velintsubstractHilb_1}
u(\alpha)
=\MCAL{H} \omega(\alpha)
-\frac{1}{2\pi}\intdef{-\pi}{\pi} G(\alpha, \alpha') \ d\alpha'
+(Q+iB)\OVL{\tau(\alpha)}-\frac{iG}{2}\tau(\alpha)
\mbox{,}
\end{align}
where
\begin{align}\label{Gdef}
G(\alpha,\alpha')=&
\omega(\alpha')
\biggmpa{
2 \mbox{Re} \left( \frac{  \tau_{\alpha'}(\alpha')}{\tau(\alpha')-\tau(\alpha)} \right) 
+\cot\Bigpa{\frac{\alpha-\alpha'}{2}}}
\nonumber \\
&-\OVL{\omega(\alpha')}
\biggmpa{
\frac{\tau_{\alpha'}(\alpha')}{\OVL{\tau(\alpha')}-\OVL{\tau(\alpha)}}
-\frac{\tau(\alpha')-\tau(\alpha)}{(\OVL{\tau(\alpha')}-\OVL{\tau(\alpha)})^2}
\OVL{\tau_{\alpha'}(\alpha')}
}
\end{align}
It is easy to see that $G(\alpha, \alpha')$ is a smooth function of $\alpha$ and $\alpha'$.
%This is due to the substraction of the leading order singular part $\MCAL{H}\omega$.
%The singular integral $\MCAL{H}\omega$ and integral with smooth integrand in (\ref{velintsubstractHilb_1}) will be treated differently with regard to filtering, which will be essential in our design of a stable scheme.

The velocity equation (\ref{velintsubstractHilb_1}) is discretized using the alternate point trapezoidal rule as
\begin{align} \label{vel_decomp}
u_i&={\MCAL H}_h \omega_i + (u_R)_i,    
\end{align}
where  ${\MCAL H}_h$ is the discrete Hilbert transform defined by 
\begin{align} \label{discrete_Hilbert}
{\MCAL H}_h f_i=\frac{h}{\pi} \DPS\sum_{\substack{j=-\frac{N}{2}+1\\(j-i)\mbox{ odd}}}^{\frac{N}{2}}
f_j\cot \left( \frac{\alpha_i-\alpha_j}{2} \right),
\end{align}
and  $(u_R)_i$  is given by
\begin{align}\label{velocity_eqn_disc}
(u_R)_i  =\frac{h}{\pi} \DPS\sum_{\substack{j=-\frac{N}{2}+1\\(j-i)\mbox{ odd}}}^{\frac{N}{2}} \left\{ -\omega_j^p G_{ij}^{(1)}
+\OVL{\omega}_j^p G_{ij}^{(2)} \right\}
+(Q+iB)\OVL{\tau}_i-\frac{iG\tau_i}{2}\mbox{,}
\end{align}
% insert 8
in which
\begin{align}
G_{ij}^{(1)}=
2 \mbox{Re} \left( \frac{S_h\tau_j}{\tau_j-\tau_i} \right)
+\cot\Bigpa{\frac{\alpha_i-\alpha_j}{2}}\mbox{,}
\end{align}
and
\begin{align} \label{vel_kernel2}
G_{ij}^{(2)}=
\frac{S_h\tau_j}{\OVL{\tau_j}-\OVL{\tau_i}}
-\frac{\tau_j-\tau_i}{(\OVL{\tau_j}-\OVL{\tau_i})^2}
\OVL{S_h\tau_j}\mbox{.}
\end{align}
% Begin of move 1 from page 14
In (\ref{velocity_eqn_disc}), we use the filtered density defined by  $\omega^p_j=\tilde{\omega}_j+g^p_j$ 
in the  discretization $(u_R)_i$ of the regular integral
(see the comment following (\ref{g_i_def})),
but not in the discretization $\MCAL{H}_h\omega_i$ of the leading order singular integral.
This targeted application of filtering is found to be necessary 
%to make use of parabolic smoothing 
to prove stability of our method.
% End of move 1 from page 14
In the discrete equations, $S_h\tau$ can be replaced by $\sigma e^{i\theta}$
%and similar for $S_h\zeta_j$ 
(i.e., the actual application of the discrete derivative operator $S_h$ is not required here), but for convenience we will continue to use $S_h\tau$ to represent the discrete version of $\tau_{\alpha}$.

In our method, we also need the discrete normal and tangential velocities,
\begin{align} \label{norm_and_tang_vel}
(u_n)_i = \mbox{Im}\bigbra{u_i e^{-i\theta_i}}\mbox{,}~~
(u_s)_i = \mbox{Re}\bigbra{u_i e^{-i\theta_i}}\mbox{,}
\end{align}
which follow from (\ref{UComponentEqn}) with $n_i=ie^{i\theta_i}$.
Care must be made in the discretization of $u_i e^{-i\theta_i}$, for reasons which will become apparent below.
We first define the commutator
\begin{align}
\bigmpa{\MCAL{H}_h,\phi_i}(\psi_i)
=\MCAL{H}_h \bigpa{
\phi_i \psi_i
}
-\phi_i
\MCAL{H}_h\bigpa{\psi_i}
\mbox{.}
\end{align}
Then, using (\ref{vel_decomp}), we discretize
\begin{align}\label{num_approx_U_e_to_the_neg_i_theta}
u_i e^{-i\theta_i}&=\MCAL{H}_h(\omega_i e^{-i\theta_i})
-\bigmpa{\MCAL{H}_h,e^{-i\theta_i}}(\omega_i^p)
+(u_R)_i e^{-i\theta_i}\mbox{,}
\end{align}
applying the filter only in the argument of the commutator.
The discrete normal and tangential velocities can be written in a particularly simple form in the important special case of viscosity matched fluids $(\beta=0)$, which is now described.\\

\noindent\underline{\textit{Viscosity matched fluids $(\beta=0)$}}.\\

%It is illustrative to consider the important special case of viscosity matched fluids, for which $\beta=0$ and $\chi=\frac{1}{2}$.
When $\beta=0$, the  nonlocal operator $\MBF{K}$ in  (\ref{I_plus_beta_K_matrix_eqn_2_1})  drops out, leading to a considerable simplification.
Taking $\beta=0$ and $\chi=\frac{1}{2}$ in (\ref{SLFORM2}) and  (\ref{g_i_def}),  we see that
\begin{align}\label{num_approx_omega_exp_i_theta}
\omega_i e^{-i\theta_i}
=-\frac{1}{4}\Bigpa{
\MCAL{S}_i-i\frac{\kappa_B}{\sigma^2}
S_h^2\theta_i
}\mbox{.}
\end{align}
Inserting this into the discrete Hilbert transform in  (\ref{num_approx_U_e_to_the_neg_i_theta}) and taking the imaginary part per (\ref{norm_and_tang_vel}) gives
\begin{align}\label{num_approx_u_n_i_decomp_1}
(u_n)_i&=\frac{\kappa_B}{4\sigma^2}\MCAL{H}_h(S_h^2\theta_i)
+\mbox{Im}\bigbra{
-\bigmpa{\MCAL{H}_h,e^{-i\theta_i}}(\omega_i^p)
+(u_R)_i e^{-i\theta_i}
}\mbox{.}
\end{align}
The significance of the decomposition (\ref{num_approx_U_e_to_the_neg_i_theta}) is now apparent:
by moving $e^{-i\theta_i}$ into the argument of discrete Hilbert transform, the leading order term of the normal velocity, namely $\frac{\kappa_B}{4\sigma^2}\MCAL{H}(S_h^2\theta_i)$, becomes linear in $\theta_i$ with a spatially constant coefficient that has the right sign to take advantage of parabolic smoothing.
This will be critical in energy estimates.
We similarly decompose the tangential velocity as
%Recall that $(u_s)_i=\mbox{Re}\bigbra{u_i e^{-i\theta_i}}$, then insert (\ref{num_approx_omega_exp_i_theta}) into (\ref{num_approx_U_e_to_the_neg_i_theta}), and take the real part to obtain
\begin{align}\label{num_approx_u_s_i_decomp_1}
(u_s)_i&=-\frac{1}{4}\MCAL{H}_h\bigpa{
\MCAL{S}_i}
+\mbox{Re}\bigbra{-[\MCAL{H}_h,e^{-i\theta_i}]\omega_i^p
+(u_R)_i e^{-i\theta_i}
}
\mbox{.}
\end{align}

It will later be  shown that the nonlocal operator $\MBF{K}$ in (\ref{I_plus_beta_K_matrix_eqn_2_1})   does not  affect the stability of the discretization.
Henceforth, we focus the analysis on the  special case of viscosity matched fluids, and later generalize to the full problem for nonzero $\beta$.\\

\noindent\underline{\textit{Discretization of evolution equations}}.\\

%We next consider the discretization of the $\theta$ and $\sigma$ equations.
%The $\theta,\sigma$ equations are discretized similarly.
The semi-discrete (continuous in time, discrete in space) equations for $\theta,\sigma$ are
\begin{align}
\bigpa{\theta_t}_i&=\frac{1}{\sigma}\bigpa{
S_h (u_n)_i+(\phi_s)_i S_h \theta_i
}\mbox{,}\label{theta_i_time_der_eqn} \\
\sigma_t&=-\ip{u_n S_h\theta}_h\mbox{,}\label{sigma_i_time_der_eqn}
\end{align}
where
\begin{equation}
\ip{f}_h=\frac{1}{N}\DPS\sum_{j=-\frac{N}{2}+1}^{\frac{N}{2}} f_j\mbox{,}
\end{equation}
is the discrete mean computed using trapezoid rule.
In order to recover the interface location from $\theta_i$ and $\sigma$, we need to introduce the pseudo-spectral antiderivative operator defined in Fourier space on functions $f$ of mean zero by
\begin{align}\label{Int_h_defn}
\widehat{\bigpa{S_h^{-1} f}}_k
=
\begin{cases}
\frac{1}{ik}\hat{f}_k\quad &\mbox{for } k\neq 0\mbox{,} \\
0\quad &\mbox{for } k=0\mbox{.}
\end{cases}
\end{align}
Then the discretization of (\ref{X1_POS_INT_EQN}), (\ref{X2_POS_INT_EQN}) can be written
\begin{equation}\label{Z_Int_disc}
\tau_i=\tau_c+S_h^{-1}\bigpa{\sigma e^{i\theta}-\ip{\sigma e^{i\theta}}_h}_i\mbox{,}
\end{equation}
where $\tau_c$ is the zero (constant) Fourier mode of $\tau_i$. This is evolved from (\ref{vel_0_eqn_41}) as
\begin{align}\label{tau_c_time_der_hat_u_0}
\frac{d\tau_c}{dt}=\hat{v}_0=\ip{v}_h\mbox{,}
\end{align}
where $\hat{v}_0$ is the zero Fourier mode of the discrete velocity $v_i$.
Equation (\ref{PHI_S_ALPHA_T}) is discretized as
\begin{equation}\label{phi_s_disc}
\bigpa{\phi_s}_i=S_h^{-1}\bigpa{u_n S_h\theta-\ip{u_n S_h\theta}_h}_i\mbox{,}
\end{equation}
and the surface tension (\ref{tension_modulus_eqn_alpha_t}) as
\begin{align}\label{tension_modulus_disc_16_1}
\MCAL{S}_i=\frac{\sigma}{\sigma_0 D_h\alpha_{0 i}}
\bigpa{
1+\MCAL{S}_{0 i}
}
-1\mbox{,}
\end{align}
where $\MCAL{S}_{0 i}$ is the discrete initial tension, and $\sigma_0$ is the initial value of $s_{\alpha}$.
The semi-discrete equation for $\alpha_{0 i}$ is obtained from (\ref{d_alpha_0_dt_eqn_12a_1}) as
\begin{align}\label{alpha_0_t_disc_16_2}
(\alpha_{0 t})_i=\frac{D_h \alpha_{0 i}}{\sigma e^{i\theta_i}}
\bigpa{
(\phi_s-u_s)e^{i\theta}
}_i\mbox{.}
\end{align}
In summary, the principal equations for the discrete scheme are (\ref{I_plus_beta_K_matrix_eqn_2_1}), (\ref{vel_decomp}), 
(\ref{theta_i_time_der_eqn})-(\ref{sigma_i_time_der_eqn}),
(\ref{tau_c_time_der_hat_u_0}), and (\ref{alpha_0_t_disc_16_2}), and are the main result of this section.

\noindent\underline{\textit{Discretization for a drop interface}}.\\

The discretization for a (nonelastic) drop interface with zero bending stress and constant surface tension,  $\kappa_B=0$ and $\MCAL{S}_i=1$, is modified from the above.  In this case the stability is more delicate, since we can no longer take advantage of the stabilizing properties of the (high derivative) bending stress term. Thus, the numerical method requires more filtering. We reinterpret $g_i^p$, originally  defined in the comment following 
(\ref{g_i_def}), to be the fully filtered $g_i$. We also now filter the leading order term in the decomposition (\ref{vel_decomp}), so that
\begin{align} \label{vel_decomp_drop}
u_i&={\MCAL H}_h \omega_i^p + (u_R)_i,    
\end{align}
where $\omega_i^p=\tilde{\omega}_i+g_i^p$.
Furthermore, we replace each occurrence of $S_h \tau_i$ in the kernels of (\ref{K_omega_matrix_eqn_3_0}) and (\ref{velocity_eqn_disc}) with its filtered version $D_h \tau_i$, or equivalently by $(\sigma e^{i \theta_i})^p$.  Finally, we replace $(S_h u_n)_i$ in (\ref{theta_i_time_der_eqn}) with $(D_h u_n)_i$.  Other aspects of the discretization remain the same as for an elastic capsule.  

\noindent\underline{\textit{Numerical example}}.\\

An example numerical calculation is shown in Figure 2. We use the BI method of \cite{HSB:2012} for the elastic capsule computation.
Their method is similar, but not identical to, that described in this section.
In particular, the  algorithm analyzed here generalizes that of \cite{HSB:2012} to include nonzero interior viscosity and membrane bending stress.  The method for the drop computation is as described in this section, and essentially the same as in \cite{MCAK1:2001}, \cite{XBS:2013}. 
More extensive numerical results using the method for capsules will be presented in later work.
\begin{figure}[ht!] 
\begin{center}
\centerline{
   \qquad \qquad \hspace{-.16in}  \subfloat[strain flow]{\includegraphics[scale=1.0]{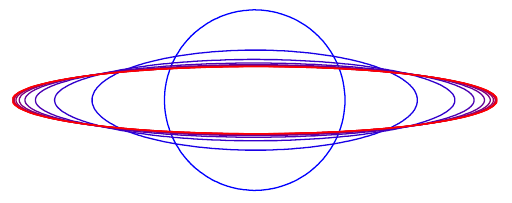}}%
    \qquad \qquad \qquad \hspace{.02in}
    \subfloat[][shear flow]{\includegraphics[scale=1.0]{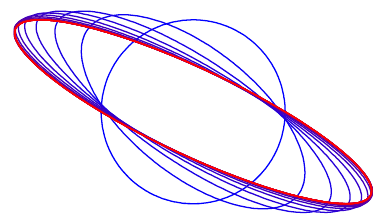}}}%
\includegraphics[scale=0.7]{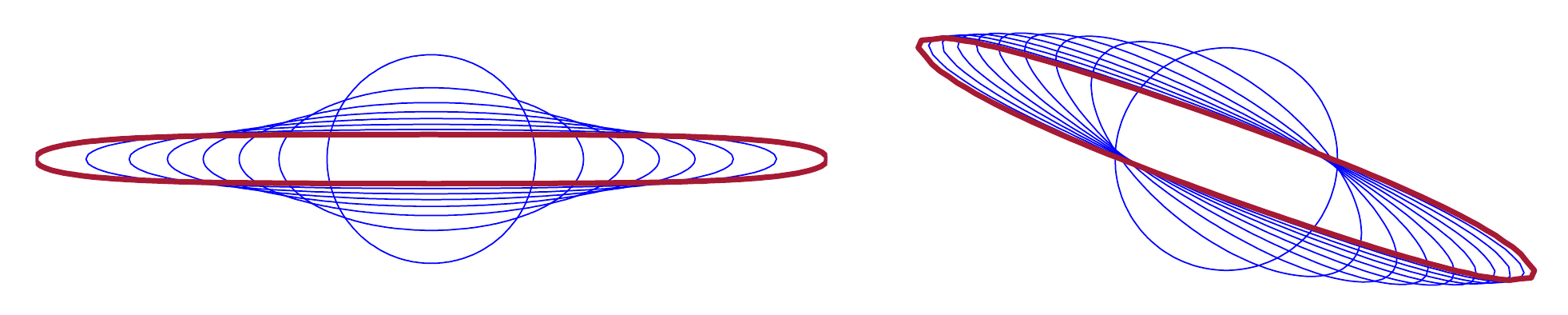} 
\caption{\label{fig-2}Time evolution of a Hookean elastic capsule (top)  and fluid drop (bottom) in (A) a pure strain flow with $Q=1$ and (B) a simple shear flow with $2B=G=-1$.
%from \cite{HSB:2012}.
Other parameters are  $\MCAL{S}_0 = 1$,  $\kappa_B=\lambda=0$ for the elastic capsule, and $\MCAL{S}=1, ~\lambda=0.01$ for the fluid drop.  The profiles are shown at intervals of $\Lap t = 1.0$ (top left) and $\Lap t = 0.5$ (top right and bottom).}
\end{center}
\end{figure}

\section{Consistency: Preliminary Lemmas}
We define the Sobolev norm
\begin{equation}
\norm{f}_s=\biggpa{\DPS\sum_{k=-\infty}^{\infty} \bigpa{1+\abs{k}^2}^s \abs{\hat{f}_k}^2}^{\frac{1}{2}}\mbox{.}
\end{equation}

The first lemma, a version of which is presented in \cite{ET:1987}, gives the accuracy of the pseudo-spectral derivative. For completeness, a proof is given in the appendix.
\begin{lemma}\label{S_h_estimate}
Let $f(\alpha)$ be a periodic $C^{s+1}[-\pi,\pi]$ function. Then
\begin{equation}
\abs{S_h f(\alpha_i)-f_{\alpha}(\alpha_i)}\leq ch^{s-\frac{1}{2}}\norm{f}_{s+1}\mbox{.}
\end{equation}
The same inequality holds for $D_h$ in place of $S_h$.
\end{lemma}

Similarly, for the pseudo-spectral anti-derivative operator we have
\begin{lemma}\label{pseudo_int_est_lemma}
Let $f$ be a periodic, zero$-$mean, $C^s[-\pi,\pi]$ function.
Then
\begin{equation}
\bigabs{\D_{\alpha}^{-1} f(\alpha_j)
-S_h^{-1} f(\alpha_j)
}
\leq c h^{s-\frac{1}{2}}\norm{f}_s\mbox{.}
\end{equation}
\end{lemma}
\begin{proof}
The proof is a simple adaptation of that for Lemma \ref{S_h_estimate}.
\end{proof}

The next lemma is a well-known result on the accuracy of trapezoid rule for periodic functions.
\begin{lemma}\label{trapezoidal_rule_estimate}
Let $f(\alpha)$ be as in Lemma \ref{S_h_estimate}.
Then
\begin{equation}
\biggabs{\DPS\sum_{j=-\frac{N}{2}+1}^{\frac{N}{2}} f(\alpha_j)h-\intdef{-\pi}{\pi}f(\alpha)d\alpha}
\leq ch^{s+1}\norm{f}_{s+1}\mbox{.}
\end{equation}
\end{lemma}
\begin{proof}
This is derived from the Euler$-$Maclaurin formula.
For more details, see \cite{GHKHH:1991}.
\end{proof}

The next lemma provides a result on the accuracy of the filtering  operator.
\begin{lemma}\label{filter_lemma}
Let $f\in C^s[-\pi,\pi]$ be periodic, and let $f^p$ be as defined in (\ref{Num_Filtering_defn}) with conditions (i)$-$(iv).
Then
\begin{equation}
\abs{f^p(\alpha_i)-f(\alpha_i)}\leq ch^{s-\frac{1}{2}}\norm{f}_s\mbox{.}
\end{equation}
\begin{proof}
The proof is similar to that for Lemma \ref{S_h_estimate}, and is omitted here.
\end{proof}
\end{lemma}
% Begin insert 9aa
%\begin{remark}
%Due to the asymmetry of the discrete Fourier transform, we will zero out the $k=\frac{N}{2}$ mode of %$(\widehat{S_h f})_k$.
%This will be important for utilizing the smoothing properties of the highest derivative term.
%It is easy to see that zeroing out this mode does not affect any of the estimates in this section.
%\end{remark}
% End insert 9aa
% TODO: hesitating removing the lemma here!

\section{Consistency}
We calculate the error when the exact solution is substituted into the discrete system of equations.
Assume the exact solution is regular enough so that $\theta(\cdot,t)\in C^{m+1}[-\pi,\pi]$, $\omega(\cdot,t)\in C^{m-1}[-\pi,\pi]$ and $\alpha_0(\cdot,t)\in C^{m+1}[-\pi,\pi]$.
We also assume the initial tension $\MCAL{S}(\cdot,0)$ is in $C^m[-\pi,\pi]$.
The different levels of regularity for the different functions follows from an analysis of the continuous evolution equations.
We denote by $u_h(\alpha_i)$, $(u_n)_h(\alpha_i)$, $\omega_h(\alpha_i)$, $(\phi_s)_h(\alpha_i)$, etc.
quantities that are evaluated by substituting the exact solution $\theta(\cdot,t)$, $s_{\alpha}(t)$, $\alpha_0(\cdot,t)$ into the discrete equations.
We make repeated use of the estimate
\begin{align}\label{tau_h_alpha_i_formula_1}
\tau_h(\alpha_i)&=\tau_c+S_h^{-1} \bigpa{s_{\alpha} e^{i\theta}-\ip{s_{\alpha} e^{i\theta}}_h}(\alpha_i)=\tau(\alpha_i)+\mbox{O}(h^{m+\frac{1}{2}})\mbox{,}
\end{align}
which follows from (\ref{Z_Int_disc}), Lemma \ref{pseudo_int_est_lemma}, and the assumption on the regularity of the exact solution.
\newline
\newline
\underline{\textit{Consistency of $\omega$ equation}}.
\newline
\newline
We first assess the smoothness of the integrand $F(\alpha, \alpha')$ in the continuous equation for $\omega$, (\ref{SLFORM2}).
%We note that by Lemma \ref{filter_lemma}, the filtered density $\omega^p$ can be replaced by an unfiltered $\omega$, incurring an $\mbox{O}(h^{m-\frac{3}{2}})$ error based on the regularity of $F(\alpha,\alpha')$ (see (\ref{F_alpha_cdot_regularity}) below).
The apparent singularity in $F(\alpha,\alpha')$ is removable, and
\begin{equation}
\DPS\lim_{\alpha'\RA\alpha} F(\alpha,\alpha')=i\Bigpa{\omega(\alpha)\kappa(\alpha)s_{\alpha}
+\frac{\OVL{\omega}(\alpha)\kappa(\alpha)\tau_{\alpha}^2(\alpha)}{s_{\alpha}}
}\mbox{.}
\end{equation}
Recalling that $\omega(\cdot)\in C^{m-1}$, $\tau_{\alpha}(\cdot)=s_{\alpha}e^{i\theta(\cdot)}\in C^{m+1}$, and that $s_{\alpha}$ is bounded away from zero, it follows that
\begin{align}\label{F_alpha_cdot_regularity}
F(\alpha,\cdot)\in C^{m-1}\mbox{.}
\end{align}

The truncation error of the discrete operator $\MBF{K} \bm{\omega}=\MBF{K} (\tilde{\bm{\omega}} + \bm{g})$ in (\ref{K_omega_matrix_eqn_3_0}) is equivalent to the truncation error of the
alternate point trapezoidal rule approximation
of $\int_{-\pi}^\pi F(\alpha, \alpha') \ d \alpha'$  in (\ref{SLFORM2}), which is its complex counterpart. We therefore consider the alternate point sum $ \sum F_h(\alpha_i, \alpha_j)2h$, where the subscript $h$ denotes evaluation of $F$ using $\tau=\tau_h$, $\partial_\alpha=S_h$, and the exact $\theta$, $s_\alpha$, and   $\omega$.
By (\ref{tau_h_alpha_i_formula_1}), we can replace $\tau_h$ in this sum by $\tau$ incurring an order $\mbox{O}(h^{m+\frac{1}{2}})\bigmpa{\DPS\min_j \bigpa{\tau(\alpha_j)-\tau(\alpha_i)} }^{-1}=\mbox{O}(h^{m-\frac{1}{2}})$ error.
There is no error in $S_h\tau_h$ since the exact solution $s_{\alpha}e^{i\theta(\alpha_i)}$ is substituted for this term.
The standard trapezoid rule discretization therefore satisfies
\begin{align}
\DPS\sum_{j=-\frac{N}{2}+1}^{\frac{N}{2}} F_h(\alpha_i,\alpha_j)h
=\DPS\sum_{j=-\frac{N}{2}+1}^{\frac{N}{2}} F(\alpha_i,\alpha_j)h
+\mbox{O}(h^{m-\frac{1}{2}})\mbox{,}
\end{align}
where we have used the above remarks to replace $F_h$ with $F$.  
Then by the error estimate for trapezoidal rule integration (Lemma \ref{trapezoidal_rule_estimate}), the truncation error is bounded as
\begin{align}\label{J_h_F_trapezoid_err_est}
\left| \DPS\sum_{\substack{j=-\frac{N}{2}+1}}^{\frac{N}{2}}
F_h(\alpha_i,\alpha_j)h
-\intdef{-\pi}{\pi} F(\alpha_i,\alpha') d\alpha' \right|
\leq c h^{m-1}\norm{F(\alpha_i,\cdot)}_{m-1}.
%=\mbox{O}(h^{m-1})\mbox{.}
\end{align}
A standard argument \cite{BHL:1996}, \cite{HLK:1991} shows that the truncation error for the alternate point trapezoidal rule quadrature of $F$ 
%(or equivalently, the truncation error of the operator $\MBF{K} \bm{\omega})$ 
is the same as for 
trapezoidal rule, i.e., 
\begin{align}\label{J_h_err_i_even_odd}
\DPS\sum_{\substack{j=-\frac{N}{2}+1\\(j-i){\rm\ odd}}}^{\frac{N}{2}}
F_h(\alpha_i,\alpha_j)2h
-\intdef{-\pi}{\pi} F(\alpha_i,\alpha') d\alpha'
=\mbox{O}(h^{m-1})\mbox{.}
\end{align}
Error estimates for $\bm{g}_i$ in   (\ref{omega_til}), (\ref{I_plus_beta_K_matrix_eqn_2_1})  are obtained using Lemma \ref{S_h_estimate}.
For example, 
\begin{equation}
\frac{S_h^2\theta(\alpha_i)}{s_{\alpha}^2}
-\frac{\theta_{\alpha\alpha}(\alpha_i)}{s_{\alpha}^2}
= \mbox{O}(h^{m-\frac{3}{2}})\mbox{,}
\end{equation}
%where we have used that $\theta(\cdot, t)$ is in $C^{m+1}[-\pi,\pi]$.
%The remaining terms are found to have higher-order error in $h$.
(cf. (\ref{g_i_def})) which is the dominant source of truncation error.
It follows that
\begin{align} \label{om_eqn_trunc_error}
\left( \bm{I}+ \beta \MBF{K}_h \right) \bm{\omega}(\alpha_i) - \bm{g}_h(\alpha_i)
=\mbox{O}(h^{m-\frac{3}{2}})\mbox{,}
\end{align}
and the consistency of (\ref{I_plus_beta_K_matrix_eqn_2_1}) results from substituting for $\bm \omega$ using (\ref{omega_til}) and noting that $ \MBF{K}_h  (\bm{g}_h^p(\alpha_i) -\bm{g}(\alpha_i))=O(h^{m-\frac{3}{2}})$.

%

\begin{comment}
%
\begin{remark}
When $\beta=0$ and $\chi=\frac{1}{2}$, equation (\ref{Disc_omega_eqn}) provides an expect relation for $\omega_i$ in terms of $\theta_i$, $\sigma$ and $\MCAL{S}_i$.
Let $\omega_h(\alpha_i)$ denote the quantity obtained by substituting the exact solution $\theta(\cdot,t)$, $s_{\alpha}(t)$ and $\MCAL{S}(\cdot,t)$ into (\ref{Disc_omega_eqn}).
Then the above remarks show that:
\begin{align}
\omega_h(\alpha_i)=\omega(\alpha_i)+\mbox{O}(h^{\frac{5}{2}})
\end{align}
\end{remark}
%
\end{comment}
%
We also need an estimate on $\bm{\omega}_h$, which is the solution of  
\begin{equation} \label{omh_eqn}
(\bm{I}+ \beta \MBF{K}_h) (\bm{\omega}_h)_i=\bm{g}_h(\alpha_i).
\end{equation}
Assume that $(\bm{\omega}_h)_i=\bm{\omega}(\alpha_i) + h^q \bm{r}_i$
for $\bm{r}_i$ not identically zero, and substitute into
(\ref{omh_eqn}) to obtain   $ (\bm{I}+ \beta \MBF{K}_h) \bm{\omega}(\alpha_i)-\bm{g}_h(\alpha_i) =
-h^q  (\bm{I}+ \beta \MBF{K}_h) \bm{r}_i$. In Lemma  \ref{I_plus_beta_K_matrix_lemma_6_1} below, it is shown that the right-hand-side of this equation is nonzero for $\beta$  sufficiently small.
It follows from (\ref{om_eqn_trunc_error}) that $q=m-3/2$, i.e.,
\begin{equation} \label{omh_est}
(\bm{\omega}_h)_i=\bm{\omega}(\alpha_i)+O(h^{m-\frac{3}{2}}).
\end{equation}

\noindent
\underline{\textit{Consistency of velocity}}.
\newline
\newline
We first consider the smoothness of the integrand $G(\alpha,\alpha')$ in (\ref{velintsubstractHilb_1}).
%We once again use Lemma \ref{filter_lemma} to replace the filtered $\omega^p$ with unfiltered $\omega$, incurring an $\mbox{O}(h^{m-\frac{3}{2}})$ error per the regularity of $G$ given below.
Note that
\begin{equation}
\DPS\lim_{\alpha'\RA\alpha}G(\alpha,\alpha')
=-\omega(\alpha)\mbox{Re}
\Bigpa{
\frac{\tau_{\alpha\alpha}}{\tau_{\alpha}}
}
+i\OVL{\omega(\alpha)}\kappa(\alpha)
\frac{\tau_{\alpha}^2(\alpha)}{s_{\alpha}}\mbox{,}
\end{equation}
and it follows that $G(\alpha,\cdot)\in C^{m-1}$.
Now, let $G_h(\alpha_i, \alpha_j)$ denote the discrete integrand in (\ref{velocity_eqn_disc}) but with $\tau_j$ replaced by $\tau_h(\alpha_j)$, $\omega^p_j$ by   $\omega_h^p(\alpha_j)$, etc. 
% insert 13
Using the same argument as that which led to (\ref{J_h_err_i_even_odd}), we deduce 
%\begin{equation}
%K_h(\alpha_i)=\DPS\sum_{\substack{j=-\frac{N}{2}\\(j-i){\rm\ odd}}}^{\frac{N}{2}}
%G_h(\alpha_i,\alpha_j)h
%\end{equation}
%Then by the same argument as for $J_h(\alpha_i)$ we have:
\begin{equation}
\DPS\sum_{\substack{j=-\frac{N}{2}+1\\(j-i){\rm\ odd}}}^{\frac{N}{2}}
G_h(\alpha_i,\alpha_j)2h
-\intdef{-\pi}{\pi} G(\alpha_i,\alpha') \ d\alpha'
=O(h^{m-\frac{3}{2}})
\mbox{,}
\end{equation}
in which the dominant $O(h^{m-\frac{3}{2}})$ contribution to the error comes from replacing $\omega_h^p(\alpha_i)$ by $\omega(\alpha_i)$, invoking  Lemma \ref{filter_lemma} and (\ref{omh_est}). This gives the truncation error of $(u_R)_i$ in (\ref{vel_decomp}). We next consider the discrete Hilbert transform in (\ref{vel_decomp}).
% end insert 13
%Let:
%\begin{align}\label{Hilbert_h_omega_alpha_i}
%\MCAL{H}_h \omega(\alpha_i)
%=\frac{1}{2\pi}\DPS\sum_{\substack{j=-\frac{N}{2}\\(j-i){\rm\ odd}}}^{\frac{N}{2}}
%\omega(\alpha_j)
%\cot\Bigpa{\frac{\alpha_i-\alpha_j}{2}}
%2h
%\end{align}
%denote the discrete Hilbert transform.
It is shown in \cite{BHL:1996} that
\begin{align}
\MCAL{H}_h \omega(\alpha_i)
-\frac{1}{2\pi}
\mbox{P.V.}\intdef{-\pi}{\pi}
\omega(\alpha')\cot\Bigpa{\frac{\alpha-\alpha'}{2}}
d\alpha'
=\mbox{O}(h^{m-2})\mbox{.}
\end{align}
This is a special case of a result proven in Section 2 of \cite{BHL:1996}, where it is shown that the order of accuracy of the discrete Hilbert transform is related to the regularity of $\omega_{\alpha}(\cdot)$, which here is $C^{m-2}$.
It follows that
\begin{align}% be careful of sign between H_h and K_h???
u_h(\alpha_i)
&=\MCAL{H}_h \omega(\alpha_i)
-\frac{1}{2\pi}
\DPS\sum_{\substack{j=-\frac{N}{2}+1\\(j-i){\rm\ odd}}}^{\frac{N}{2}}
G_h(\alpha_i,\alpha_j)2h
+(Q+iB)\tau_h(\alpha_i)\nonumber \\
&-\frac{iG\tau_h(\alpha_i)}{2}=u(\alpha_i)+\mbox{O}(h^{m-2})
\mbox{,}
\end{align}
which shows the consistency of the velocity discretization. 
From this it is easy to see that
\begin{align}
(u_s)_h(\alpha_i)&=u_s (\alpha_i)+\mbox{O}(h^{m-2}),
\label{tangent_h_err_order} \\
(u_n)_h(\alpha_i)&=u_n (\alpha_i)+\mbox{O}(h^{m-2}),  \label{normal_h_err_order} \\
S_h(u_n)_h(\alpha_i)&=
%S_h(u_n+\mbox{O}(h^{m-2}))(\alpha_i)
(u_n)_{\alpha}(\alpha_i)+\mbox{O}(h^{m-3})
\mbox{.}\label{normal_tangent_h_err_order}
\end{align}
%We note that a better estimate (i.e. one with higher order error) can be obtained for $(u_s)_h$, but (\ref{tangent_h_err_order}) will suffice for our purposes.
In addition, from (\ref{phi_s_disc}),
\begin{align}\label{phi_s_err_eqn}
(\phi_s)_h(\alpha_i)&=S_h^{-1}\bigpa{
(u_n)_h S_h\theta -\ip{(u_n)_h S_h\theta}_h
}(\alpha_i)=\phi_s(\alpha_i)+\mbox{O}(h^{m-2})\mbox{,}
\end{align}
where the latter equality follows from 
(\ref{normal_h_err_order}). 
%Lemmas \ref{trapezoidal_rule_estimate}-\ref{pseudo_int_est_lemma}, (\ref{normal_h_err_order}), and $u_n(\cdot)\in C^{m-1}$.
%In particular, note that the leading source of error in (\ref{normal_tangent_h_err_order}) and (\ref{phi_s_err_eqn}) comes from the $\mbox{O}(h^{m-2})$ term in (\ref{normal_h_err_order}).
Combined, the above results show that the truncation errors for the $\theta,s_{\alpha},\alpha_0$ evolution equations (\ref{theta_i_time_der_eqn})-(\ref{sigma_i_time_der_eqn}), (\ref{alpha_0_t_disc_16_2}) are given by
\begin{align}
\FD{t}\theta(\alpha_i)&=\frac{1}{s_{\alpha}}\bigmpa{S_h (u_n)_h(\alpha_i)
+S_h \theta(\alpha_i) (\phi_s)_h(\alpha_i)}+\mbox{O}(h^{m-3})\mbox{,}\label{theta_err_eqn} \\
\FD{t} s_{\alpha}&=-\ip{(u_n)_h(\cdot)S_h\theta(\cdot)}_h+\mbox{O}(h^{m-2})\mbox{,}
\label{s_alpha_err_eqn}\\
\PAD{t}{\alpha_0}(\alpha_i)&=\frac{S_h\alpha_0(\alpha_i)}{s_{\alpha}e^{i\theta(\alpha_i)}}
\bigpa{
(u_n)_h ie^{i\theta}+(\phi_s)_h e^{i\theta}-u_h
}(\alpha_i)+\mbox{O}(h^{m-2})\mbox{.}
\end{align}
We also need to check consistency of the discrete version of kinematic condition (\ref{vel_eqn_9a_1}).
Differentiate (\ref{Z_Int_disc}) with respect to $t$ to obtain
\begin{align}
\frac{d\tau_i}{dt}
=\frac{d\tau_c}{dt}
+S_h^{-1}\Bigpa{
\frac{d\sigma}{dt}
e^{i\theta}
+i\sigma e^{i\theta}\frac{d\theta}{dt}
-\ip{
\frac{d\sigma}{dt}
e^{i\theta}
+i\sigma e^{i\theta}\frac{d\theta}{dt}
}_h
}_i\mbox{,}
\end{align}
where from (\ref{tau_c_time_der_hat_u_0}),
\begin{align}
\frac{d\tau_c}{dt}=\ip{v}_h\mbox{.}
\end{align}
Then it is easy to see that
\begin{align}
\frac{d\tau}{dt}(\alpha_i)
=\frac{d\tau_h}{dt}(\alpha_i)
+\mbox{O}(h^{m-3})\mbox{.}
\end{align}
Taken together, the above results prove the following consistency result:
\begin{lemma}
Under the assumption that $\theta(\cdot,t)$ and $\alpha_0(\cdot,t)$ are in $C^{m+1}[-\pi,\pi]$,
$\MCAL{S}_0(\cdot)$ is in $C^m[-\pi,\pi]$, and $\omega(\cdot,t)$ is in $C^{m-1}[-\pi,\pi]$,
the exact solution of the evolution equations satisfy the discrete equations with a truncation error at most of size O($h^{m-3}$).
\end{lemma}

\section{Statement of Main Convergence Theorem}
\noindent To show convergence of the numerical method, we need to establish the stability of the discrete scheme.
% Begin insert 10
We first do this for special case of viscosity matched fluids, for which $\beta=0$ and $\chi=\frac{1}{2}$.
% End insert 10
Define the errors between the exact and numerical solutions as
\begin{align}
\dot{\theta}_j=\theta_j-\theta(\alpha_j)\mbox{,}\nonumber \\
\dot{\omega}_j=\omega_j-\omega_h(\alpha_j)\mbox{,}\nonumber \\
\dot{u}_j=u_j-u_h(\alpha_j)\mbox{,}
\end{align}
and so forth. To show stability, we plan to obtain a system of evolution equations for these errors and perform energy estimates to show they remain bounded for $t\leq T$, where $T$ is the assumed existence time for an exact solution to the continuous problem.

\begin{comment}
To illustrate this idea, let us find an equation for $\dot{\theta}$. First, substitute the exact solution into (\ref{theta_err_eqn})$-$(\ref{s_alpha_err_eqn}) and use the consistency lemma to get:
\begin{align}\label{theta_salpha_dot}
\frac{d\dot{\theta}_i}{dt}&=\frac{1}{\sigma_i}\bigmpa{S_h (u_n)_i+S_h \theta_i(\phi_s)_i}\nonumber \\
&-\frac{1}{\sigma_h}\bigmpa{S_h (u_n)_h(\alpha_i)+S_h \theta(\alpha_i)(\phi_s)_h(\alpha_i)}
+\mbox{O}(h^{m-3})\nonumber \\
\frac{d\dot{\sigma}_i}{dt}&=S_h (\phi_s)_i+(S_h \theta_i)(u_n)_i\nonumber \\
&-\bigmpa{S_h (\phi_s)_h(\alpha_i)+S_h \theta(\alpha_i)(u_n)_h(\alpha_i)}
+\mbox{O}(h^{m-2})\nonumber \\
&=-\ip{(u_n) S_h(\theta(\cdot))}_i
+\ip{(u_n)_h(\alpha_i)S_h(\theta(\alpha_i))}_h+\mbox{O}(h^{m-2})\mbox{.}
\end{align}
Now, the right hand side of the above equation can be written in terms of $\dot{\theta}_i,\dot{s}_{\alpha}$,
and the variations in velocities:
\begin{align}\label{vel_dots_eqn}
\dot{(u_n)}_i&=(u_n)_i-(u_n)_h(\alpha_i)\mbox{, }
\dot{(\phi_s)}_i=(\phi_s)_i-(\phi_s)_h(\alpha_i)\mbox{.}
\end{align}
\end{comment}
Therefore, our first task is to estimate quantities such as $\dot{(u_n)}_i$ and $\dot{(\phi_s)}_i$ in terms of the errors $\dot{\theta}_i,\dot{\sigma}_i$, $\dot{\alpha}_{0i}$.
This can be done by identifying the most singular part in the variation $\dot{u}_i=u_i-u_h(\alpha_i)$ of the complex velocity.
The estimates can be separated into linear and nonlinear terms in $\dot{\theta}_i,\dot{\sigma}_i$, $\dot{\alpha}_{0i}$.
The nonlinear terms can be controlled by the high accuracy of the method for smooth solutions.
Thus the leading order error contribution comes from the linear terms.

% Begin insert 11a
%
\begin{comment}
\begin{remark}
If $\dot{\phi}$ is a scalar quantity, we will sometimes use the notation $\mbox{O}(\dot{\phi})$ to denote a bounded operator in $l^2$, i.e.
\begin{align}
\norm{\mbox{O}(\dot{\phi})}_{l^2}
\leq c\abs{\dot{\phi}}\mbox{.}
\end{align}
Thus, $\mbox{O}(\dot{\phi})$ is equivalent to $A_0(\dot{\phi})$.
For example, if $f(\cdot)\in C[-\pi,\pi]$, then $f(\alpha_i)(\dot{\phi})=\mbox{O}(\dot{\phi})$.
\end{remark}
\end{comment}
% End insert 11
%

We now state the convergence theorem for our numerical method:
\begin{thm}\label{MainThm}
Assume that for $0\leq t\leq T$ there exists a smooth solution of the continuous problem (\ref{S_ALPHA_EVOLVEQN})$-$(\ref{THETA_EVOLVEQN}), 
%(\ref{vel_eqn_9a_1}), 
(\ref{alpha_0_t_disc_16_2}) with $\theta(\cdot,t)$, $\alpha_0(\cdot,t)$ in $C^{m+1}[-\pi,\pi]$ and $\MCAL{S}(\cdot,0)\in C^m[-\pi,\pi]$ for $m$ sufficiently large, and that:
\begin{align}
\DPS\min_{ 0\leq t\leq T} s_{\alpha}(t)>c
\mbox{, for some }c>0\mbox{.}
\end{align}
If $\sigma^{(h)}$, $\theta^{(h)}$ and $\alpha_0^{(h)}$ denote the numerical solution for $s_{\alpha},~\theta$ and $\alpha_0$, then for $h$ and $\beta$ sufficiently small and for all $0 \leq t \leq T$,
\begin{align}
\norm{\sigma^{(h)}(t)-s_{\alpha}(t)}_{l^2}&\leq c(T)h^s\mbox{,}\nonumber \\
\norm{\theta^{(h)}(t)-\theta(\cdot,t)}_{l^2}&\leq c(T)h^s\mbox{,}\nonumber \\
\norm{\alpha_0^{(h)}(t)-\alpha_0(\cdot,t)}_{l^2}&\leq c(T)h^s\mbox{,}
\end{align}
where $s=m-l$ and $l$ is small positive integer that is independent of $m$ (i.e., $s$ is near $m$).  In addition, the discrete interface shape $\tau^{(h)}(t)$ satisfies
\begin{align}
    \norm{\tau^{(h)}(t)-\tau(\cdot,t)}_{l^2}&\leq c(T)h^s\mbox{.}
\end{align}
Here
\begin{align}
\norm{u}_{l^2}&= \Bigpa{h\DPS\sum_{j=-\frac{N}{2}+1}^{\frac{N}{2}}\abs{u_j}^2}^\frac{1}{2} \mbox{.}
\end{align}
\end{thm}

\section{Stability: Preliminaries}
% Begin move ** to ?
Following \cite{BHL:1996} and \cite{CH:1998}, we introduce notation for an $n$-th order smoothing operator $A_{-n}$ which acts on a discrete function $\phi_j$ and satisfies
\begin{align}
\norm{D_h^k(A_{-n}(\dot{\phi}))}_{l^2}&\leq c\norm{\dot{\phi}}_{l^2}\mbox{ and }
\norm{A_{-n}(S_h^k(\dot{\phi}))}_{l^2}\leq c\norm{\dot{\phi}}_{l^2}
\mbox{ for $0\leq k\leq n$,}
\end{align}
where $S_h$ is the spectral derivative (\ref{S_h_defn}) and   $D_h$ is the spectral derivative operator with smoothing.
When $n=0$, $A_0(\dot{\phi_j})$ denotes a bounded operator in $l^2$,
\begin{align}
\norm{A_0(\dot{\phi})}_{l^2}&\leq c\norm{\dot{\phi}}_{l^2}\mbox{.}
\end{align}
\begin{remark}
Note that if $f(\dot{\phi}_i)=A_0(\dot{\phi}_i)$, then $h^s f(\dot{\phi}_i)=A_{-s}(\dot{\phi}_i)$.
However, $f(\dot{\phi}_i)=A_{-s}(\dot{\phi}_i)$ does not imply $f=\mbox{O}(h^s)$.
%An example is $f(\dot{\phi}_i)=1+h^s\dot{\phi}_i$, which is an $A_{-s}(\dot{\phi}_i)$, but not $\mbox{O}(h^s)$.
\end{remark}
\begin{remark}
We use the expression $A_{-s}(\dot{\phi}_i)$ to denote a generic high-order smoothing operator.
Generally, $s$ is an integer near $m$, 
%e.g. $m+1$ or $m-1$, 
where $m$ defines the regularity of the continuous solution (e.g.,  $\theta(\cdot,t)\in C^{m+1}$, etc.).
Similarly, we denote by  $\mbox{O}(h^s)$  a generic high-order discretization error. %again where $s$ is near $m$.
At the end of our proof, we choose $m$ and $s$ large enough so that all the estimates go through.
\end{remark}
\begin{remark}
Unless otherwise noted, we use the phrase ``smooth function" to denote a generic function $f\in C^s[-\pi,\pi]$ with high order regularity.
\end{remark}
%\begin{remark}
%We sometimes will use the notation
%$A_{-s}(\dot{\phi}_i,\dot{\psi}_i,\ldots,\dot{\omega}_i)=A_{-s}(\dot{\phi}_i)+A_{-s}(\dot{\psi}_i)+%\ldots+A_{-s}(\dot{\omega}_i)$.
%\end{remark}
% end move ** to ?
% Begin of copy paste 1
We define a time 
\begin{align} \label{T*def}
T^{\ast}\equiv \sup\bigbra{t: 0\leq t\leq T, \norm{\dot{\sigma}}_{l^2}
,\norm{\dot{\theta}}_{l^2},\norm{\dot{\zeta}}_{l^2},\norm{\dot{\alpha_0}}_{l^2}\leq h^{\frac{7}{2}} }
\mbox{,}
\end{align}
where the  power of $h$ in (\ref{T*def}) is chosen for so that the estimates below easily go through.
All the estimates we obtain are valid for $t\leq T^{\ast}$. We ``close the argument" and prove Theorem \ref{MainThm} by showing at the end that $T^{\ast}=T$, the existence time of the continuous solution.
% End of copy paste 1
We make repeated use of the inequalities
\begin{align}\label{dot_theta_sigma_alpha_0_assumptions_135}
\norm{\dot{\theta}}_{\infty}\leq h^3\mbox{, }
\norm{\dot{\sigma}}_{\infty}\leq h^3\mbox{, and }
\norm{\dot{\alpha_0}}_{\infty}\leq h^3\mbox{, for }
t\leq T^{\ast}.
\end{align}
The above estimate on $\dot{\theta}$ follows from $h\abs{\dot{\theta}_i}^2\leq \norm{\dot{\theta}}_{l^2}^2$, for $t\leq T^{\ast}$, so that $\norm{\dot{\theta}}_{\infty}\leq h^{-\frac{1}{2}}\norm{\dot{\theta}}_{l^2}\leq h^3$, with similar estimates applying to $\norm{\dot{\sigma}}_{\infty}$ and $\norm{\dot{\alpha_0}}_{\infty}$.
\newline
\newline
\underline{\textit{Preliminary Lemmas}}.
\newline
\newline
%We state a number of preliminary lemmas that will be repeatedly used in the analysis.

We will frequently encounter a discrete operator of the form:
\begin{equation}\label{RHDEFN}
R_h(\phi_i)= \DPS\sum_{\substack{j=-\frac{N}{2}+1\\(j-i){\rm\ odd}}}^{\frac{N}{2}}f(\alpha_i,\alpha_j) \phi_j (2h)\mbox{,}
\end{equation}
where $f(\alpha,\alpha')$ is a smooth periodic function in both variables, and $\phi$ is a discrete periodic function. Beale, Hou and Lowengrub \cite{BHL:1996} prove the following estimate on $R_h$ applied to a filtered discrete function $\phi_i^p$:
\begin{lemma}\label{R_h_lemma}
Assume $f(\alpha,\alpha')$ is a smooth periodic function in both $\alpha$ and $\alpha'$, with $f(\cdot,\cdot)$ 
%in C^r$ 
%for $r>2$, then $R_h$ defined in equation (\ref{RHDEFN}) satisfies:
%\begin{equation}\label{RHPhiP1}
%R_h(\phi^p)=A_{-1}(\phi)\mbox{.}
%\end{equation}
%If, in addition, $\rho(x)$ satisfies $\rho'(\pm \pi)=0$ and $f(\cdot_1,\cdot_2)\in 
in $C^r$ for $r>3$. Then
\begin{equation}\label{RHPhiP2}
R_h(\phi^p_i)=A_{-2}(\phi_i)\mbox{.}
\end{equation}
\end{lemma}
%Unless otherwise noted, a superscript $p$ will henceforth denote a filter satisfying $\rho'(\pm \pi)=0$.
We note that the application of the filter is essential in  (\ref{RHPhiP2}) due to aliasing error. To see this, consider the following example adapted from \cite{BHL:1996}.
Let $g(\alpha)=e^{2i\alpha}$, define
\begin{align}
f(\alpha,\alpha')=\frac{1}{2\pi}\bigpa{g(\alpha)-g(\alpha')}\cot\Bigpa{\frac{\alpha-\alpha'}{2}}
\mbox{, }
f(\alpha,\alpha)=\frac{g_{\alpha}(\alpha)}{\pi}\mbox{, }
\end{align}
and let $\phi_i=e^{i\alpha_i(\frac{N}{2}-1)}$. Then using Lemma \ref{Cauchy_Transform_Lemma} below and the fact that $e^{i\alpha_i(\frac{N}{2}+1)}$ is aliased to $e^{i\alpha_i(-\frac{N}{2}+1)}$, we have
\begin{align}
R_h(\phi_i)&=\frac{e^{2i\alpha_i}}{2\pi}\DPS\sum_{\substack{j=-\frac{N}{2}+1\\(j-i){\rm\ odd}}}^{\frac{N}{2}}
e^{i\alpha_j(\frac{N}{2}-1)}
\cot\Bigpa{\frac{\alpha_i-\alpha_j}{2}}
(2h)\nonumber\\
&-\frac{1}{2\pi}\DPS\sum_{\substack{j=-\frac{N}{2}+1\\(j-i){\rm\ odd}}}^{\frac{N}{2}}
e^{i\alpha_j(-\frac{N}{2}+1)}
\cot\Bigpa{\frac{\alpha_i-\alpha_j}{2}}
(2h) \nonumber\\
&=-i\bigpa{e^{2i\alpha_i}e^{i\alpha_i(\frac{N}{2}-1)}+e^{i\alpha_i(-\frac{N}{2}+1)}}\mbox{}\nonumber\\
&=-2ie^{i\alpha_i(-\frac{N}{2}+1)}
=-2ig(\alpha_i)\phi_i=A_0(\phi_i)\mbox{.}
\end{align}

\begin{remark}\label{R_h_A_0_no_filt_equiv_remark}
If no filtering is applied, then it is easy to see that
\begin{equation}\label{RHPhiNoFilter}
R_h(\phi_i)=A_0(\phi_i)\mbox{.}
\end{equation}
\end{remark}
\noindent
Indeed we note that by the Schwartz inequality,
\begin{align}
\norm{R_h(\phi)}_{l^2}&= \Biggpa{h\DPS\sum_{i=-\frac{N}{2}+1}^{\frac{N}{2}}
\Biggabs{
\DPS\sum_{\substack{j=-\frac{N}{2}+1\\(j-i){\rm\ odd}}}^{\frac{N}{2}}
f(\alpha_i,\alpha_j) \phi_j (2h)}^2
}^{\frac{1}{2}}\nonumber\\
&\leq 2h \norm{\phi}_{l^2} \left(  \DPS\sum_{i=-\frac{N}{2}+1}^{\frac{N}{2}}
\DPS\sum_{\substack{j=-\frac{N}{2}+1\\(j-i){\rm\ odd}}}^{\frac{N}{2}}
\abs{f(\alpha_i,\alpha_j)}^2 
\right)^\frac{1}{2}
\nonumber\\
&\leq 4\norm{f}_{l^2}\norm{\phi}_{l^2}\mbox{,}
\end{align}
where
\begin{align}
\norm{f}_{l^2}=\biggpa{
h^2
\DPS\sum_{i=-\frac{N}{2}+1}^{\frac{N}{2}}
\DPS\sum_{j=-\frac{N}{2}+1}^{\frac{N}{2}}
\abs{f(\alpha_i,\alpha_j)}^2
}^{\frac{1}{2}}
\mbox{.}
\end{align}
%which gives (\ref{RHPhiNoFilter}).

The Hilbert transform $\MCAL{H}\omega$ is the leading order part (i.e., least regular term) in the velocity (\ref{velintsubstractHilb_1}).
%The discrete version of $\MCAL{H}(\omega)$ is:
%\begin{align}\label{H_h_disc}
%\MCAL{H}_h(\omega_i)&=\frac{1}{2\pi}
%\DPS\sum_{\substack{j=-\frac{N}{2}+1\\(j-i){\rm\ %odd}}}^{\frac{N}{2}}
%\omega_j\cot\Bigpa{\frac{\alpha_i-\alpha_j}{2}}(2h)\mbox{.}
%\end{align}
This will be seen to play a crucial rule in the stability of our discretization.
The continuous Hilbert transform satisfies
\begin{align}
\widehat{(\MCAL{H} f)}_k &= -i\mbox{sgn}(k)\hat{f}_k\mbox{,}
\label{H_cont_disc_real_line_prop_1} \\
\MCAL{H} (\MCAL{H} f(\alpha)) &= -f(\alpha)\mbox{,}
\label{H_cont_disc_real_line_prop_2} \\
\MCAL{H} (f_{\alpha})(\alpha) &= (\MCAL{H} f)_{\alpha}(\alpha)
\label{H_cont_disc_real_line_prop_3}
\end{align}
for a periodic function $f$ with zero mean.
The following lemma from \cite{BHL:1996} shows that the discrete transform acts in the same way.
\begin{lemma}\label{Cauchy_Transform_Lemma}
Assume that $f$ satisfies $\Hat{f}_0=\Hat{f}_{\frac{N}{2}}=0$. The discrete Hilbert transform  (\ref{discrete_Hilbert}) satisfies the following properties:
\end{lemma}
\begin{align}
\widehat{(\MCAL{H}_h f)}_k &= -i\mbox{sgn}(k)\hat{f}_k\mbox{,}
\label{H_h_disc_real_line_prop_1} \\
\MCAL{H}_h (\MCAL{H}_h f_i) &= -f_i\mbox{,}
\label{H_h_disc_real_line_prop_2} \\
\MCAL{H}_h (S_h f_i)&=S_h (\MCAL{H}_h f_i)
=\frac{1}{\pi}\DPS\sum_{(j-i)\mbox{\rm\ odd}} \frac{f_i-f_j}{(\alpha_i-\alpha_j)^2} (2h)\mbox{.}
\label{H_h_disc_real_line_prop_3}
\end{align}
where $\DPS\sum_{(j-i){\rm\ odd}}$ is defined in (\ref{H_h_disc_lim}) below.
The first equality above also implies $\norm{\MCAL{H}_h f}_{l^2}=\norm{f}_{l^2}$.
\begin{proof}
We transform the kernel in (\ref{discrete_Hilbert}) from a representation in the periodic domain to an equivalent representation in the infinite domain.
This involves application of the formula \cite{AIMark:1977},
\begin{align}
\frac{1}{2}\cot(\frac{z}{2})=\frac{1}{z}+\DPS\sum_{k=1}^{\infty}
\frac{2z}{z^2-(2k\pi)^2}\mbox{,}
\end{align}
from which it is easy to obtain
(see \cite{BHL:1996} for details)
\begin{align}\label{H_h_disc_lim}
&\frac{1}{2}\DPS\sum_{\substack{j=-\frac{N}{2}+1\\(j-i){\rm\ odd}}}^{\frac{N}{2}}
f(\alpha_j)\cot\Bigpa{\frac{\alpha_i-\alpha_j}{2}}(2h)
=\DPS\lim_{M\RA\infty}\DPS\sum_{\substack{j=-M(N+\frac{1}{2})+1\\(j-i){\rm\ odd}}}^{N(M+\frac{1}{2})}
\frac{f(\alpha_j)}{\alpha_i-\alpha_j}(2h)\mbox{,} \nonumber \\
\equiv& \DPS\sum_{(j-i){\rm\ odd}}\frac{f(\alpha_j)}{\alpha_i-\alpha_j} (2h)\mbox{.}
\end{align}
where $\alpha_j=jh$ is extended outside the interval $-\frac{N}{2} +1 \leq j \leq \frac{N}{2}$.  We note that the first equality in (\ref{H_h_disc_lim}) relies on the particular form of the bounds in the right hand sum, although we will use the notation $\sum_{(j-i){\rm\ odd}}$   to denote more general infinite sums.
It follows from   (\ref{H_h_disc_lim}) with $f_j$ in place of $f(\alpha_j)$ that
\begin{align}\label{H_h_disc_real_line}
\MCAL{H}_h f_i&=\frac{1}{\pi}\DPS\sum_{(j-i){\rm\ odd}}
\frac{f_j}{\alpha_i-\alpha_j} (2h)\mbox{,}
\end{align}
is an equivalent form of (\ref{discrete_Hilbert}).
This form is proven to satisfy properties (\ref{H_h_disc_real_line_prop_1})-(\ref{H_h_disc_real_line_prop_3}) in \cite{BHL:1996}.
\end{proof}
We also need the following result on the commutator of the discrete Hilbert transform and a smooth function, from \cite{BHL:1996}.
\begin{lemma}\label{Commutator_thm_H_h_g}
Let $g(\cdot)\in C^r$ for $r>4$, and consider the commutator
\begin{align}
\bigmpa{\MCAL{H}_h,g}(\phi^p_i)
&=\MCAL{H}_h(g(\alpha_i)\phi^p_i)-g(\alpha_i)\MCAL{H}_h(\phi_i^p)\mbox{.}
\end{align}
%Then $\bigmpa{\MCAL{H}_h,g}(\phi^p)\in A_{-1}(\phi^p)$.
%If in addition $\rho(x)$ satisfies $\rho'(\pm\pi)=0$ and $g(\cdot)\in C^r$ with $r>4$, then 
Then $\bigmpa{\MCAL{H}_h,g}(\phi^p_i) =  A_{-2}(\phi_i)$.
\end{lemma}
\begin{proof}
Let
\begin{align}
f(\alpha,\alpha')=\frac{1}{2\pi}\bigpa{g(\alpha)-g(\alpha')}
\cot\Bigpa{
\frac{\alpha-\alpha'}{2}
}
\mbox{ }
\in C^{r-1}
\end{align}
and apply Lemma \ref{R_h_lemma}.
The result follows from noting that for this $f$,
\begin{align}
R_h(\phi_i^p)=\MCAL{H}_h(g(\alpha_i)\phi^p_i)-g(\alpha_i)\MCAL{H}_h \phi_i^p\mbox{.}
\end{align}
\end{proof}
We will also need a lemma on the commutator of the filtering operator and a smooth function.
The proof can be found in \cite{CeniPHDthesis:1995}.
\begin{lemma}\label{G_h_p_lemma}
Let $f(\alpha_i)\in C^r$ for $r\geq 2$, and $\phi\in l^2$.
Define:
\begin{align}
G_h^p(\phi_i)=\bigpa{f(\alpha_i)\phi_i}^p-f(\alpha_i)\phi_i^p\mbox{.}
\end{align}
Then $G_h^p(\phi_i)= A_{-1}(\phi_i)$.
\end{lemma}
In our stability analysis, we will need an analogue of the product rule for discrete derivative operators, proven in \cite{BHL:1996}.
\begin{lemma}\label{product_rule_for_D_h_lemma}
Assume $f(\cdot)\in C^3$ and $w\in l^2$. Then we have:
\begin{equation} \label{eq:discrete_product_rule}
D_h(f(\alpha_i)w_i)=f(\alpha_i)D_h(w_i)+w_i^q f_{\alpha}(\alpha_i)
+ A_{-1}(w_i)\mbox{,}
\end{equation}
where $\hat{w}_k^q=\hat{w}_k q(kh)$, $q(x)=\FD{x}(x\rho(x))$, and $A_0$ is a bounded operator.
\end{lemma}
We will apply the following lemmas to obtain expressions for the variation of velocities and other quantities in our problem.
Recall the definition  of error between the exact and numerical solution,
\begin{equation}\label{dot_f_defn}
\dot{f}_i=f_i-f(\alpha_i)\mbox{.}
\end{equation}
It is straightforward to derive the following product rule for errors:
\begin{lemma}\label{product_rule_for_dots_lemma}
Let $\dot{f}_i,\dot{g}_i$ and $f(\alpha_i),g(\alpha_i)$ be as defined in (\ref{dot_f_defn}).
Then $\bigpa{f_i g_i}^{\cdot}=\dot{f}_i g(\alpha_i)+f(\alpha_i) \dot{g}_i+\dot{f}_i \dot{g}_i$.
\end{lemma}
\begin{comment}
\begin{proof}
\begin{align}
\bigpa{f_i g_i}^{\cdot}&=f_i g_i-f(\alpha_i) g(\alpha_i) \nonumber \\
&=f_i g(\alpha_i)-f(\alpha_i) g(\alpha_i)+ f_i g_i - f_i g(\alpha_i) \nonumber \\
&=(f_i-f(\alpha_i)) g(\alpha_i)+ f_i (g_i - g(\alpha_i)) \nonumber \\
&=\dot{f}_i g(\alpha_i)+f_i \dot{g}_i\nonumber \\
%&=\dot{f}_i g(\alpha_i)+(f(\alpha_i)+\dot{f}_i) \dot{g}_i\nonumber \\
&=\dot{f}_i g(\alpha_i)+f(\alpha_i)\dot{g}_i+\dot{f}_i\dot{g}_i\mbox{.}
\end{align}
%Note if $\dot{f}_i$ is $\mbox{O}(h^{m_1})$, and if $\dot{g}_i$ has error $\mbox{O}(h^{m_2})$, then $\bigpa{\dot{fg}}_i$ is $\mbox{O}(h^{\min\{m_1,m_2\}})$.
\end{proof}
\end{comment}
The above lemma can be easily extended to products of three or more quantities.
We also have
\begin{lemma}\label{quotient_rule_1_over_f_lemma}
\begin{align}\label{quotient_eqn_f}% labeled by line 2010,2085?
\Bigpa{\frac{1}{f_i}}^{\cdot}&=-\frac{\dot{f}_i}{f^2(\alpha_i)}
+\frac{\dot{f}_i^2}{f^2(\alpha_i)\bigpa{f(\alpha_i)+\dot{f}_i}}
\mbox{.}
\end{align}
\end{lemma}
\begin{proof}
\begin{align}
\Bigpa{\frac{1}{f_i}}^{\cdot}&=\frac{1}{f_i}-\frac{1}{f(\alpha_i)}
=\frac{f(\alpha_i)-f_i}{f_i f(\alpha_i)}
=-\frac{\dot{f}_i}{f(\alpha_i)\bigpa{f(\alpha_i)+\dot{f}_i}} \label{one_over_fdot}
\mbox{,}
\end{align}
where in the last equality, we have eliminated $f_i$ using $f_i=f(\alpha_i)+\dot{f}_i$.
After decomposing the right hand side of (\ref{one_over_fdot}) into a sum of linear and nonlinear terms in $\dot{f}_i$, we obtain the result. 
\end{proof}
We will need the following results on $\dot{\sigma}=\sigma-s_{\alpha}$.
\begin{lemma}\label{s_alpha_err_int_prod_lemma}
Let $f$ be a smooth function and $\ip{f(\cdot)}_h=0$. Then:
\begin{align}\label{s_alpha_err_int_prod_eqn}
S_h^{-1}\bigpa{f(\alpha_i)\dot{\sigma}}
&=\dot{\sigma} S_h^{-1}\bigpa{f(\alpha_i)}
=A_{-s}(\dot{\sigma}_i)\mbox{,}
\end{align}
where $A_{-s}(\dot{\sigma}_i)$ is here interpreted for the spatially independent $\dot{\sigma}$ as
\begin{align}
A_{-s}(\dot{\sigma}_i)
=\dot{\sigma} g(\alpha_i)
\end{align}
for some smooth function $g$.
\end{lemma}
\begin{proof}
The relation (\ref{s_alpha_err_int_prod_eqn}) follows from the spatial independence of $\dot{\sigma}$. The second equality follows from the smoothness of $f$.
\end{proof}
Similarly we have from Lemma \ref{S_h_estimate}
\begin{lemma}\label{num_approx_S_h_f_dot_sigma_decomp_lemma}
Let $f$ be a smooth function. Then
\begin{align}\label{num_approx_S_h_f_dot_sigma_decomp_eqn_165}
S_h\bigpa{f(\alpha_i)\dot{\sigma}}
&=f_{\alpha}(\alpha_i)\dot{\sigma}+O(h^s)=A_{-s}(\dot{\sigma}_i)+\mbox{O}(h^s)\mbox{,}
\end{align}
and the same is true for $D_h$ instead of $S_h$.
\end{lemma}
%\begin{proof}
%We have
%\begin{align}
%S_h\bigpa{f(\alpha_i)\dot{\sigma}}
%&=\dot{\sigma} S_h f(\alpha_i)=\dot{\sigma} f_{\alpha}(\alpha_i)+\mbox{O}(h^s)
%\mbox{,}
%\end{align}
%using Lemma \ref{S_h_estimate}, and (\ref{num_approx_S_h_f_dot_sigma_decomp_eqn_165}) %immediately follows.
%\end{proof}
We will make repeated use of the following result from \cite{CH:1998}:
\begin{lemma}\label{tan_vector_est_lemma}
For $\dot{\theta},\dot{\sigma}$ satisfying $\norm{\dot{\theta}}_{\infty}\leq h^3, \norm{\dot{\sigma}}_{\infty}<h^3$, then
\begin{align}\label{s_alpha_e_i_theta_err_eqn_simp}
\bigpa{\sigma e^{i\theta}}_i^{\cdot}
&=i s_{\alpha} e^{i\theta(\alpha_i)}\dot{\theta}_i
+e^{i\theta(\alpha_i)} \dot{\sigma}_i
+A_{-3}(\dot{\theta}_i)=\mbox{O}(h^3)\mbox{.}
\end{align}
\end{lemma}
\begin{proof}
We have
\begin{align}\label{s_alpha_e_i_theta_err_eqn}
\bigpa{\sigma e^{i\theta}}_i^{\cdot}
&=\sigma e^{i\theta_i}-s_{\alpha}e^{i\theta(\alpha_i)}
\nonumber\\
&=s_{\alpha}e^{i\theta(\alpha_i)}\bigpa{e^{i\dot{\theta}_i}-1}
+\bigpa{\sigma-s_{\alpha}}e^{i\theta(\alpha_i)}e^{i\dot{\theta}_i}
\nonumber\\
&=s_{\alpha}e^{i\theta(\alpha_i)}\bigpa{i\dot{\theta}_i+r(\dot{\theta}_i)}
+\dot{\sigma}_i e^{i\theta(\alpha_i)}\bigpa{1+i\dot{\theta}_i+r(\dot{\theta}_i)},
\end{align}
where $\norm{r(\dot{\theta})}_{l^2}\leq \norm{\dot{\theta}^2}_{l^2}\leq h^3\norm{\dot{\theta}}_{l^2}$.
The first equality in (\ref{s_alpha_e_i_theta_err_eqn_simp}) follows by using $\norm{\dot{\sigma}}_{\infty}\leq h^3$ and the above estimate on $\norm{r(\dot{\theta})}_{l^2}$ to show that the nonlinear terms in the variation are $A_{-3}(\dot{\theta})$.
The second equality in (\ref{s_alpha_e_i_theta_err_eqn_simp}) readily follows from (\ref{s_alpha_e_i_theta_err_eqn}), the bounds $\norm{\dot{\theta}}_{\infty}\leq h^3, \norm{\dot{\sigma}}_{\infty}<h^3$, and the above estimate on $\norm{r(\dot{\theta})}_{l^2}$.
\end{proof}

In estimating the variations, we will make use of the discrete Parseval equality.
First, recall that the $l^2$ inner product is
\begin{align}
\bigpa{f,g}_h=h\DPS\sum_{i=-\frac{N}{2}+1}^{\frac{N}{2}}
\OVL{f_i} g_i\mbox{,}
\end{align}
for $f,g\in l^2$. Then, we have
\begin{lemma}\label{IP_prod_lemma}
(Discrete Parseval's equality).
Let $f,g\in l^2$.
Then
\begin{align}
\bigpa{f,g}_h=2\pi\DPS\sum_{k=-\frac{N}{2}+1}^{\frac{N}{2}}
\OVL{\hat{f}_k} \hat{g}_{k}\mbox{.}
\end{align}
In particular, when $g_i=f_i$
\begin{align}
\norm{f}_{l^2}^2=2\pi\DPS\sum_{k=-\frac{N}{2}+1}^{\frac{N}{2}}
\abs{\hat{f}_k}^2\mbox{.}
\end{align}
\end{lemma}
A simple consequence of Parseval's equality is that derivatives can be transferred to a smooth function, similar to integration by parts.
\begin{lemma}\label{product_rule_f_S_h_g_lemma}
Let $f(\cdot)\in C^{s+1}$ and $g\in l^2$. Then
\begin{align}
\bigpa{f(\cdot),S_h g}_h=-\bigpa{S_h f(\cdot), g}_h
=A_{-s}(g)\mbox{.}
\end{align}
The same result holds for the discrete average:
\begin{align}
\ip{f(\cdot)S_h g}_h=-\ip{S_h f(\cdot) g}_h
= A_{-s}(g)\mbox{.}
\end{align}
Thus, when considered as an operator $Q$ on $g$,
\begin{align}
Q(g)\equiv\ip{f(\cdot)g}_h =  A_{-(s+1)}(g)\mbox{.}
\end{align}
\end{lemma}
\begin{proof}
By Parseval's equality
\begin{align}
\bigpa{f(\cdot),S_h g}_h&=2\pi\DPS\sum_{k=-\frac{N}{2}+1}^{\frac{N}{2}}
\OVL{\hat{f}_k} (ik)\hat{g}_k\nonumber \\
&=2\pi\DPS\sum_{k=-\frac{N}{2}+1}^{\frac{N}{2}}
\OVL{(-ik)\hat{f}_k} \hat{g}_k\nonumber \\
&=-\bigpa{S_h f(\cdot),g}_h\mbox{.}
\end{align}
The result on discrete average follows similarly.
%from:
%\begin{align}
%\ip{f(\cdot) g}_h=\frac{1}{2\pi}\bigpa{\OVL{f},g}_h\mbox{.}
%\end{align}
\end{proof}
\noindent
Note these results also hold for $D_h$.
We will also make use of the fact that, for smooth $f$,  $S_h^{-1}(f(\alpha_i)\phi_i)$ is a smoothing operator on $\phi_i$.
The proof of this includes, as a by-product, an `integration by parts' formula for $S_h^{-1}$.
\begin{lemma}\label{Int_product_rule_f_S_h_g_lemma}
Let $f\in C^3$, $\phi \in l^2$, and assume $f(\alpha_i)\phi_i$ has zero mean, i.e., $\ip{f(\cdot)\phi}_h\equiv 0$. Then
\begin{align}\label{num_approx_Int_prod_f_S_h_phi_1}
S_h^{-1}(f(\alpha_i)S_h \phi_i)
% integration by parts?
=-S_h^{-1}(S_h f(\alpha_i) \phi_i)
+f(\alpha_i)\phi_i
+A_0(\phi_i)\mbox{,}
\end{align}
and hence
\begin{align}\label{num_approx_Int_prod_phi_S_h_f_A}
S_h^{-1}(f(\alpha_i) \phi_i)
=A_{-1}(\phi_i)\mbox{.}
\end{align}
\end{lemma}
\noindent
The proof of Lemma \ref{Int_product_rule_f_S_h_g_lemma} is technical and is relegated to the appendix.
% End move 3 to line approx 1995?

The above lemma can be used to prove the following estimate on $\tau_i$:
\begin{lemma}\label{zeta_dot_i_lemma}
Let $\tau_i$ represent the discrete interface. Then
\begin{align}\label{zeta_dot_i_err_est_main_in_lemma}
\dot{\tau}_i&=i S_h^{-1}\bigpa{\sigma e^{i\theta(\alpha_i)} \dot{\theta}_i}
+A_{-3}(\dot{\theta}_i)+A_{-s}(\dot{\sigma}_i)
+\dot{\tau_c} \\
&= A_{-1}(\dot{\theta}_i)+A_{-s}(\dot{\sigma}_i)
+\dot{\tau_c}. 
\end{align}
%It will be shown later that $\dot{\tau_c}=\mbox{O}(h^{\frac{5}{2}})$.
\end{lemma}
\begin{proof}
Recall that
\begin{align}
\tau_i=S_h^{-1} \bigpa{
\sigma e^{i\theta_i}-\ip{\sigma e^{i\theta}}_h
}+\tau_c\mbox{,}
\end{align}
where $\tau_c=\hat{\tau}_0$ is the $k=0$ Fourier mode of $\tau_i$.
Taking the variation, we have
\begin{align}\label{zeta_dot_i}
\dot{\tau}_i=S_h^{-1} \bigpa{
(\sigma e^{i\theta})_i^{\cdot}-\ip{(\sigma e^{i\theta})^{\cdot}}_h
}+\dot{\tau_c}\mbox{.}
\end{align}
We now substitute the first equality in (\ref{s_alpha_e_i_theta_err_eqn_simp}) and use Lemma \ref{product_rule_f_S_h_g_lemma} to find that
%, which states that for $f\in C^s$ and an arbitrary variation $\dot{\phi}$, $\ip{f(\cdot)\dot{\phi}}_h=A_{-s}(\dot{\phi})$.
%Hence,
\begin{align}
\dot{\tau}_i&=i S_h^{-1} \bigpa{
\sigma e^{i\theta(\alpha_i)} \dot{\theta}_i
}
+S_h^{-1} \bigpa{
e^{i\theta(\alpha_i)} \dot{\sigma}
}\nonumber\\
&+\dot{\tau_c}
+A_{-3}(\dot{\theta_i})+A_{-s}(\dot{\sigma_i})\mbox{.}
\end{align}
Finally, from Lemma \ref{num_approx_S_h_f_dot_sigma_decomp_lemma} the second term on the right hand side of the above relation is $A_{-s}(\dot{\sigma_i})$, 
and by Lemma \ref{Int_product_rule_f_S_h_g_lemma} the first term is $A_{-1}(\dot{\theta})$, which gives the result.
\end{proof}
\begin{remark}\label{S_h_dot_zeta_i_remark}
Evidently, we also have from (\ref{zeta_dot_i})
\begin{align}
S_h \dot{\tau}_i=\bigpa{
(\sigma e^{i\theta})_i^{\cdot}-\ip{(\sigma e^{i\theta})_i^{\cdot}}_h
}\mbox{,}
\end{align}
and following the same reasoning as in the proof of Lemma \ref{zeta_dot_i_lemma}
\begin{align}\label{S_h_dot_zeta_i_A_estimate_eqn_174}
S_h \dot{\tau}_i&= i s_{\alpha} e^{i\theta(\alpha_i)} \dot{\theta}_i
+A_{-3}(\dot{\theta}_i)
+A_{-s}(\dot{\sigma}_i)\nonumber\\
&=A_0(\dot{\theta}_i)+A_{-s}(\dot{\sigma}_i)\mbox{.}
\end{align}
Since $\norm{\dot{\theta}}_{l^2}\leq h^{\frac{7}{2}}, \norm{\dot{\sigma}}_{l^2}<h^{\frac{7}{2}}$ and $\norm{\dot{\theta}}_{\infty}\leq h^3, \norm{\dot{\sigma}}_{\infty}<h^3$, it also follows that $\norm{S_h\dot{\tau}}_{l^2}\leq h^{\frac{7}{2}}$ and $\norm{S_h\dot{\tau}}_{\infty}\leq h^3$.
\end{remark}

\section{Estimates for the variation of velocities}
Recall the discrete equation for velocity has been decomposed as $u_i=\MCAL{H}_h \omega_i + (u_R)_i$ (cf. (\ref{vel_decomp})).
We further decompose  $(u_R)_i$ as $(u_R)_i=U_{1,i}+U_{2,i}+U_{3,i}+U_{4,i}$, where
\begin{align}
U_{1,i}&=
%\MCAL{H}_h\omega_i
-\frac{h}{\pi}
\DPS\sum_{\substack{j=-\frac{N}{2}+1\\(j-i)\mbox{ odd}}}^{\frac{N}{2}}
\omega_j^p \biggpa{
2\mbox{Re}\Bigpa{\frac{S_h \tau_j}{\tau_j-\tau_i}}
+\cot\Bigpa{\frac{\alpha_i-\alpha_j}{2}}
}
\mbox{,}\label{UI12_disc_est} \\
U_{2,i}&=\frac{h}{\pi}
\DPS\sum_{\substack{j=-\frac{N}{2}+1\\(j-i)\mbox{ odd}}}^{\frac{N}{2}}
\frac{\OVL{\omega}_j^p S_h \tau_j}{\OVL{\tau_j}-\OVL{\tau_i}}\mbox{,}
\label{UI3_disc_est} \\
U_{3,i}&=-\frac{h}{\pi}
\DPS\sum_{\substack{j=-\frac{N}{2}+1\\(j-i)\mbox{ odd}}}^{\frac{N}{2}}
\frac{\OVL{\omega}_j^p (\tau_j-\tau_i) \OVL{S_h\tau_j}}{(\OVL{\tau_j}-\OVL{\tau_i})^2}\mbox{,}
\label{UI4_disc_est}\\
U_{4,i}&=(Q+iB)\tau_i-\frac{iG\tau_i}{2}\mbox{,}
\label{UI5_disc_est}
\end{align}
for $i=-\frac{N}{2}+1,\cdots,\frac{N}{2}$.
%Note that $\MCAL{H}_h(\omega_i)$ has been substituted for the sum with cotangent kernel in (\ref{velocity_eqn_disc}).
% TODO: reference eqn 66?

Variations of the $U_{l,i},l=1,2,\cdots,4$ are calculated using Lemmas \ref{product_rule_for_dots_lemma} and \ref{quotient_rule_1_over_f_lemma}.
We represent these variations as the sum of linear and nonlinear quantities in the variation, so that
\begin{align}\label{dot_U_1i_L}
\dot{U}_{1,i}&=
%\MCAL{H}_h\dot{\omega}_i
-\frac{h}{\pi}
\DPS\sum_{\substack{j=-\frac{N}{2}+1\\(j-i)\mbox{ odd}}}^{\frac{N}{2}}
\dot{\omega}_j^p
\biggpa{
2\mbox{Re}\Bigpa{\frac{S_h \tau_h(\alpha_j)}{\tau_h(\alpha_j)-\tau_h(\alpha_i)}}
+\cot\Bigpa{\frac{\alpha_i-\alpha_j}{2}}
}\nonumber \\
&-\frac{h}{\pi}
\DPS\sum_{\substack{j=-\frac{N}{2}+1\\(j-i)\mbox{ odd}}}^{\frac{N}{2}}
\biggpa{\omega_h^p(\alpha_j)
2\mbox{Re}\Bigpa{\frac{S_h \dot{\tau}_j}{\tau_h(\alpha_j)-\tau_h(\alpha_i)}
-\frac{S_h \tau_h(\alpha_j)}{(\tau_h(\alpha_j)-\tau_h(\alpha_i))^2}
(\dot{\tau}_j-\dot{\tau}_i)
}
}\nonumber \\
&+\dot{U}_{1,i}^{NL}\mbox{,}
\end{align}
where $ \dot{U}_{1,i}^{NL}$ represents the nonlinear terms in the variation. Expressions for the nonlinear terms are given in the appendix. 

Continuing,
\begin{align}\label{dot_U_2i_L}
\dot{U}_{2,i}&=\frac{h}{\pi}
\DPS\sum_{\substack{j=-\frac{N}{2}+1\\(j-i)\mbox{ odd}}}^{\frac{N}{2}}
\Biggmpa{
\frac{\OVL{\dot{\omega}_j^p} S_h\tau_h(\alpha_j)+\OVL{\omega_h^p}(\alpha_j)S_h\dot{\tau}_j}{\OVL{\tau_h}(\alpha_j)
-\OVL{\tau_h}(\alpha_i)}
-\bigpa{\OVL{\dot{\tau}_j}-\OVL{\dot{\tau}_i}} \frac{\OVL{\omega^p_h}(\alpha_j) S_h\tau_h(\alpha_j)}{[\OVL{\tau_h}(\alpha_j)-\OVL{\tau_h}(\alpha_i)]^2}
}\nonumber \\
&+\dot{U}_{2,i}^{NL}\mbox{.}
\end{align}
and
\begin{align}
\dot{U}_{3,i}&=-\frac{h}{\pi} \label{U3_dot}
\DPS\sum_{\substack{j=-\frac{N}{2}+1\\(j-i)\mbox{ odd}}}^{\frac{N}{2}}
\Biggmpa{
\Bigpa{
\OVL{\dot{\omega}_j^p}\bigmpa{\tau_h(\alpha_j)-\tau_h(\alpha_i)}\OVL{S_h\tau_h}(\alpha_j)
+\OVL{\omega_h^p}(\alpha_j)\bigmpa{\dot{\tau}_j-\dot{\tau}_i}\OVL{S_h\tau_h}(\alpha_j)\nonumber\\
&+\OVL{\omega_h^p}(\alpha_j)\bigmpa{\tau_h(\alpha_j)-\tau_h(\alpha_i)}\OVL{S_h\dot{\tau}_j}
}\Bigpa{
\frac{1}{\OVL{\tau_h}(\alpha_j)-\OVL{\tau_h}(\alpha_i)}
}^2
\nonumber \\
&-\frac{
2\OVL{\omega_h^p}(\alpha_j)\bigmpa{\tau_h(\alpha_j)-\tau_h(\alpha_i)}\OVL{S_h\tau_h}(\alpha_j)\bigmpa{\OVL{\dot{\tau}_j}-\OVL{\dot{\tau}_i}}
}{\bigmpa{\OVL{\tau_h}(\alpha_j)-\OVL{\tau_h}(\alpha_i)}^3}
}
\nonumber \\
&+\dot{U}_{3,i}^{NL}\mbox{.}
\end{align}
Finally,
\begin{align}\label{UI5_disc}
\dot{U}_{4,i}=(Q+iB)\OVL{\dot{\tau}_i}-\frac{iG\dot{\tau}_i}{2}\mbox{.}
\end{align}

In the next section, we compute the leading order contribution to the variation in the velocity. This computation uses the following estimates.
\begin{lemma}\label{I_h_lemma}
Let $\dot{q}\in l^2$ be a variation of some quantity, and let
\begin{align}
I_h \dot{q}_i=\frac{h}{\pi}
\DPS\sum_{\substack{j=-\frac{N}{2}+1\\(j-i){\rm\ odd}}}^{\frac{N}{2}}
\frac{f(\alpha_j)\dot{q}_j}{\tau(\alpha_j)-\tau(\alpha_i)}\mbox{,}
\end{align}
with $f(\alpha)$ and $\tau(\alpha)$ smooth, and $\tau_{\alpha}(\alpha)\neq 0$.
Define
\begin{align}\label{K_1h_f_zeta_decomp}
K_{1h}[f,\tau](\dot{q}_i)
&=\frac{h}{\pi}\DPS\sum_{\substack{j=-\frac{N}{2}+1\\(j-i){\rm\ odd}}}^{\frac{N}{2}}
\dot{q}_j
\Bigmpa{\frac{f(\alpha_j)}{\tau(\alpha_j)-\tau(\alpha_i)}
-\frac{f(\alpha_i)}{2\tau_{\alpha}(\alpha_i)}\cot\Bigpa{\frac{\alpha_j-\alpha_i}{2}}
}\mbox{.}
\end{align}
Then
\begin{align}\label{I_h_decomp}
I_h \dot{q}_i =
-\frac{f(\alpha_i)}{2\tau_{\alpha}(\alpha_i)}\MCAL{H}_h \dot{q}_i
+K_{1h}[f,\tau](\dot{q}_i)\mbox{,}
\end{align}
and $K_{1h}[f,\tau](\dot{q}_i^p)=A_{-2}(\dot{q}_i)$, when filtering is applied.
\end{lemma}
\begin{proof}
The decomposition in (\ref{I_h_decomp}) follows from adding and substracting $-\frac{f(\alpha_i)}{2\tau_{\alpha}(\alpha_i)}\MCAL{H}_h q_i$.
The fact that $K_{1h}$ is a smoothing operator on $\dot{q}^p$ then follows by noting that the kernel within brackets in $K_{1h}$ is smooth, and applying Lemma \ref{R_h_lemma}.
\end{proof}
\begin{lemma}\label{J_h_lemma}
Under the same assumptions as in Lemma \ref{I_h_lemma}, let
\begin{align}
J_h \dot{\tau}_i
&=\frac{h}{\pi}\DPS\sum_{\substack{j=-\frac{N}{2}+1\\(j-i){\rm\ odd}}}^{\frac{N}{2}}
\frac{f(\alpha_j)\bigpa{\dot{\tau}_j-\dot{\tau}_i}}{\bigmpa{\tau(\alpha_j)-\tau(\alpha_i)}^2}\mbox{,}
\end{align}
and define
\begin{align}\label{K_2h_f_zeta_defn}
&K_{2h}[f,\tau](\dot{\tau}_i)
=\frac{h}{\pi}\DPS\sum_{\substack{j=-\frac{N}{2}+1\\(j-i){\rm\ odd}}}^{\frac{N}{2}}
\bigpa{\dot{\tau}_j-\dot{\tau}_i}\Biggmpa{
\frac{f(\alpha_j)}{\bigmpa{\tau(\alpha_j)-\tau(\alpha_i)}^2}
-\frac{f(\alpha_i)}{4\tau_{\alpha}^2(\alpha_i)\sin^2\bigpa{\frac{\alpha_j-\alpha_i}{2}}}\nonumber \\
&-\frac{f_\alpha(\alpha_i) \tau_\alpha(\alpha_i) - f(\alpha_i) \tau_{\alpha \alpha}(\alpha_i)}{2\tau_{\alpha}^3(\alpha_i)}
\cot\Bigpa{\frac{\alpha_j-\alpha_i}{2}}
}\mbox{.}
\end{align}
Then
\begin{align}\label{J_h_zeta_err}
J_h \dot{\tau}_i
&=-\frac{f(\alpha_i)}{2\tau_{\alpha}^2(\alpha_i)}\MCAL{H}_h(S_h \dot{\tau}_i)
-\frac{f_\alpha(\alpha_i) \tau_\alpha(\alpha_i) - f(\alpha_i) \tau_{\alpha \alpha}(\alpha_i)}{2\tau_{\alpha}^3(\alpha_i)}
\MCAL{H}_h \dot{\tau}_i
\nonumber \\
&+K_{2h}[f,\tau](\dot{\tau}_i)\mbox{,}
\end{align}
and
\begin{align}\label{K_2h_commut_zeta_filter_err}
K_{2h}[f,\tau](\dot{\tau}_i^p)
=A_{-2}(\dot{\tau}_i)\mbox{.}
\end{align}
\end{lemma}
\begin{proof}
Write
\begin{align}\label{UI1_third_term_est}
\frac{h}{\pi}&\DPS\sum_{\substack{j=-\frac{N}{2}+1\\(j-i){\rm\ odd}}}^{\frac{N}{2}}
\frac{f(\alpha_j)\bigpa{\dot{\tau}_j-\dot{\tau}_i}}{\bigmpa{\tau(\alpha_j)-\tau(\alpha_i)}^2}
=\frac{h}{\pi}\Biggmpa{
\frac{f(\alpha_i)}{4\tau_{\alpha}^2(\alpha_i)}
\DPS\sum_{\substack{j=-\frac{N}{2}+1\\(j-i){\rm\ odd}}}^{\frac{N}{2}}
\frac{\dot{\tau}_j-\dot{\tau}_i}{\sin^2\bigpa{\frac{\alpha_j-\alpha_i}{2}}}
\nonumber\\
&+\frac{f_\alpha(\alpha_i) \tau_\alpha(\alpha_i) - f(\alpha_i) \tau_{\alpha \alpha}(\alpha_i)}{2\tau_{\alpha}^3(\alpha_i)}
\DPS\sum_{\substack{j=-\frac{N}{2}+1\\(j-i){\rm\ odd}}}^{\frac{N}{2}}
\bigpa{\dot{\tau}_j-\dot{\tau}_i}
\cot\Bigpa{\frac{\alpha_j-\alpha_i}{2}}
}
\nonumber\\
&+K_2[f,\tau](\dot{\tau}^p)
\end{align}
by adding and subtracting $\frac{h}{\pi}$ times the quantity in brackets.
To make use of the formula (\ref{H_h_disc_real_line_prop_3}) for the derivative of the Hilbert transform,  apply the identity \cite{CKP:2005}
\begin{align}
\frac{1}{\sin^2 \frac{z}{2}}
=4\DPS\sum_{n=-\infty}^{\infty}
\frac{1}{(z-2n\pi)^2}
\end{align}
to transform the representation of the first sum within brackets in (\ref{UI1_third_term_est}) from a periodic to an infinite domain, 
\begin{align}\label{zeta_over_sine_square_complex_expansion}
\DPS\sum_{\substack{j=-\frac{N}{2}+1\\(j-i){\rm\ odd}}}^{\frac{N}{2}}
\frac{\dot{\tau}_j-\dot{\tau}_i}{\sin^2\bigpa{\frac{\alpha_j-\alpha_i}{2}}}
=4 \DPS\sum_{(j-i){\rm\ odd}}
\frac{\dot{\tau}_j-\dot{\tau}_i}{\bigpa{\alpha_j-\alpha_i}^2}\mbox{,}
\end{align}
where we have used the periodicity of $\dot{\tau}_j$.
%Substitute (\ref{zeta_over_sine_square_complex_expansion}) for the first sum within the brackets in (\ref{UI1_third_term_est}), and 
Identify
\begin{align}
\MCAL{H}_h(S_h \dot{\tau}_i)
&=\frac{2h}{\pi}\DPS\sum_{(j-i){\rm\ odd}}
\frac{\dot{\tau}_i-\dot{\tau}_j}{(\alpha_i-\alpha_j)^2}\mbox{,}
\end{align}
and
\begin{align}\label{Disc_Hilbert_Transform_zeta_dot_eqn_208}
\MCAL{H}_h \dot{\tau}_i
=-\frac{h}{\pi}
\DPS\sum_{\substack{j=-\frac{N}{2}+1\\(j-i){\rm\ odd}}}^{\frac{N}{2}}
\bigpa{\dot{\tau}_j-\dot{\tau}_i}
\cot\Bigpa{\frac{\alpha_j-\alpha_i}{2}}\mbox{}
\end{align}
to obtain (\ref{J_h_zeta_err}) (also using that the discrete Hilbert transform of a constant is zero).
%In (\ref{Disc_Hilbert_Transform_zeta_dot_eqn_208}) we have used the identity
%\begin{align}
%\DPS\sum_{\substack{j=-\frac{N}{2}+1\\(j-i){\rm\ odd}}}^{\frac{N}{2}}
%f_i\cot\Bigpa{\frac{\alpha_j-\alpha_i}{2}}=0.
%\end{align}
%which holds for a generic $f$.
Finally, (\ref{K_2h_commut_zeta_filter_err}) follows from the observation that the quantity within brackets in the definition of $K_{2h}$, 
%in (\ref{K_2h_f_zeta_defn}), 
namely, 
\begin{align}
h(\alpha,\alpha')&=\frac{f(\alpha')}{[\tau(\alpha')-\tau(\alpha)]^2}
-\frac{f(\alpha)}{4\tau_{\alpha}^2(\alpha)\sin^2\bigpa{\frac{\alpha'-\alpha}{2}}}
\nonumber\\
&-\frac{f_\alpha(\alpha_i) \tau_\alpha(\alpha_i) - f(\alpha_i) \tau_{\alpha \alpha}(\alpha_i)}{2\tau_{\alpha}^3(\alpha_i)} \cot\Bigpa{\frac{\alpha'-\alpha}{2}}
\end{align}
is a smooth function of $\alpha$ and $\alpha'$.
\end{proof}

The following lemmas are derived similarly and are presented without proof.
\begin{lemma}\label{L_h_lemma}
Under the same assumptions as Lemma \ref{I_h_lemma}, let
\begin{align}
L_h \dot{q}_i
=\frac{h}{\pi}
\DPS\sum_{\substack{j=-\frac{N}{2}+1\\(j-i){\rm\ odd}}}^{\frac{N}{2}}
\frac{f(\alpha_j)\bigpa{\tau(\alpha_j)-\tau(\alpha_i)}\dot{q}_j}{
\bigmpa{\OVL{\tau(\alpha_j)}-\OVL{\tau(\alpha_i)}}^2
}\mbox{,}
\end{align}
and define
\begin{align}
&K_{3h}[f,\tau](\dot{q}_i)
=\frac{h}{\pi}\DPS\sum_{\substack{j=-\frac{N}{2}+1\\(j-i){\rm\ odd}}}^{\frac{N}{2}}
\dot{q}_j\Biggmpa{
\frac{\bigpa{\tau(\alpha_j)-\tau(\alpha_i)}f(\alpha_j)}{
\bigmpa{\OVL{\tau(\alpha_j)}-\OVL{\tau(\alpha_i)}}^2
}
+\frac{f(\alpha_i)\tau_{\alpha}(\alpha_i)}{2\OVL{\tau_\alpha^2(\alpha_i)}}
\cot\Bigpa{\frac{\alpha_i-\alpha_j}{2}}
}\mbox{.}
\end{align}
Then
\begin{align}
L_h \dot{q}_i
=-\frac{f(\alpha_i)\tau_{\alpha}(\alpha_i)}{2\OVL{\tau_\alpha^2(\alpha_i)}}
\MCAL{H}_h\dot{q}_i+K_{3h}[f,\tau](\dot{q}_i)\mbox{,}
\end{align}
and $K_{3h}[f,\tau](\dot{q}_i^p)=A_{-2}(\dot{q}_i)$.
\end{lemma}
\begin{lemma}\label{M_h_lemma}
Under the same assumption as Lemma \ref{I_h_lemma}, let
\begin{align}
M_h \OVL{\dot{\tau}}
=\frac{h}{\pi}
\DPS\sum_{\substack{j=-\frac{N}{2}+1\\(j-i){\rm\ odd}}}^{\frac{N}{2}}
\frac{f(\alpha_j)\bigpa{\tau(\alpha_j)-\tau(\alpha_i)}
\bigpa{\OVL{\dot{\tau}_j}-\OVL{\dot{\tau}_i}}}{
\bigmpa{\OVL{\tau(\alpha_j)}-\OVL{\tau(\alpha_i)}}^3
}\mbox{,}
\end{align}
and define
\begin{align}
&K_{4h}[f,\tau](\OVL{\dot{\tau}_j})
=\frac{h}{\pi}\DPS\sum_{\substack{j=-\frac{N}{2}+1\\(j-i){\rm\ odd}}}^{\frac{N}{2}}
\bigpa{\OVL{\dot{\tau}_j}-\OVL{\dot{\tau}_i}}\Biggmpa{
\frac{\bigpa{\tau(\alpha_j)-\tau(\alpha_i)}f(\alpha_j)}{
\bigmpa{\OVL{\tau(\alpha_j)}-\OVL{\tau(\alpha_i)}}^2
}
\nonumber\\
&-\frac{f(\alpha_i)\tau_{\alpha}(\alpha_i)}{4\OVL{\tau_{\alpha}^3(\alpha_j)}\sin^2\bigpa{\frac{\alpha_j-\alpha_i}{2}}}
-\frac{p[f](\alpha_i)}{2\OVL{\tau_{\alpha}^3(\alpha_i)}}
\cot\Bigpa{\frac{\alpha_j-\alpha_i}{2}}
}\mbox{,}
\end{align}
where
\begin{align}
p[f](\alpha_i)=\frac{\tau_{\alpha}(\alpha_i)f(\alpha_i)}{2}
\Bigpa{
\frac{\tau_{\alpha\alpha}(\alpha_i)}{\tau_{\alpha}(\alpha_i)}
-\frac{3\OVL{\tau_{\alpha\alpha}(\alpha_i)}}{\OVL{\tau_{\alpha}(\alpha_i)}}
}
+\tau_{\alpha}(\alpha_i)f_{\alpha}(\alpha_i)\mbox{.}
\end{align}
Then
\begin{align}
&M_h \OVL{\dot{\tau}_i}
=-\frac{f(\alpha_i)\tau_{\alpha}(\alpha_i)}{2\OVL{\tau_{\alpha}^3(\alpha_i)}}
\MCAL{H}_h (\OVL{S_h \dot{\tau}_i})
-\frac{p[f](\alpha_i)}{2\OVL{\tau_{\alpha}^3(\alpha_i)}}
\MCAL{H}_h \OVL{\dot{\tau}_i}
+K_{4h}[f,\tau](\OVL{\dot{\tau}_i})\mbox{,}
\end{align}
and $K_{4h}[f,\tau](\OVL{\dot{\tau}_i^p})=A_{-2}(\dot{\tau}_i)$.
\end{lemma}

\begin{remark}
If filtering is not applied, then it is easy to see that $K_{1h}[\cdot,\cdot](\dot{q})=A_0(\dot{q})$ and
$K_{3h}[\cdot,\cdot](\dot{q})=A_0(\dot{q})$. Similarly, 
$K_{2h}[\cdot,\cdot](\dot{\tau})=A_0(\dot{\tau})$ and $K_{4h}[\cdot,\cdot](\overline{\dot{\tau}})=A_0(\dot{\tau})$.
These estimates are a consequence of Remark  \ref{R_h_A_0_no_filt_equiv_remark}.
\end{remark}

\subsection{Leading order velocity variations} \label{sec:vel_variations}
We identify the most singular terms in the variation of the complex velocity, $\dot{u}_i=\MCAL{H}_h{\dot{\omega}}+ \sum_{l=1}^4 \dot{U}_{l,i} $.
% Begin insert 25
First, note that in (\ref{dot_U_1i_L})-(\ref{U3_dot}) we can replace $\omega_h$ with $\omega$,  $\tau_h$ with $\tau$, and $S_h\tau_h$ with $\tau_{\alpha}$, incurring by consistency a high-order or  $\mbox{O}(h^s)$ error.
%(Recall that the expression $\mbox{O}(h^s)$ is used to denote a generic high-order discretization error).
Now, consider   $\dot{U}_1$ in (\ref{dot_U_1i_L}), and 
denote the first sum  (with the above replacements)  by $K_{0h}(\dot{\omega}_i^p)$, so that 
\begin{align}
K_{0h}(\dot{\omega}_j^p)=\frac{h}{\pi}
\DPS\sum_{\substack{j=-\frac{N}{2}+1\\(j-i)\mbox{ odd}}}^{\frac{N}{2}}
\dot{\omega}_j^p
\biggpa{
2\mbox{Re}\Bigpa{\frac{ \tau_\alpha(\alpha_j)}{\tau(\alpha_j)-\tau(\alpha_i)}}
+\cot\Bigpa{\frac{\alpha_i-\alpha_j}{2}}
}.
\end{align}
Next apply Lemmas \ref{I_h_lemma} and \ref{J_h_lemma} to find
\begin{align}\label{UI1_disc_HK_decomp}
\dot{U}_{1,i}&=
%\MCAL{H}_h \dot{\omega}_i
-K_{0h}(\dot{\omega}_i^p)
-K_{1h}[\omega^p,\tau](S_h\dot{\tau}_i)
-K_{1h}[\omega^p,\OVL{\tau}](\OVL{S_h\dot{\tau}_i})
\nonumber\\
&-\omega_{\alpha}^p(\alpha_i)
\Bigpa{\frac{\MCAL{H}_h\dot{\tau}_i}{2\tau_{\alpha}(\alpha_i)}
+\mbox{c.c.}
}
+K_{2h}[\omega^p\tau_{\alpha},\tau](\dot{\tau}_i)
+K_{2h}[\omega^p\OVL{\tau_{\alpha}},\OVL{\tau}](\OVL{\dot{\tau}_i})
\nonumber\\
&+\dot{U}_{1,i}^{NL}+\mbox{O}(h^s)\mbox{,}
\end{align}
where c.c. denotes the complex conjugate of the previous term.
The leading order contribution to $\dot{U}_{2,i}$ in (\ref{dot_U_2i_L}) is determined from Lemmas \ref{I_h_lemma} and \ref{J_h_lemma} as
\begin{align}\label{UI3_disc_HK_decomp}
\dot{U}_{2,i}&=-\frac{1}{2}\frac{\tau_{\alpha}(\alpha_i)}{\OVL{\tau_{\alpha}}(\alpha_i)}\MCAL{H}_h \OVL{\dot{\omega}_i^p}
+K_{1h}[\tau_{\alpha},\OVL{\tau}](\OVL{\dot{\omega}_i^p})
-\frac{\OVL{\omega^p}(\alpha_i)}{2\OVL{\tau_{\alpha}}(\alpha_i)}
\MCAL{H}_h(S_h \dot{\tau}_i)
+K_{1h}[\OVL{\omega^p},\OVL{\tau}](S_h \dot{\tau}_i)
\nonumber\\
&+\frac{\OVL{\omega^p}(\alpha_i)\tau_{\alpha}(\alpha_i)}{2\OVL{\tau}_{\alpha}^2(\alpha_i)}
\MCAL{H}_h(\OVL{S_h \dot{\tau}_i})
+\frac{1}{2\OVL{\tau}_{\alpha}^3(\alpha_i)}
\bigpa{
\OVL{\tau_\alpha}(\alpha_i)(\OVL{\omega}^p \tau_{\alpha})_{\alpha}(\alpha_i)
-\OVL{\omega}^p(\alpha_i)\tau_{\alpha}(\alpha_i)\OVL{\tau_{\alpha\alpha}}(\alpha_i)
}
\MCAL{H}_h(\OVL{\dot{\tau}_i})
\nonumber\\
&-K_{2h}[\OVL{\omega}^p\tau_{\alpha},\OVL{\tau}](\OVL{\dot{\tau}_i})
+\dot{U}_{2,i}^{NL}+\mbox{O}(h^s)\mbox{.}
\end{align}
Similarly, we find from (\ref{U3_dot}) and Lemmas \ref{J_h_lemma}-\ref{M_h_lemma} that
\begin{align}\label{UI4_disc_HK_decomp}
\dot{U}_{3,i}&=
\frac{\tau_{\alpha}(\alpha_i)}{2\OVL{\tau_{\alpha}}(\alpha_i)}
\MCAL{H}_h \OVL{\dot{\omega}_i^p}
-K_{3h}[\OVL{\tau_{\alpha}},\tau](\OVL{\dot{\omega}_i^p})
+\frac{\OVL{\omega^p}(\alpha_i)}{2\OVL{\tau_{\alpha}}(\alpha_i)}
\MCAL{H}_h(S_h \dot{\tau}_i)
+\frac{\OVL{\omega^p_\alpha}(\alpha_i)}{2\OVL{\tau_{\alpha}}(\alpha_i)}
\MCAL{H}_h \dot{\tau}_i
\nonumber\\
&-K_{2h}[\OVL{\omega^p\tau_{\alpha}},\OVL{\tau}](\dot{\tau}_i)
-\frac{\OVL{\omega^p}(\alpha_i)\tau_{\alpha}(\alpha_i)}{2\OVL{\tau_{\alpha}^2}(\alpha_i)}
\MCAL{H}_h(\OVL{S_h \dot{\tau}_i})
-K_{3h}[\OVL{\omega^p},\tau](\OVL{S_h \dot{\tau}_i})
\nonumber\\
&
-\frac{p[2 \OVL{\omega^p\tau_{\alpha}}](\alpha_i)}{2\OVL{\tau}_{\alpha}^3(\alpha_i)}
\MCAL{H}_h(\OVL{\dot{\tau}_i})
+K_{4h}[2\OVL{\omega^p\tau_{\alpha}},\tau](\OVL{\dot{\tau}_i})
+\dot{U}_{3,i}^{NL}+\mbox{O}(h^s)\mbox{,}
\end{align}
where
\begin{align}
p[2 \OVL{\omega^p\tau_{\alpha}}](\alpha_i)&=\OVL{\omega}^p(\alpha_i)
\Bigpa{
\tau_{\alpha\alpha}(\alpha_i)
\OVL{\tau_{\alpha}}(\alpha_i)
-3\OVL{\tau_{\alpha\alpha}}(\alpha_i)
\tau_{\alpha}(\alpha_i)
}
+2\tau_{\alpha}(\alpha_i)
(\OVL{\omega\tau_{\alpha}})_{\alpha}(\alpha_i)\mbox{.}
\end{align}
%Finally, $\dot{U}_{4,i}$ is easily seen to consist only of low order terms, i.e., involving smoothing operators.
%A more precise statement is given in the summary below.

A simplified representation of the leading order velocity variations is
provided in Section \ref{sec:summary_variation}. Before presenting this, we consider the variation of $\omega$.

\subsection{Leading order variation of $\omega$} \label{sec:variation_omega}
For the case of viscosity matched fluids, in which $\beta=0$ and $\chi=1/2$, the %integral term in the continuous equation for $\omega$ (or equivalently the alternate point trapezoidal rule sum in the discrete equation) drops out, and the 
equation for $\omega$ localizes.
Taking the variation of  (\ref{omega_til}) using (\ref{g_i_def}) with $\tilde{\omega}=0$, it is easy to see that
\begin{align}\label{omega_disc_err_decomp}
&\dot{\omega}_i
=-\frac{1}{4}\bigpa{
\dot{\MCAL{S}}_i e^{i\theta(\alpha_i)}
+\MCAL{S}(\alpha_i)(e^{i\theta})_i^{\cdot}
}\nonumber\\
&+\frac{\kappa_B}{4 s_{\alpha}^2}\bigmpa{
i\theta_{\alpha\alpha}(\alpha_i)(e^{i\theta})_i^{\cdot}
+i e^{i\theta(\alpha_i)}S_h^2\dot{\theta}_i
}\nonumber\\
&+\frac{\kappa_B}{2 s_{\alpha}}
i e^{i\theta(\alpha_i)}\theta_{\alpha\alpha}(\alpha_i)
\bigpa{\frac{1}{\sigma}}_i^{\cdot} +\dot{\omega}_i^{NL}\mbox{,}
\end{align}
where $\dot{\omega}_i^{NL}$ contains nonlinear terms or products of the variations on the right hand side.
%$\bigpa{\frac{1}{\sigma}}_i^{\cdot}$, $(e^{i\theta})_i^{\cdot}$, $\dot{\MCAL{S}}_i$, and discrete derivatives $S_h\dot{\theta}$ and $S_h^2\dot{\theta}$.

We now give estimates for each of the terms in (\ref{omega_disc_err_decomp}).
Lemma \ref{tan_vector_est_lemma}  provides the estimate (taking $\sigma=s_\alpha=1$),
\begin{align}\label{Exp_i_theta_err}
(e^{i\theta})_i^{\cdot}
=i e^{i\theta(\alpha_i)} \dot{\theta}_i+A_{-3}(\dot{\theta}) = O(h^3)\mbox{.}
\end{align}
%The variation of surface tension $\dot{\MCAL{S}}$ is given in (\ref{num_approx_tension_cal_S}).
It is easily seen from (\ref{tension_modulus_disc_16_1})
% TODO: change ref from (84) to (83)
and Lemmas \ref{product_rule_for_dots_lemma} and \ref{quotient_rule_1_over_f_lemma}, that
\begin{align}\label{dot_tension_eqn_47a1}
\dot{\MCAL{S}}_i=\frac{1+\MCAL{S}_0(\alpha_i)}{s_{0\alpha}\alpha_{0\alpha}(\alpha_i)}
\dot{\sigma}
-\frac{s_{\alpha}(1+\MCAL{S}_0(\alpha_i))}{s_{0\alpha}\alpha_{0\alpha}^2(\alpha_i)}
(D_h\dot{\alpha_0})_i+\dot{\MCAL{S}}_i^{NL}+\mbox{O}(h^s)\mbox{,}
\end{align}
where we have assumed $\dot{\MCAL{S}}_0=0$, i.e., the initial tension $\MCAL{S}_{0i}$ is exactly $\MCAL{S}_0(\alpha_i)$, and similarly $\dot{\sigma}_0=0$.
An estimate for $\dot{\MCAL{S}}_i^{NL}$ is readily obtained as
\begin{align} \label{dot_S_NL}
\dot{\MCAL{S}}_i^{NL}=
%\mbox{O}(\dot{\sigma}^2+(D_h\dot{\alpha_0})^2)\mbox{,}\nonumber\\
%&=h^3 A_{-s}(\dot{\sigma})+h^2 A_0(D_h\dot{\alpha_0})\mbox{,}\nonumber\\
A_{-s}(\dot{\sigma})+A_{-1}(\dot{\alpha_0})\mbox{,}
\end{align}
using $\norm{\dot{\sigma}}_{\infty}\leq h^3$ and $\norm{D_h(\dot{\alpha_0})}_{\infty}\leq ch^2$.
It follows that (\ref{dot_tension_eqn_47a1}) can be written as
\begin{align}\label{num_approx_tension_cal_S}
\dot{\MCAL{S}}_i=A_{-s}(\dot{\sigma}_i)-\tilde{f}(\alpha_i)(D_h\dot{\alpha}_{0})_i+A_{-1}(\dot{\alpha}_{0i})+\mbox{O}(h^s)\mbox{,}
\end{align}
where
\begin{align}\label{tilde_f_alpha_eqn_251}
\tilde{f}(\alpha)=\frac{s_{\alpha}(1+\MCAL{S}_0(\alpha))}{s_{0\alpha}\alpha_{0\alpha}^2(\alpha)}>0
\end{align}
is a smooth, real and positive function.
% Begin insert 1
The positivity of $\tilde{f}(\alpha)$ will be seen to be critical.
Indeed, it is found to be necessary for the well-posedness of the continuous equations.
% End insert 1
We note that $\tilde{f}(\alpha_i)(D_h\dot{\alpha}_0)_i=A_0(D_h\dot{\alpha}_{0i})$, but to make energy estimates we retain the specific form in (\ref{num_approx_tension_cal_S}).

An estimate on the nonlinear term,
\begin{align}\label{omega_disc_err_dot_NL}
\dot{\omega}_i^{NL}
=A_{-1}(\dot{\theta}_i)+A_{-s}(\dot{\sigma}_i)+A_{-2}(\dot{\alpha}_{0i})
%+h^2 A_0( \dot{\tau}_c)
+\mbox{O}(h^s),
\end{align}
is derived at the end of this section.
We also need the following estimate for $\Bigpa{\frac{1}{\sigma}}^{\cdot}$:
%\begin{lemma}\label{one_over_sigma_err_dot_lemma}
%Assume that $\norm{\dot{\sigma}}_{l^2}\leq h^{\frac{7}{2}}$ and $s_{\alpha}>a>0$ for some positive constant $a$. Then
\begin{align}\label{one_over_sigma_err_dot}
\Bigpa{\frac{1}{\sigma}}^{\cdot}
&=-\frac{\dot{\sigma}}{s_{\alpha}^2}
+\frac{\dot{\sigma}^2}{s_{\alpha}^2\bigpa{s_{\alpha}+\dot{\sigma}}}=O(h^3)\mbox{,}
\end{align}
%\end{lemma}
which follows from from Lemma \ref{quotient_rule_1_over_f_lemma},
the boundedness $s_{\alpha}$ away from zero, and $\norm{\dot{\sigma}}_{\infty}<h^3$.
This also implies that 
\begin{align}\label{one_over_sigma_err_dot_rel}
f(\alpha_i)\Bigpa{\frac{1}{\sigma}}^{\cdot}
&=A_{-s}(\dot{\sigma}_i)\mbox{ for smooth }f\mbox{.}
\end{align}

Together, (\ref{omega_disc_err_decomp}), (\ref{Exp_i_theta_err}), (\ref{dot_tension_eqn_47a1}), (\ref{omega_disc_err_dot_NL}) and the above remarks tell us that
\begin{align}\label{dot_omega_err_decomp}
\dot{\omega}_i&=\kappa_B \frac{i e^{i\theta(\alpha_i)}}{4 s_{\alpha}^2}S_h^2\dot{\theta}_i
+A_0(\dot{\theta}_i)+A_{-s}(\dot{\sigma}_i) 
+\frac{1}{4}e^{i\theta(\alpha_i)}\tilde{f}(\alpha_i)(D_h\dot{\alpha_0})_i \nonumber \\
&+A_{-1}(\dot{\alpha}_{0i})
%+h^2 A_0(\dot{\tau}_c)
+\mbox{O}(h^s)\mbox{,}
\end{align}
where $\tilde{f}(\alpha)>0$ is given by (\ref{tilde_f_alpha_eqn_251}). The filtered quantity $\dot{\omega}^p_i$ satisfies the same estimate except $S_h^2 \theta_i$ in the first term is replaced by $D_h^2 \theta_i$, 
%per the remarks following  (\ref{g_i_def}) and (\ref{vel_kernel2}), 
i.e.,
\begin{align}\label{dot_omega^p_err_decomp}
\dot{\omega}^p_i&=\kappa_B \frac{i e^{i\theta(\alpha_i)}}{4 s_{\alpha}^2}D_h^2\dot{\theta}_i
+A_0(\dot{\theta}_i)+A_{-s}(\dot{\sigma}_i) 
+\frac{1}{4}e^{i\theta(\alpha_i)}\tilde{f}(\alpha_i)(D_h\dot{\alpha_0})_i \nonumber \\
&+A_{-1}(\dot{\alpha}_{0i})
%+h^2 A_0(\dot{\tau}_c)
+\mbox{O}(h^s)\mbox{.}
\end{align}

%This further implies that
%\begin{align}\label{dot_omega_err_decomp_A_268}
%A_{-2}(\dot{\omega})=A_0(\dot{\theta})+A_{-s}(\dot{\sigma})
%+A_{-1}(\dot{\alpha_0})
%+h^2 A_0(\dot{\tau_c})
%+\mbox{O}(h^s)\mbox{,}
%\end{align}
%and therefore $A_{-2}(\dot{\omega}^p)$ can be replaced with $A_{-1}(\dot{\alpha_0})$ in
% TODO: reference eqns (228),(243),(245)?
%(\ref{UI_decomp}), (\ref{normal_vel_err_decomp}), (\ref{dot_u_s_expression_46b_1}), which leads to (\ref{dot_u_n_expression_46d_1}), (\ref{dot_u_s_expression_46d_2}).
% TODO: reference eqns (246),(247)
%We note that there are no $A_0(S_h\dot{\theta})$ terms in (\ref{dot_omega_err_decomp}), since these have canceled out.
%This will lead to an important simplification in the energy estimates developed later.

We complete the derivation of (\ref{dot_omega_err_decomp}) and (\ref{dot_omega^p_err_decomp}) by giving some details of the estimate (\ref{omega_disc_err_dot_NL}) for the nonlinear term $\dot{\omega}_i^{{NL}}$.
This term contains, for example, products of  $(e^{-i\theta})_i^{\cdot}$ and ${\Bigpa{\frac{1}{\sigma}}}^\cdot$, each of which is $O(h^3)$, with discrete derivatives of $\dot{\theta}$ and  $\dot{\alpha_{0i}}$.  It is easy to see that these products satisfy the estimate (\ref{omega_disc_err_dot_NL}).
This verifies (\ref{dot_omega_err_decomp}) which is the main result of this section.

\subsection{Summary of velocity variation} \label{sec:summary_variation}
The velocity variation is the sum of the contributions from $\MCAL{H}_h \dot \omega_i$ and  $(\dot{u}_R)_i=\sum_{l=1}^4 \dot{U}_{l,i}$. 
We anticipate the main contribution will come from $\MCAL{H}_h \dot{\omega}_i$.
The leading order $\MCAL{H}_h \OVL{\dot{\omega}_i^p}$ term in the sum  $\dot{U}_{2,i}+\dot{U}_{3,i}$ cancels out.
This is related to the smoothness of the kernel (\ref{Gdef}) in the  velocity equation. 
%This is related to the fact that the sum of the variation $\dot{U}_{2,i}+\dot{U}_{3,i}$ comes from the discretization of the integral
%\begin{align}
%\int_{\gamma}\omega(\tau,t)d\frac{\tau-\tau}{\OVL{\tau}-\OVL{\tau}}\mbox{,}
%\end{align}
%for which the kernel is smooth, i.e. the integrand is a smoothing operator on $\omega$.
The next order terms in $\dot{U}_{2,i}+\dot{U}_{3,i}$ containing $\MCAL{H}_h S_h\dot{\tau}_i$ and its conjugate also cancel out.
This has important consequences in the stability of the discrete equations in the drop evolution problem, i.e., with $\MCAL{S}=$ constant and $\kappa_B=0$.

To identify lower order terms in the velocity variation, we first apply 
Remarks  \ref{R_h_A_0_no_filt_equiv_remark} and \ref{S_h_dot_zeta_i_remark} to see   
\begin{align}\label{K_star_h_A_decomp_eqn_232}
K_{1 h}[\cdot,\cdot](S_h\dot{\tau}_i) ~~\mbox{and}~~K_{3 h}[\cdot,\cdot](S_h\dot{\tau}_i)
=A_0(\dot{\theta}_i)+A_{-s}(\dot{\sigma})\mbox{,}
\end{align}
and similarly using Lemma  \ref{zeta_dot_i_lemma},
\begin{align}\label{K_star_h_A_decomp_eqn_233}
K_{2 h}[\cdot,\cdot](\dot{\tau}_i) ~~\mbox{and}~~K_{4 h}[\cdot,\cdot](\dot{\tau}_i)
=A_{-1}(\dot{\theta}_i)+A_{-s}(\dot{\sigma}) 
%+ A_0(\dot{\tau}_c) 
\mbox{.}
\end{align}
Lemma \ref{R_h_lemma} implies that
$K_{0 h}[\cdot,\cdot](\dot{\omega}^p)
=A_{-2}(\dot{\omega})$ in view of the smoothness of the kernel, and similarly for
$K_{1 h}[\cdot,\cdot](\dot{\omega}_i^p)$ and $K_{3 h}[\cdot,\cdot](\dot{\omega}_i^p)$,
but this will not be sufficient to prove stability. Instead, we obtain a refined estimate by substituting for  $\dot{\omega}^p$  using (\ref{dot_omega^p_err_decomp}) and absorbing  smooth functions such as $e^{i \theta(\alpha_i)}$
into the kernels.
This gives
\begin{align} \label{K0h_omega_est}
K_{0 h}[\cdot,\cdot](\dot{\omega}^p) &= A_{-2}(D_h^2 \dot{\theta}_i)+A_{-2}(D_h \dot{\alpha}_{0i})
%+A_0(\dot{\theta}_i)
+A_{-s}(\dot{\sigma}_i) 
%+A_{-1}(\dot{\alpha}_{0i})
%+h^2 A_0(\dot{\tau}_c)
+\mbox{O}(h^s) \nonumber \\ 
&=A_0(\dot{\theta}_i)
+A_{-1}(\dot{\alpha}_{0i})+A_{-s}(\dot{\sigma}_i) 
%+h^2 A_0(\dot{\tau}_c)
+\mbox{O}(h^s),
\end{align}
and similarly for $K_{1 h}[\cdot,\cdot](\dot{\omega}^p)$ and $K_{3 h}[\cdot,\cdot](\dot{\omega}^p)$.

We also need to estimate $\dot{U}_4$ defined in (\ref{UI5_disc}).
From Lemma \ref{zeta_dot_i_lemma}, we immediately see that
\begin{align}\label{UI5_disc_decomp_A}
\dot{U}_{4,i}=A_{-1}(\dot{\theta}_i)+A_{-s}(\dot{\sigma})
+A_0(\dot{\tau}_c)\mbox{.}
\end{align}
Note that $U_4$ is the only velocity term in which $\tau_c(t)$ appears.
The other terms in our velocity decomposition, $U_1,\cdots,U_3$ depend only on the difference $\tau_j-\tau_i$, for which $\tau_c$ cancels out.
In the appendix, we show that the nonlinear term satisfies
\begin{align} \label{u_dot_nonlinear_term_est}
\dot{u}_i^{NL}
=\DPS\sum_{n=1}^4
\dot{U}_{n,i}^{NL}
=A_0(\dot{\theta}_i)+A_{-1}(\dot{\alpha}_{0i})
+A_{-s}(\dot{\sigma})
+\mbox{O}(h^s)\mbox{.}
\end{align}

In summary, the above remarks show that 
\begin{align}\label{UR_est}
(\dot{u}_R)_i
=\DPS\sum_{l=1}^4
\dot{U}_{l,i}
=A_0(\dot{\theta}_i)+A_{-1}(\dot{\alpha}_{0i})+A_{-s}(\dot{\sigma}_i)
+A_0(\dot{\tau}_c)
+\mbox{O}(h^s)\mbox{,}
\end{align}
which is the main result of this section.

\subsection{Tangential and normal velocity variations}
The discrete normal velocity is given by
\begin{align}\label{num_approx_disc_normal_vel}
(u_n)_i=\mbox{Re}\bigbra{u_i \OVL{n_i}}=\mbox{Im}\bigbra{u_i e^{-i\theta_i}}\mbox{,}
\end{align}
using $n_i=i e^{i\theta_i}$.
We need the variation $\dot{u_n}_i$.
From (\ref{num_approx_u_n_i_decomp_1}), this is
\begin{align}\label{num_approx_u_n_i_decomp_2}
(\dot{u_n})_i&=\frac{\kappa_B}{4 s_{\alpha}^2}\MCAL{H}_h(S_h^2\dot{\theta}_i)
+\frac{\kappa_B}{4}\Bigpa{\frac{1}{\sigma^2}}^{\cdot}
\MCAL{H}_h(\theta_{\alpha\alpha}(\alpha_i))
+\mbox{Im}\bigbra{
-\bigmpa{\MCAL{H}_h,e^{-i\theta(\alpha_i)}}(\dot{\omega}_i^p)\nonumber\\
-&\bigmpa{\MCAL{H}_h,(e^{-i\theta})_i^{\cdot}}(\omega^p(\alpha_i))
+(\dot{u_R})_i e^{-i\theta(\alpha_i)}+u_R(\alpha_i)(e^{-i\theta})_i^{\cdot}
}
+(\dot{u_n})_i^{NL}+\mbox{O}(h^s)
\mbox{,}
\end{align}
where $(\dot{u_n})_i^{NL}$ represents products of the variations on the right hand side.

It is straightforward to estimate each of the terms in (\ref{num_approx_u_n_i_decomp_2}).
Clearly,
\begin{align}
\frac{\kappa_B}{4}\Bigpa{\frac{1}{\sigma^2}}^{\cdot}
\MCAL{H}_h(\theta_{\alpha\alpha}(\alpha_i))
=A_{-s}(\dot{\sigma}_i)\mbox{,}
\end{align}
(cf. Lemma \ref{num_approx_S_h_f_dot_sigma_decomp_lemma}).
Using Lemma \ref{Commutator_thm_H_h_g} and the same arguments that lead to (\ref{K0h_omega_est}),
\begin{align}
\bigmpa{\MCAL{H}_h,e^{-i\theta(\alpha_i)}}(\dot{\omega}_i^p)
=A_0(\dot{\theta}_i)
+A_{-1}(\dot{\alpha}_{0i})+A_{-s}(\dot{\sigma}_i) 
%+h^2 A_0(\dot{\tau}_c)
+\mbox{O}(h^s).
\end{align}
It is also easy to see that  
\begin{align}\label{Disc_Hilbert_Transform_e_neg_i_theta_dot_eqn_238}
\bigmpa{\MCAL{H}_h,(e^{-i\theta})_i^{\cdot}}(\omega^p(\alpha_i))
=A_0(\dot{\theta})\mbox{.}
\end{align}
Indeed, we can write
\begin{align}
\bigmpa{\MCAL{H}_h,(e^{-i\theta})_i^{\cdot}}(\omega^p(\alpha_i))
&=\MCAL{H}_h\bigpa{(e^{-i\theta})_i^{\cdot}(\omega^p(\alpha_i))}
-\omega^p(\alpha_i)\MCAL{H}_h\bigpa{(e^{i\theta})_i^{\cdot}}
\nonumber\\
+\omega^p(\alpha_i)\MCAL{H}_h\bigpa{(e^{i\theta})_i^{\cdot}}
-&(e^{i\theta})_i^{\cdot} \MCAL{H}_h(\omega^p(\alpha_i))\mbox{,}
\end{align}
and note that the first two terms combine to form an integral operator with a smooth kernel on the (unfiltered) $(e^{i\theta})_i^{\cdot}$, while the latter two terms are clearly $A_0(\dot{\theta})$ functions.
% Begin insert 15
Remark \ref{R_h_A_0_no_filt_equiv_remark} then implies (\ref{Disc_Hilbert_Transform_e_neg_i_theta_dot_eqn_238}).
% End insert 15
An estimate for $(\dot{u}_R)_i$ is given in (\ref{UR_est}).
Finally, it is straightforward to show that 
%using the second equality in (\ref{Exp_neg_i_theta_err}) and the previous estimate that
%\begin{align}
%\norm{\dot{u_n}^{NL}}_{l^2}
%=h^3\bigpa{
%A_{-2}(\dot{\omega^p})+A_0(\dot{\theta})
%+A_{-s}(\dot{\sigma})+\mbox{O}(\dot{\tau}_c)
%+\mbox{O}(h^s)
%}\mbox{,}
%\end{align}
%which we write as
%\begin{align}
%(\dot{u_n})_i^{NL}=
%A_{-5}(\dot{\omega^p})+A_{-3}(\dot{\theta})
%+A_{-s}(\dot{\sigma})+\mbox{O}(h^3\dot{\tau}_c)
%+\mbox{O}(h^s)\mbox{.}
%\end{align}
the nonlinear term $(\dot{u}_n)^{NL}_i$ involves higher order smoothing operators.
The above arguments demonstrate that 
\begin{align}\label{dot_u_n_expression_46d_1}
(\dot{u_n})_i&=
\frac{\kappa_B}{4 s_{\alpha}^2}\MCAL{H}_h(S_h^2\dot{\theta}_i)
+A_{-1}(\dot{\alpha}_{0i})
+A_0(\dot{\theta}_i)+A_{-s}(\dot{\sigma})
+ A_0(\dot{\tau}_c)
+\mbox{O}(h^s)\mbox{.}
\end{align}

We will also need the variation in the tangential velocity $\dot{u}_s$, since this appears in the evolution equation for $\dot{\alpha}_0$.
We leave it to the reader to show, using the same arguments as for $\dot{(u_n)}_i$, that
\begin{align}\label{dot_u_s_expression_46b_1}
\dot{(u_s)}_i&=-\frac{1}{4}\MCAL{H}_h \bigpa{
\dot{\MCAL{S}}_i
}
+A_{-1}(\dot{\alpha}_{0i})
+A_0(\dot{\theta}_i)+A_{-s}(\dot{\sigma})
\nonumber\\
&+A_0(\dot{\tau}_c)
+\mbox{O}(h^s)\mbox{.}
\end{align}
The estimate in  (\ref{num_approx_tension_cal_S}) 
implies $\dot{\MCAL{S}}_i$  can be replaced by $-\tilde{f}(\alpha_i)(D_h\alpha_0)_i$,
where $\tilde{f}(\alpha)>0$ is defined in (\ref{tilde_f_alpha_eqn_251}).
%TODO: chage ref(255)
It follows  that
\begin{align}
\dot{(u_s)}_i&=\frac{1}{4}\MCAL{H}_h (\tilde{f}(\alpha_i)D_h\dot{\alpha_0}_i)
+A_0(\dot{\theta}_i)+A_{-s}(\dot{\sigma})
\nonumber\\
&+\mbox{O}(\dot{\tau}_c)
+\mbox{O}(h^s)\mbox{.}\label{dot_u_s_expression_46d_2}
\end{align}

The next quantity we need to estimate is the variation of $\phi_s$.
Taking the variation of (\ref{phi_s_disc}), we find
\begin{align}\label{dot_phi_s_i_error_A_decomp_1}
(\dot{\phi_s})_i&=S_h^{-1}\bigpa{
\dot{u_n} \theta_{\alpha}(\cdot)
+u_n(\cdot) S_h \dot{\theta}
-\ip{\dot{u_n}\theta_{\alpha}(\cdot)+u_n(\cdot)S_h\dot{\theta}}_h
}_i\nonumber\\
&+\mbox{O}(h^s)
+(\dot{\phi}_s^{NL})_i\mbox{,}
\end{align}
where
\begin{align}
(\dot{\phi}_s^{NL})_i=
S_h^{-1}\bigpa{
\dot{u_n} S_h\dot{\theta}
-\ip{\dot{u_n}S_h\dot{\theta}}_h
}\mbox{.}
\end{align}
%and recall the $\mbox{O}(h^s)$ term comes from, e.g., replacing $u_{nh}(\cdot)$ with $u_n(\cdot)$.
We readily obtain from Lemmas \ref{product_rule_f_S_h_g_lemma} and \ref{Int_product_rule_f_S_h_g_lemma} with (\ref{dot_u_n_expression_46d_1})
the estimate
\begin{align}\label{tangential_vel_err_decomp}
(\dot{\phi_s})_i&=
A_0(S_h\dot{\theta}_i)
+A_{-2}(\dot{\alpha_0}_i)
+A_{-s}(\dot{\sigma})
+A_0(\dot{\tau}_c)
+\mbox{O}(h^s)\mbox{.}
\end{align}
In obtaining this estimate, we have used $A_{-1}(S_h^2\dot{\theta})=A_0(S_h\dot{\theta})$, and the nonlinear terms 
%have been estimated using $\norm{S_h \dot{\theta}}_{\infty}\leq\frac{1}{h}\norm{\dot{\theta}}_{\infty}\leq h^2$ and (\ref{dot_u_n_expression_46d_1}), and 
are found to involve higher order smoothing operators than the  terms already present in (\ref{tangential_vel_err_decomp}).

%In order to complete the estimates, an estimate of $\dot{\omega}$ is required.

\section{Evolution equations for the error}
An evolution equation for $\dot{\theta}$ is formed by substituting the exact solution $s_{\alpha}$, $\theta(\alpha_i)$ into (\ref{theta_i_time_der_eqn}), using consistency, and subtracting the result from (\ref{theta_i_time_der_eqn}).
This gives
\begin{align}\label{theta_i_dot_time_der_eqn_decomp}
\frac{d\dot{\theta}_i}{dt}
&=\frac{1}{s_{\alpha}}\bigpa{
S_h \dot{(u_n)}_i+\phi_s(\alpha_i) S_h(\dot{\theta}_i)
+\theta_{\alpha}(\alpha_i)\dot{(\phi_s)}_i
}\nonumber\\
&+ \left(\frac{1}{{\sigma}} \right)^\cdot\bigpa{
(u_n)_{\alpha}(\alpha_i)+\phi_s(\alpha_i)\theta_{\alpha}(\alpha_i)
}\nonumber\\
&+\Theta_i^{{NL}}+\mbox{O}(h^s)\mbox{.}
\end{align}
where the nonlinear term $\Theta_i^{{NL}}$ contains products of the variations on the right hand side.
In (\ref{theta_i_dot_time_der_eqn_decomp}), we have also used consistency to replace, for example, $S_h\theta(\alpha_i)$ with $\theta_{\alpha}(\alpha_i)$, incurring an $\mbox{O}(h^s)$ error.
% TODO: modify eqn from (243) to (246)
It is easy to see 
that the nonlinear term satisfies
\begin{align}
\dot{\Theta}_i^{{NL}}=
A_0(\dot{\theta}_i)+A_{-s}(\dot{\sigma})
+h^3 A_0 (\dot{\tau}_c)
+\mbox{O}(h^s)\mbox{,}
\end{align}
%where for the leading order part, we have used the estimate
%\begin{align}
%\Bigpa{\frac{1}{\sigma}}_i^{\cdot} (S_h (u_n))_i^{\cdot}
%&=\mbox{O}(h^3)\bigpa{
%S_h A_0(S_h^2 \dot{\theta_i})+A_{-s}(\dot{\sigma})
%+A_0(\dot{\tau}_c)
%+\mbox{O}(h^s)
%}\nonumber\\
%&=A_0(\dot{\theta}_i)+A_{-s}(\dot{\sigma})
%+h^3 A_0(\dot{\tau}_c)
%+\mbox{O}(h^s)\mbox{,}
%\end{align}
%which follows from (\ref{dot_u_n_expression_46d_1})  and the second equation of %(\ref{one_over_sigma_err_dot}).
The relation (\ref{theta_i_dot_time_der_eqn_decomp}) can be further simplified using (\ref{dot_u_n_expression_46d_1}), (\ref{tangential_vel_err_decomp}) and the first equation of (\ref{one_over_sigma_err_dot}), which give
% TODO: change the coeff's b4 Hilbert transform!!!
\begin{align}\label{theta_i_dot_time_der_eqn_decomp_A}
\frac{d\dot{\theta}_i}{dt}&=\frac{\kappa_B}{4s_{\alpha}^3}
\MCAL{H}_h(S_h^3\dot{\theta}_i)
+S_h A_0(\dot{\theta}_i)+A_0(S_h \dot{\theta}_i)
+S_h A_{-1}(\dot{\alpha_0}_i)
\nonumber\\
&+A_{-s}(\dot{\sigma})
+A_0(\dot{\tau}_c)
+\mbox{O}(h^s)\mbox{.}
\end{align}
%In deriving (\ref{theta_i_dot_time_der_eqn_decomp_A}), we have made use of the fact that $\dot{\tau}_c$ is spatially independent and can be pulled outside of a spatial operator.

The evolution equation for $\dot{\sigma}$ is derived similarly
so that from  (\ref{sigma_i_time_der_eqn}), 
%We substitute the exact solution into (\ref{sigma_i_time_der_eqn}), use consistency, and substract the result from (\ref{sigma_i_time_der_eqn}) to obtain
\begin{align}\label{sigma_i_dot_time_der_eqn_decomp}
\frac{d\dot{\sigma}}{dt}
%&=-\ip{u_n S_h\theta}_h
%+\ip{(u_n)_h(\cdot) S_h\theta(\cdot)}+\mbox{O}(h^s)\nonumber\\
=-\ip{\dot{(u_n)}\theta_{\alpha}(\cdot)}_h-\ip{u_n(\cdot)S_h\dot{\theta}}_h
-\ip{\dot{(u_n)}S_h\dot{\theta}}_h+\mbox{O}(h^s)\mbox{.}
\end{align}
The nonlinear term is estimated as
\begin{align}\label{sigma_i_dot_time_der_eqn_decomp_A}
\ip{\dot{(u_n)}S_h\dot{\theta}}_h
=A_0(\dot{\theta})+A_{-s}(\dot{\sigma})+ A_{-3}(\dot{\alpha}_{0i})
+h^2 A_0(\dot{\tau}_c)
+\mbox{O}(h^s)\mbox{,}
\end{align}
using (\ref{dot_u_n_expression_46d_1}) and the bound $\norm{S_h\dot{\theta}}_{l^2}\leq h^2$.
Equation (\ref{sigma_i_dot_time_der_eqn_decomp}) can be further simplified using Lemma \ref{product_rule_f_S_h_g_lemma}  which together with (\ref{dot_u_n_expression_46d_1}) gives 
\begin{align}\label{sigma_i_dot_time_der_eqn_decomp_A_2}
\frac{d\dot{\sigma}}{dt}&=A_0(\dot{\theta})+A_{-s}(\dot{\sigma})+ A_{-1}(\dot{\alpha}_{0i})
+ A_0(\dot{\tau}_c)
+\mbox{O}(h^s)\mbox{.}
\end{align}

We also need the variation of the evolution equation (\ref{alpha_0_t_disc_16_2}) for  $\dot{\alpha_0}$.
Let
\begin{align}
u_t(\alpha)=(\phi_s(\alpha)-u_s(\alpha))e^{i\theta(\alpha)}
\end{align}
be the difference between the tangential interface velocity at a fixed $\alpha$ and the tangential fluid velocity at $\tau(\alpha)$. 
Taking the variation of (\ref{alpha_0_t_disc_16_2}) gives
\begin{align}\label{dot_alpha_0_time_der_eqn_50_b1}
\Bigpa{\frac{d\dot{\alpha_0}}{dt}}_i
&=\frac{u_t(\alpha_i)}{s_{\alpha}e^{i\theta(\alpha_i)}}
(D_h\dot{\alpha_0})_i
-\frac{\alpha_{0\alpha}(\alpha_i)u_t(\alpha_i)}{s_{\alpha}^2 e^{i\theta(\alpha_i)}}
\dot{\sigma}
-\frac{\alpha_{0\alpha}(\alpha_i)u_t(\alpha_i)}{s_{\alpha} (e^{i\theta(\alpha_i)})^2}
(e^{i\theta_i})^{\cdot}\nonumber\\
&+\frac{\alpha_{0\alpha}(\alpha_i)}{s_{\alpha} e^{i\theta(\alpha_i)}}
(\dot{u}_t)_i+(\dot{\alpha_0}^{NL})_i+\mbox{O}(h^s)\mbox{,}
\end{align}
where $\dot{\alpha_0}^{NL}$ contains  nonlinear terms.
%which can be compactly represented as
%\begin{align}
%\dot{\alpha_0}_i^{NL}
%=\mbox{O}\bigpa{
%(S_h\dot{\alpha_0})^2
%+\dot{\sigma}^2
%+(e^{i\theta})^{\cdot 2}
%+(\dot{u}_t)^2
%}_i
%\end{align}
From (\ref{dot_u_s_expression_46d_2}) and (\ref{tangential_vel_err_decomp}), we have (taking $\chi=\frac{1}{2}$),
\begin{align}\label{u_t_i_cdot_eqn_50_c1}
(u_t)_i^{\cdot}
&=-\frac{e^{i\theta(\alpha_i)}}{4}
\MCAL{H}_h\bigpa{
\tilde{f}(\alpha_i)
D_h\dot{\alpha_0}_i
}
+A_0(S_h\dot{\theta}_i)
+A_{-s}(\dot{\sigma})
+A_{-1}(\dot{\alpha_0}_i)
\nonumber\\
&+A_0(\dot{\tau_c})
+\mbox{O}(h^s)\mbox{,}
\end{align}
where we have also used Lemma \ref{tan_vector_est_lemma}.
Define the smooth functions
\begin{align}\label{u_t_i_cdot_eqn_50_c2}
\tilde{f_1}(\alpha)=\frac{u_t(\alpha)}{s_{\alpha}e^{i\theta(\alpha)}}
\mbox{ and }
\tilde{f_2}(\alpha)=\frac{\tilde{f}(\alpha)(\alpha_0)_{\alpha}(\alpha)}{4 s_{\alpha}}
\mbox{,}
\end{align}
% Begin insert 2
where, crucially,  $\alpha_{0 \alpha}(\alpha)$, $\tilde{f}(\alpha)$ (cf. (\ref{tilde_f_alpha_eqn_251})) and hence $\tilde{f_2}(\alpha)$ are all positive functions.
We note that $\alpha_{0 \alpha}(\alpha)=1$ at $t=0$, and the positive definiteness of $\alpha_{0 \alpha}$ is a consequence of $\alpha_0(\alpha)$ being a one-to-one mapping, which is related to the physical property that material fluid points cannot overlap.
% end insert 2
Using (\ref{u_t_i_cdot_eqn_50_c1}),  (\ref{u_t_i_cdot_eqn_50_c2}), and the commutator identity Lemma \ref{Commutator_thm_H_h_g},  equation (\ref{dot_alpha_0_time_der_eqn_50_b1}) can be written
\begin{align}\label{u_t_i_cdot_eqn_50_c3}
\Bigpa{\frac{d\dot{\alpha_0}}{dt}}_i
&=\tilde{f_1}(\alpha_i)
(D_h\dot{\alpha_0})_i
-\tilde{f_2}(\alpha_i)
\MCAL{H}_h D_h\dot{\alpha_0}_i
+A_0(S_h\dot{\theta}_i)
\nonumber\\
&+A_{-s}(\dot{\sigma})
+A_{-1}(\dot{\alpha_0}_i)
+A_0(\dot{\tau_c})
+\mbox{O}(h^s)\mbox{.}
\end{align}
The nonlinear terms are smoother or smaller than terms that are already present in (\ref{u_t_i_cdot_eqn_50_c3}), as is easily verified.

Next, we derive the evolution equation for $\dot{\tau_c}$. 
%We have defined $\tau_c$ 
%as the $k=0$ Fourier mode of the discrete interface $\tau_i$.
%Recall the evolution equation for $\tau_c$ is (\ref{tau_c_time_der_hat_u_0}),
%\begin{align}
%\frac{d\tau_c}{dt}\Big\rvert_{\alpha}
%=\hat{v}_0\mbox{.}
%\end{align}
This immediately follows from (\ref{tau_c_time_der_hat_u_0}), i.e.,
\begin{align}\label{zeta_c_dot_time_der_eqn}
\frac{d\dot{\tau}_c}{dt}\Big\rvert_{\alpha}
=\hat{\dot{v}}_0.
\end{align}
where $\hat{v}_0$ is the $k=0$ Fourier mode of the interface velocity (\ref{vel_eqn_9a_1}). 
%is the evolution equation for $\dot{\tau}_c$.
Equations (\ref{theta_i_dot_time_der_eqn_decomp_A}), (\ref{sigma_i_dot_time_der_eqn_decomp_A_2}), (\ref{u_t_i_cdot_eqn_50_c3}), and (\ref{zeta_c_dot_time_der_eqn}) are the main result of this section.

%%%%%%%%%%%%%%%%%%%%%%%%%%%%%%%%%%%%%%%%%%%%%%
% New section %%%%%%%%%%%%%%%%%%%%%%%%%%%%%%%%

\section{Energy estimates}
Recall that we have defined a time $T^*$ in (\ref{T*def}), and  
%\begin{align}
%T^{\ast}\equiv \sup\bigbra{t: 0\leq t\leq T, \norm{\dot{\sigma}}_{l^2}
%,\norm{\dot{\theta}}_{l^2},\norm{\dot{\tau}}_{l^2},\norm{\dot{\alpha_0}}_{l^2}\leq h^{\frac{7}{2}} }
%\mbox{.}
%\end{align}
%The power of $h$ in the above definition is chosen for convenience, so that the estimates below easily go through.
all the estimates we obtain are valid for $t\leq T^{\ast}$. We close this so-called ``$T^*$ argument'' and prove Theorem \ref{MainThm} by showing at the end that $T^{\ast}$ can be extended to $T$, the existence time for the continuous problem.

Define the energy
\begin{align}
E(t)=\dot{\sigma}^2
+\bigpa{\dot{\theta},\dot{\theta}}_h
+\bigpa{\dot{\alpha_0},\dot{\alpha_0}}_h
+\abs{\dot{\tau_c}}^2\mbox{,}
\end{align}
and take the time derivative 
% TODO: correct 1/2 problem!
\begin{align}\label{Energy_time_der_decomp_2}
\frac{1}{2}\frac{dE}{dt}&=\dot{\sigma}\dot{\sigma}_t
+\bigpa{\dot{\theta},\dot{\theta}_t}_h
+\bigpa{\dot{\alpha_0},\dot{\alpha_0}_t}_h
+2\mbox{Re}\bigpa{\OVL{\dot{\tau_c}}\dot{\tau_c}_t}\mbox{.}
\end{align}
We will bound the  above by $E(t)+O(h^s)$.

The first product on the right side of (\ref{Energy_time_der_decomp_2}) is readily bounded using (\ref{sigma_i_dot_time_der_eqn_decomp_A_2}) and Young's inequality,
\begin{align}
\dot{\sigma} \dot{\sigma}_t 
\leq cE+\mbox{O}(h^s)\mbox{.}
\end{align}
The other inner products in (\ref{Energy_time_der_decomp_2}) are bounded by making use of parabolic smoothing.  We have from (\ref{theta_i_dot_time_der_eqn_decomp_A})
\begin{align} \label{energy_deriv_theta_dot}
\bigpa{\dot{\theta},\dot{\theta}_t}_h &=
\Bigpa{\dot{\theta},
\frac{\kappa_B}{4 s_{\alpha}^3}\MCAL{H}_h(S_h^3\dot{\theta})
+S_h A_0(\dot{\theta})+A_0(S_h \dot{\theta})\\
&+S_h A_{-1}(\dot{\alpha}_{0i})+A_{-s}(\dot{\sigma})
+ A_0(\dot{\tau}_c)
+\mbox{O}(h^s)
}_h.  \nonumber
\end{align}
The first term in this inner product is evaluated as
\begin{align}\label{Energy_est_52_1}
\Bigpa{
\dot{\theta},
\frac{\kappa_B}{4 s_{\alpha}^3}\MCAL{H}_h(S_h^3\dot{\theta})
}_h
=-\frac{\kappa_B \pi}{2 s_{\alpha}^3}
\DPS\sum_{k=-\frac{N}{2}+1}^{\frac{N}{2}-1}
\abs{k}^3\abs{\hat{\dot{\theta}}_k}^2\mbox{,}
\end{align}
where we have used the discrete Parseval relation (Lemma \ref{IP_prod_lemma}) 
and  (\ref{H_h_disc_real_line_prop_1}).
%and the real valuedness of $\dot{\theta}$ (so that $\abs{\hat{\dot{\theta}}_k}=\abs{\hat{\dot{\theta}}_{-k}}$).
The sum extends to $k=\frac{N}{2}-1$, in view of zeroing out the $\frac{N}{2}$ mode of $S_h\dot{\theta}$.
The next term is bounded using  Lemma \ref{product_rule_f_S_h_g_lemma} and Young's inequality. First, introduce a generic discrete function $\dot{g}_{1i}=A_0(\dot{\theta}_i)$. Then
\begin{align}\label{Energy_est_52_2}
\bigabs{\bigpa{\dot{\theta},S_h A_0(\dot{\theta})}_h}
&=\bigabs{\bigpa{S_h\dot{\theta},A_0(\dot{\theta})}_h}\nonumber\\
&\leq 2 \pi  \DPS\sum_{k=-\frac{N}{2}+1}^{\frac{N}{2}-1}
\abs{k}\abs{\hat{\dot{\theta}}_k}   \abs{(\widehat{\dot{g}}_1)_k}
\mbox{}
\nonumber\\
&\leq \pi \biggmpa{\DPS\sum_{k=-\frac{N}{2}+1}^{\frac{N}{2}-1}
\bigpa{
k^2 \abs{\hat{\dot{\theta}}_k}^2
+   \abs{(\widehat{\dot{g}}_1)_k}^2
}
}
\mbox{}
\nonumber\\
&\leq c\biggpa{\DPS\sum_{k=-\frac{N}{2}+1}^{\frac{N}{2}-1}
k^2\abs{\hat{\dot{\theta}}_k}^2 + E
}\mbox{,}
\end{align}
for a constant $c$.
In the last inequality we have used
\begin{align}\label{Fourier_theta_norm_est_54_1}
\DPS\sum_{k=-\frac{N}{2}+1}^{\frac{N}{2}-1}
   \abs{(\widehat{\dot{g}}_1)_k}^2
\leq\norm{\dot{g_1}}_{l^2}^2
\leq c\norm{\dot{\theta}}_{l^2}^2
\leq cE.
\end{align}
We similarly introduce a generic discrete function $\dot{g}_{2i}=A_0(S_h \dot{\theta}_i)$ and bound
\begin{align}\label{Fourier_theta_norm_est_54_2}
\bigabs{\bigpa{\dot{\theta},A_0(S_h\dot{\theta})}_h}
&\leq 2 \pi \DPS\sum_{k=-\frac{N}{2}+1}^{\frac{N}{2}}
\abs{\hat{\dot{\theta}}_k}
   \abs{(\widehat{\dot{g}}_2)_k} \nonumber\\
&\leq \pi \biggmpa{
\DPS\sum_{k=-\frac{N}{2}+1}^{\frac{N}{2}}
\bigpa{
\abs{\hat{\dot{\theta}}_k}^2
+   \abs{(\widehat{\dot{g}}_2)_k}^2
}
}\nonumber\\
&\leq c\biggpa{\DPS\sum_{k=-\frac{N}{2}+1}^{\frac{N}{2}-1}
k^2\abs{\hat{\dot{\theta}}_k}^2 + E
}\mbox{,}
\end{align}
where $c>0$ and the last inequality follows from the bound
\begin{align}
\DPS\sum_{k=-\frac{N}{2}+1}^{\frac{N}{2}} 
   \abs{(\widehat{\dot{g}}_2)_k}^2
\leq\norm{A_0(S_h\dot{\theta})}_{l^2}^2
\leq c\norm{S_h\dot{\theta}}_{l^2}^2
=c\DPS\sum_{k=-\frac{N}{2}+1}^{\frac{N}{2}-1}
k^2\abs{\hat{\dot{\theta}}_k}^2\mbox{.}
\end{align}
The inner product $(\dot{\theta}, S_h A_{-1}(\dot{\alpha}_{0})) = -(S_h \dot{\theta},  A_{-1}(\dot{\alpha}_{0}))$ is controlled following the same analysis as in  (\ref{alpha_0_dot_energy}) below. 
The remaining terms in (\ref{Energy_time_der_decomp_2}) are bounded as
\begin{align}
\bigabs{\bigpa{\dot{\theta},
A_{-s}(\dot{\sigma})
+A_0(\dot{\tau}_c)
+\mbox{O}(h^s)
}_h
}\leq cE+\mbox{O}(h^s)\mbox{.}
\end{align}

% Begin energy insert 1
Next, we estimate the inner product $(\dot{\alpha_0}_t,\dot{\alpha_0})_h$ in (\ref{Energy_time_der_decomp_2}).
Substitute (\ref{u_t_i_cdot_eqn_50_c3}) for $\dot{\alpha_0}_t$ to obtain
\begin{align}\label{alpha_0_dot_t_energy_est_eqn_52a_1}
\bigpa{\dot{\alpha_0}_t,\dot{\alpha_0}}_h
&=
\bigpa{
\tilde{f_1}(\cdot)
D_h\dot{\alpha_0}
,\dot{\alpha_0}
}_h
-\bigpa{\tilde{f_2}(\cdot)
\MCAL{H}_h (D_h\dot{\alpha_0})
,\dot{\alpha_0}
}_h
+\bigpa{A_0(S_h\dot{\theta})
\nonumber\\
&+A_{-s}(\dot{\sigma})
+A_{-1}(\dot{\alpha_0})
+A_0(\dot{\tau_c})
,\dot{\alpha_0}
}_h
+\mbox{O}(h^s)\mbox{.}
\end{align}
The first inner product on the right hand side of (\ref{alpha_0_dot_t_energy_est_eqn_52a_1}), which can be written $\bigpa{\tilde{f_1}(\cdot)\dot{\alpha_0},D_h\dot{\alpha_0}}_h$, is estimated using Lemmas  \ref{product_rule_for_D_h_lemma} and \ref{product_rule_f_S_h_g_lemma}  as
\begin{align}\label{f_1_tilde_D_h_alpha_0_dot_t_energy_est_eqn_52b_1}
\bigpa{\tilde{f_1}(\cdot)\dot{\alpha_0},D_h\dot{\alpha_0}}_h
&=-\bigpa{% integration by parts?
D_h(\tilde{f_1}(\cdot)\dot{\alpha_0})
,\dot{\alpha_0}
}_h\nonumber\\
&=-\bigpa{
\tilde{f_1}(\cdot)D_h\dot{\alpha_0}
+\dot{\alpha}_0^q\tilde{f}_{1\alpha}(\cdot)
+ A_{-1} (\dot{\alpha_0})
,\dot{\alpha_0}
}_h\mbox{.}
\end{align}
Move the first inner product on the right hand side of (\ref{f_1_tilde_D_h_alpha_0_dot_t_energy_est_eqn_52b_1}) to the left hand side
(also moving the real function $\tilde{f_1}$ to the other side of the inner product) to obtain
\begin{align}
2\bigpa{
\tilde{f_1}(\cdot)\dot{\alpha_0}
,D_h\dot{\alpha_0}
}_h
=-\bigpa{\dot{\alpha}_0^q\tilde{f}_{1\alpha}(\cdot)
+ A_{-1} (\dot{\alpha_0})
,\dot{\alpha_0}
}_h\mbox{.}
\end{align}
This shows that the inner product on the left is bounded by the energy, i.e.,
\begin{align}
\bigabs{
\bigpa{\tilde{f_1}(\cdot)\dot{\alpha_0},D_h\dot{\alpha_0}}_h
}\leq cE\mbox{.}
\end{align}

The second inner product on the right hand side of (\ref{alpha_0_dot_t_energy_est_eqn_52a_1}) can be written $-\bigpa{\Lambda_h^p\dot{\alpha_0},\tilde{f_2}(\cdot)\dot{\alpha_0}}_h$, where we have defined $\Lambda_h^p=\MCAL{H}_h D_h$.
To bound this inner product, we make essential use of the positive definiteness of $\tilde{f_2}$.
We first write:
\begin{align}\label{Lambda_h_p_dot_alpha_0_inner_prod_f_2_tilde_dot_alpha_0_eqn_52e1}
-\bigpa{
\Lambda_h^p\dot{\alpha_0},\tilde{f_2}(\cdot)\dot{\alpha_0}
}_h
=-\bigpa{
\sqrt{\tilde{f_2}(\cdot)}\Lambda_h^p\dot{\alpha_0}
,\sqrt{\tilde{f_2}(\cdot)}\dot{\alpha_0}
}_h\mbox{,}
\end{align}
then move $\sqrt{\tilde{f_2}(\cdot)}$ inside the argument of the operator $\Lambda_h^p$, which by Lemma \ref{Commutator_thm_H_h_g} and the discrete product rule Lemma \ref{product_rule_for_D_h_lemma} introduces a commutator and other terms whose inner product with $\dot{\alpha}_0$ can be bounded by energy.
If we define $\dot{\tilde{\alpha}}_0=\sqrt{\tilde{f_2}(\cdot)}\dot{\alpha}_0$, then the preceding statements imply that
\begin{align}\label{f_2_tilde_Lambda_h_p_alpha_0_dot_inner_prod_eqn_52g_1}
-\bigpa{
\sqrt{\tilde{f_2}(\cdot)}\Lambda_h^p\dot{\alpha_0}
,\sqrt{\tilde{f_2}(\cdot)}\dot{\alpha_0}
}_h
=-\bigpa{
\Lambda_h^p\dot{\tilde{\alpha}}_0
,\dot{\tilde{\alpha}}_0
}_h
+r,
\end{align}
where $r\in\MBB{R}$ satisfies $\abs{r}<cE$.
The inner product on the right hand side of (\ref{f_2_tilde_Lambda_h_p_alpha_0_dot_inner_prod_eqn_52g_1}) satisfies
\begin{align}\label{Lambda_h_p_dot_tilde_alpha_0_inner_prod_dot_tilde_alpha_0_eqn_298}
\bigpa{
\Lambda_h^p\dot{\tilde{\alpha}}_0
,\dot{\tilde{\alpha}}_0
}_h
= 2 \pi \DPS\sum_{k=-\frac{N}{2}+1}^{\frac{N}{2}}
\abs{k}\rho(kh)\abs{\dot{\tilde{\alpha}}_0}^2
>0
\mbox{.}
\end{align}
Combining (\ref{Lambda_h_p_dot_alpha_0_inner_prod_f_2_tilde_dot_alpha_0_eqn_52e1})-(\ref{Lambda_h_p_dot_tilde_alpha_0_inner_prod_dot_tilde_alpha_0_eqn_298}) shows that
\begin{align}
\bigpa{
\Lambda_h^p\dot{\alpha}_0
,\tilde{f_2}(\cdot)\dot{\alpha_0}
}_h
\leq cE\mbox{,}
\end{align}
which gives the desired estimate on the second inner product in (\ref{alpha_0_dot_t_energy_est_eqn_52a_1}).

The third inner product that we need to estimate is $\bigpa{A_0(S_h\dot{\theta}),\dot{\alpha_0}}_h$.
This is bounded using Young's inequality as
\begin{align}
\bigabs{
\bigpa{A_0(S_h\dot{\theta}),\dot{\alpha_0}}_h
}&\leq \frac{1}{2}
\bigpa{
\norm{A_0(S_h\dot{\theta})}_{l^2}^2
+\norm{\dot{\alpha_0}}_{l^2}^2
}\mbox{,}\nonumber\\
&\leq c\bigpa{
\norm{S_h\dot{\theta}}_{l^2}^2
+\norm{\dot{\alpha_0}}_{l^2}^2
}\mbox{,}\nonumber\\
&\leq c\Bigpa{
\DPS\sum_{k=-\frac{N}{2}+1}^{\frac{N}{2}-1}
\abs{k}^2
\abs{\hat{\dot{\theta}}_k}^2
+E
}\mbox{.} \label{alpha_0_dot_energy}
\end{align}
The first sum above is controlled by parabolic smoothing (i.e. by the dominant contribution from the leading order term (\ref{Energy_est_52_1})).
% TODO: ref (290)
%The second sum above is bounded by the energy.

The other inner products in (\ref{alpha_0_dot_t_energy_est_eqn_52a_1}) are clearly bounded by $cE$.
Putting these estimates together, we obtain the bound
\begin{align}
\bigpa{\dot{\alpha_0}_t,\dot{\alpha_0}}_h
&\leq c\Bigpa{
\DPS\sum_{k=-\frac{N}{2}+1}^{\frac{N}{2}-1}
\abs{k}^2
\abs{\hat{\dot{\theta}}_k}^2
+E
}+\mbox{O}(h^s)\mbox{.}
\end{align}
% End energy insert 1

The final term in (\ref{Energy_time_der_decomp_2}) is estimated using (\ref{zeta_c_dot_time_der_eqn}) and Young's inequality as
% TODO: fix the bold face v problem!
\begin{align}
2\abs{\mbox{Re}\bigpa{\OVL{\dot{\tau_c}}\dot{\tau_c}_t}}
&\leq 2\abs{\dot{\tau_c}}\abs{\dot{\tau_c}_t}\mbox{,}
\nonumber\\
&\leq 2\abs{\dot{\tau_c}}
\abs{\widehat{\dot{v}}_0}\mbox{,}\nonumber\\
&\leq \abs{\dot{\tau_c}}^2
+\norm{
\dot{v}
}_{l^2}^2\mbox{,}
\end{align}
where we recall that $\dot{v}=\bigpa{u_n i e^{i\theta}+\phi_s e^{i\theta}}^{\cdot}$.
From expressions for $\dot{u_n}$ and $\dot{\phi_s}$ given in (\ref{dot_u_n_expression_46d_1}) and (\ref{tangential_vel_err_decomp}), it is easy to see that
\begin{align}
\norm{
\dot{v}
}_{l^2}^2
&\leq c\bigpa{
\norm{S_h\dot{\theta}}_{l^2}^2
+\norm{A_{-1}(\dot{\alpha}_0)}_{l^2}^2
+\abs{\dot{\tau_c}}^2
+\norm{A_0(\dot{\theta})}_{l^2}^2
\nonumber\\
&+\norm{A_{-s}(\dot{\sigma})}_{l^2}^2
+\mbox{O}(h^s)
}\mbox{.}
\end{align}
Therefore,
\begin{align}\label{Energy_zeta_c_est_57_1}
2\abs{\mbox{Re}\bigpa{\OVL{\dot{\tau_c}}\dot{\tau_c}_t}}
&\leq c\Bigpa{
\DPS\sum_{k=-\frac{N}{2}+1}^{\frac{N}{2}-1}
\abs{k}^2
\abs{\hat{\dot{\theta}}_k}^2
+cE+\mbox{O}(h^s)
}\mbox{,}
\end{align}
which gives the desired bound on the last term in  (\ref{Energy_time_der_decomp_2}).

We now put these estimates together.
First, set
\begin{align}
d_1=\DPS\min_{0\leq t\leq T}\frac{\kappa_B \pi}{2s_{\alpha}}\mbox{,}
\end{align}
and note by assumptions on $s_{\alpha}$ that 
%$0<c_1<d_1<\infty$, so in particular 
$d_1$ is bounded away from zero and infinity.
Then from the above estimates,
%(\ref{Energy_est_52_1}), (\ref{Energy_est_52_2}), (\ref{Fourier_theta_norm_est_54_1}), (\ref{Fourier_theta_norm_est_54_2}), and (\ref{Energy_zeta_c_est_57_1}), 
there exists positive constants $d_2,d_3$, such that (\ref{Energy_time_der_decomp_2}) can be bounded as
\begin{align}\label{Energy_time_der_est_59_1}
\frac{dE}{dt}\leq \DPS\sum_{k=-\frac{N}{2}+1}^{\frac{N}{2}-1}
\bigpa{
-d_1\abs{k}^3+d_2 k^2
}\abs{\hat{\dot{\theta}}_k}^2
+d_3 E
+\mbox{O}(h^s)\mbox{.}
\end{align}
% Replace 1
Equation (\ref{Energy_time_der_est_59_1}) can be written as
\begin{align}
\frac{dE}{dt}
&\leq \DPS\sum_{k=-\frac{N}{2}+1}^{\frac{N}{2}-1}
d_4\abs{\hat{\dot{\theta}}_k}^2
+d_3 E
+\mbox{O}(h^s)\mbox{,}
\end{align}
where
\begin{align}
d_4=\DPS\max_{-\frac{N}{2}+1\leq k\leq\frac{N}{2}-1}
\bigpa{
-d_1\abs{k}^3
+d_2 k^2
}\mbox{,}
\end{align}
and note that $d_4$ is bounded away from zero and infinity.
It readily follows that there exists a positive constant $c$ such that
\begin{align}\label{Energy_time_der_est_60_1}
\frac{dE}{dt}\leq cE+\mbox{O}(h^s)\mbox{, with }
E(0)=0
\end{align}
for $t\leq T^{\ast}$, which is the main result of this section.

Stability and convergence of our numerical method now follows from
application of Gronwall's inequality to (\ref{Energy_time_der_est_60_1}), which  gives
\begin{align}
E(t)\leq c h^s t(1+e^t)\mbox{ for } t\leq T^{\ast}\mbox{,}
\end{align}
or
\begin{align}
E(t)\leq c(T^{\ast}) h^s \mbox{,}
\end{align}
It follows that
\begin{align}
\norm{\dot{\sigma}}_{l^2}^2\mbox{, }
\norm{\dot{\theta}}_{l^2}^2\mbox{, }
\norm{\dot{\alpha}_0}_{l^2}^2\mbox{, }
\norm{\dot{\tau}}_{l^2}^2\leq c(T^{\ast}) h^s\mbox{,}
\end{align}
where we have used $\norm{\dot{\tau}}_{l^2}^2\leq cE$, which follows from Lemma \ref{zeta_dot_i_lemma} and $\left| {\dot{\tau_c}} \right|^2\leq cE$.
We choose $m$ large enough, so that $s$ can be picked to satisfy $s\geq 8$.
(Recall that $m$ characterizes the smoothness of the continuous solution, and $s$ is near $m$).
Then
\begin{align}\label{sigma_theta_zeta_l2_norm_sq_bd_est_61_2}
\norm{\dot{\sigma}}_{l^2}\mbox{, }
\norm{\dot{\theta}}_{l^2}\mbox{, }
\norm{\dot{\alpha_0}}_{l^2}\mbox{, }
\norm{\dot{\tau}}_{l^2}\leq c(T^{\ast}) h^{\frac{s}{2}}
<h^{\frac{7}{2}}
\end{align}
for $h$ small enough.
It follows from the definition (\ref{T*def}) that we can extend $T^{\ast}$ to $T^{\ast}=T$, so that the bounds (\ref{sigma_theta_zeta_l2_norm_sq_bd_est_61_2}) are valid throughout the entire interval $0\leq t\leq T$ in which a smooth continuous solution exists.
This completes the proof of the convergence of our method for $\beta=0$, $\chi=\frac{1}{2}$ and $\kappa_B>0$.

\section{Unequal viscosities ($\beta\neq 0$).} \label{sec:unequal_viscosities}
The case $\beta\neq 0$ corresponds to a viscosity contrast between the internal and external fluids.
In this case, we must account for the additional nonlocal equation (\ref{I_plus_beta_K_matrix_eqn_2_1}) and term $\bm{\tilde{\omega}}$.

We start by taking the variation of (\ref{I_plus_beta_K_matrix_eqn_2_1}), which is written as
\begin{align} \label{variation_int_eqn_omega}
(\MBF{I}+\beta\MBF{K})\dot{\bm{\tilde{\omega}}}_i
=-\beta\dot{\MBF{K}}
\left( \bm{\tilde{\omega}}(\alpha_i)+\MBF{g}^p(\alpha_i) \right)- \beta \MBF{K} \dot{\MBF{g}}^p_i
\mbox{,}
\end{align}
where, e.g., 
%$\MBF{K}$ is defined in (\ref{K_omega_matrix_eqn_3_0}) and (\ref{I_plus_beta_K_matrix_eqn_2_1})
\begin{align}
\dot{\MBF{K}}\bm{\tilde{\omega}}(\alpha_i)
&=\DPS\sum_{\substack{j=-\frac{N}{2}+1\\(j-i)\mbox{ odd}}}^{\frac{N}{2}}
\begin{pmatrix}
\bigpa{\dot{K}_R^{(1)}}_{i,j} + \bigpa{ \dot{K}_R^{(2)}}_{i,j}  & -\bigpa{\dot{K}_I^{(2)}}_{i,j} \\
\bigpa{\dot{K}_I^{(2)}}_{i,j} & \bigpa{\dot{K}_R^{(1)}}_{i,j} -\bigpa{\dot{K}_R^{(2)}}_{i,j} 
\end{pmatrix}
\begin{pmatrix}
\tilde{\omega}_1(\alpha_j) \\
\tilde{\omega}_2(\alpha_j)
\end{pmatrix}
(2h)
\mbox{,}
\end{align}
and $\dot{\MBF{g}}^p$ is the  variation of the filtered version of   (\ref{g_i_def}).
Here we have defined, e.g.,
\begin{align}
{\MBF{K}}\dot{\bm{\tilde{\omega}}}_i
=\MBF{K}_{i,j}\dot{\bm{\tilde{\omega}}}_j
=\bigpa{
\MBF{K}_h(\alpha_i,\alpha_j)
+\dot{\MBF{K}}_{i,j}}\dot{\bm{\tilde{\omega}}}_j
,
\end{align}
which contains both linear and nonlinear terms in the variation.
We now use the fact that the kernels $\bigpa{K_R^{(1)}}_{i,j}$, $\bigpa{K_R^{(2)}}_{i,j}$, and $\bigpa{K_I^{(2)}}_{i,j}$ are simple modifications of 
the kernels in (\ref{UI12_disc_est}), (\ref{UI4_disc_est}),
%For example, the kernel in (\ref{UI12_disc_est}) is equivalent to $\bigpa{K_R^{(1)}}_{i,j}$ if in (\ref{UI12_disc_est}),    Re is changed into Im, $\sigma$ is added, and the cotangent term is omitted.
%The kernels $\bigpa{K_R^{(2)}}_{i,j}$ and $\bigpa{K_I^{(2)}}_{i,j}$ are equal to the imaginary and real parts of the kernel in (\ref{UI4_disc_est}).
so that  the same analysis as in  Section \ref{sec:summary_variation} shows that 
\begin{align} \label{dot_Kh_est}
\dot{\MBF{K}} \left( {\bm{\tilde{\omega}}(\alpha_i)+\MBF{g}}^p(\alpha_i) \right)
=A_0(\dot{\theta}_i)+A_{-s}(\dot{\sigma})+A_{-1}(\dot{\alpha_{0i}})
+\mbox{O}(h^s)\mbox{}
\end{align}
(the above relation is defined as holding for each component of the vector).
%This is easily seen by omitting the terms containing $\dot{\omega}$ in (\ref{UI1_disc_HK_decomp}), (\ref{UI4_disc_HK_decomp}), and estimating the remain terms with the aid of (\ref{K_star_h_A_decomp_eqn_232}), (\ref{K_star_h_A_decomp_eqn_233}).
Following the same argument as that leading to (\ref{K0h_omega_est}),  we also have
\begin{align} \label{K_dot_gp_est}
    \MBF{K} {\MBF{\dot{g}}}^p_i=A_0(\dot{\theta}_i)+A_{-s}(\dot{\sigma})+A_{-1}(\dot{\alpha_{0i}})
    +A_0(\dot{\tau}_c)
+\mbox{O}(h^s)\mbox{.}
\end{align}

Together, the above results imply that  (\ref{variation_int_eqn_omega}) can be written as
\begin{align} \label{omegatil_bound}
(\MBF{I}+\beta\MBF{K})\dot{\bm{\tilde{\omega}}}_i
=A_0(\dot{\theta})+A_{-s}(\dot{\sigma})+A_{-1}(\dot{\alpha_0}_i) +A_0(\dot{\tau}_c)
+\mbox{O}(h^s)\mbox{.}
\end{align}
Apparently, if $(\MBF{I}+\beta\MBF{K})^{-1}$ is bounded, then we can bound $\dot{\bm{\tilde{\omega}}}$ by the right-hand side of (\ref{omegatil_bound}).
This implies that $\dot{\bm{\tilde{\omega}}}$ does not contribute to the leading order analysis, and the energy estimates for $\beta \ne 0$ proceed exactly as in the case $\beta=0$.
The bound on $(\MBF{I}+\beta\MBF{K})^{-1}$ is given by the following lemma:
\begin{lemma}\label{I_plus_beta_K_matrix_lemma_6_1}
Assume $\theta(\cdot,t)\in C^s$ and $s_{\alpha}\neq 0$.
Then there exists constants $h_0>0$, $\beta_0>0$, and  $c>0$, such that for all $h$ with $0<h\leq h_0$, and $\beta$ with $0<\beta<\beta_0$, $(\MBF{I}+\beta\MBF{K})^{-1}\leq c$.
%Furthermore, the same estimates hold for the filtered operator $\MBF{K}^p\bm{\omega}_i=\MBF{K}\bm{\omega}_i^p$.
\end{lemma}
The proof of Lemma \ref{I_plus_beta_K_matrix_lemma_6_1} will be deferred to the appendix.

\section{Drop with constant surface tension}

Finally, we consider the case of a drop interface with zero membrane bending stress and constant surface tension,  $\kappa_B=0$ and  $\MCAL{S}=1$, at first for equal interior and exterior viscosities or $\beta=0$. As described in Section \ref{sec:numerical_method}, the discretization  is then modified from that for an elastic interface. Since the highest, second derivative term is eliminated from (\ref{g_i_def}), the stability analysis must also be modified.  

When $\kappa_B=0$ and $\MCAL{S}=1$, we have from (\ref{omega_disc_err_decomp}), (\ref{Exp_i_theta_err}), and Lemma \ref{G_h_p_lemma},
\begin{align} \label{dot_omega_p}
\dot{\omega}_i^p = h(\alpha_i) {\dot{\theta}_i}^p + A_{-1}(\dot{\theta}_i)+A_{-s}(\dot{\sigma}),
\end{align}
where $h(\alpha)=i e^{i \theta(\alpha)}$ is a smooth function. The velocity is decomposed as in (\ref{vel_decomp_drop}), and we adapt the analysis in Section 
\ref{sec:summary_variation} to estimate $(\dot{u}_R)_i$. In particular, we make use of the additional filtering for the drop problem and  the smoothing properties of convolutions to replace (\ref{K_star_h_A_decomp_eqn_232}), (\ref{K_star_h_A_decomp_eqn_233}), etc., with the improved estimates
\[
K_{lh}[\cdot,\cdot](\dot{\omega}_i^p)=A_{-1}(\dot{\theta}_i)+A_{-s}(\dot{\sigma}),
\]
for $l=0,1,3$, and
\[
K_{lh}[\cdot,\cdot](D_h \dot{\tau}_i)=A_{-2}(\dot{\theta}_i)+A_{-s}(\dot{\sigma}),
\]
for $l=1,3$. In making the former estimate we absorb the smooth function $h(\alpha)$ 
%from (\ref{dot_omega_p}) 
into the kernel.  The analysis of the nonlinear term in the velocity variation is also modified.  Specifically, we use (\ref{dot_omega_p}) to obtain the improved estimate
\begin{align} \label{u_nonlinear_drop}
\dot{u}_i^{NL} = h^{3/2} A_0(\dot{\theta_i}) + A_{-s} (\dot{\sigma}) + O(h^s).
\end{align}
It follows that 
\begin{equation} \label{dot_ur_drop}
(\dot{u}_R)_i=A_{-1}(\dot{\theta_i}) + A_{-s} (\dot{\sigma}) + A_0(\dot{\tau}_c)+ O(h^s).
\end{equation}

The variation of the normal velocity is given by  (\ref{num_approx_u_n_i_decomp_2}) with $\kappa_B=0$.  The first commutator there  is estimated as $A_{-2}(\dot{\theta})+A_{-s} (\dot{\sigma})$ after inserting (\ref{dot_omega_p}), absorbing $h(\alpha_i)$ into the kernel, and using the smoothing property in Lemma \ref{Commutator_thm_H_h_g}. 
We use (\ref{Exp_i_theta_err}) to  write the second commutator as the sum   $-(1/4) \MCAL{H}_h \dot{\theta}_i+
h_1(\alpha_i) \dot{\theta}_i$ plus some smoother terms, where  $h_1(\alpha)$ is a smooth real function.  The remaining terms in (\ref{num_approx_u_n_i_decomp_2}) are smoother. The final estimate is
\begin{align}
    (\dot{u}_n)_i=-\frac{1}{4} \MCAL{H}_h \dot{\theta}_i^p+h_1(\alpha_i) \dot{\theta}_i + A_{-1}(\dot{\theta_i}) + A_{-s} (\dot{\sigma}) + A_0(\dot{\tau}_c)+ O(h^s).
\end{align}
From (\ref{dot_phi_s_i_error_A_decomp_1}) and Lemma \ref{Int_product_rule_f_S_h_g_lemma}, we also  find that 
\[
(\dot{\phi}_s)_i = A_{0}(\dot{\theta_i}) + A_{-s} (\dot{\sigma}) + A_0(\dot{\tau}_c)+ O(h^s),
\]
and it is easy to see that $\dot{u}_s$ satisfies the same estimate as $\dot{\phi}_s$.

\noindent
\underline{\textit{Energy estimates}}.
\newline

The energy is defined as
\begin{align}
E(t)=\dot{\sigma}^2
+\bigpa{\dot{\theta},\dot{\theta}}_h
+\abs{\dot{\tau_c}}^2\mbox{,}
\end{align}
and we form $dE/dt$ as in (\ref{Energy_time_der_decomp_2}).  Two inner products 
from  $(\dot{\theta}, \dot{\theta}_t)_h$ 
involve a derivative of $\dot{\theta}$.
These are  $(\dot{\theta}, -(1/4) \Lambda_h^p \dot{\theta})_h$ and  $(\dot{\theta}, h_1(\cdot) D_h
\dot{\theta})_h$, where we recall that $h_1(\alpha)$ is a smooth real function. The first of these is negative definite and thus decreases the energy, and the second
can be estimated with the aid of (\ref{eq:discrete_product_rule}) as
\begin{align}
    (\dot{\theta},h_1(\cdot) D_h \dot{\theta})&=(\dot{\theta}, D_h (h_1 (\cdot)  \dot{\theta}))+(\dot{\theta}, A_0(\dot{\theta}) ) \\
    &= -(D_h \dot{\theta}, h_1(\cdot) \dot{\theta})_h+(\dot{\theta}, A_0(\dot{\theta})),
\end{align}
so that 
\begin{align}
     (\dot{\theta}, h_1(\cdot) D_h \dot{\theta})_h = (\dot{\theta}, A_0(\dot{\theta})).
\end{align}
This shows that $(\dot{\theta}, h_1(\cdot) D_h \dot{\theta})_h$ is bounded by the energy. It is easily seen that all the other inner products in $dE/dt$ can be bounded by the energy, from which the convergence of the method  readily follows for $\kappa_B=0$, $\MCAL{S}=1$, and $\beta=0$.

\noindent
\newline
\underline{\textit{Drop with viscosity contrast ($\beta \ne 0$)}}.
\newline

The analysis for a drop with $\beta \ne 0$ follows that in Section \ref{sec:unequal_viscosities} with a few minor changes. In view of the additional filtering in the numerical method for the drop problem,
%described in Section \ref{sec:numerical_method}, 
the right hand sides of the estimates (\ref{dot_Kh_est})-(\ref{omegatil_bound}) can be improved to $A_{-1}(\dot{\theta}_i)+A_{-s}(\dot{\sigma})
+A_0(\dot{\tau}_c)$. It readily follows that the additional term $\dot{\bm{\tilde{\omega}}}$ from inverting (\ref{variation_int_eqn_omega}) does not contribute to the leading order analysis, and the energy estimates for $\beta \ne 0$ proceed exactly as in the case $\beta=0$.

\section{Conclusions}
% Begin insert 16
A convergence proof has been presented for a boundary integral method for interfacial Stokes flow.
While previous convergence studies of the boundary integral method exist for interfacial potential flow, this is the first work that we are aware of for the important case of interfacial Stokes flow.
Our analysis has focused on a spectrally accurate numerical method, adapted in this paper from  \cite{HSB:2012},  \cite{XBS:2013}, for a Hookean elastic capsule with membrane bending stress evolving in an externally applied strain or shear flow. The method is rather general, and setting the interfacial tension $\MCAL{S}$ to a constant and the bending stress $\kappa_B$ to zero  gives a method for simulating a drop or bubble in a shear or strain flow which has been widely applied. 
The method is based on an arclength-angle parameterization of the interface which was introduced in \cite{HLS:1994} and first adapted to Stokes flow in \cite{MCAK1:2001}, \cite{MCAK2:2002}.
%to remove numerical stiffness in an efficient manner.
% TODO: added ref [50],[8]

The main task in the proof is to estimate the variations or errors such as  $\dot{\theta}=\theta_i-\theta(\alpha_i)$  between the discrete and exact solutions at time $t$.
This is done by estimating the most singular terms in the variations, and separating into linear and nonlinear terms.
The nonlinear terms are controlled by the high (spectral) accuracy of the method for smooth solutions, and thus the crux of the proof is show the stability of linear terms in the variation, which is done with the aid of energy estimates.

The presence of high derivatives due to the bending forces requires a substantially different analysis from previous proofs of the convergence of the boundary integral method for potential flow.
In particular, our energy estimates make significant use of the smoothing properties of the highest derivative term, or so-called 'parabolic smoothing', to control lower order derivatives.
This allows us to close the energy estimates and prove stability of the method.

The proof also clarifies the role of numerical filtering.
%in the particular boundary integral method analyzed in this thesis.
We find that targeted filtering is necessary to control the potentially destabilizing effect of aliasing errors and prove stability of the method.
Crucially, however, our analysis shows that the filter should not be applied to the highest derivative term coming from the membrane bending stress, so that the smoothing properties of this term can utilized.

Our work also provides a convergence analysis of a widely used boundary integral method for drops and bubbles without a surrounding elastic membrane, in which 
$\MCAL{S}=\mbox{constant}$ and $\kappa_B=0$ is zero.

Another important application of  BI methods is in computing the evolution of vesicles with inextensible membranes, in which $s_\alpha \equiv 1$. 
BI methods for inextensible vesicles have been developed in, e.g.,  \cite{Quaife:2016}, \cite{Veerapaneni:2009}. The inextensibility constraint can be approximately satisfied in our method by choosing a large constant $E$ in (\ref{tension_dimensional}) (and using a different nondimensionalization), which keeps $s_\alpha$ near $1$ \cite{Beale:2012}.  In future work, we may consider the convergence analysis for an algorithm  in which  the tension $\MCAL{S}(\alpha,t)$ is chosen to exactly  enforce the inextensibility constraint. 
%, i.e.,
%\begin{align}
%\MBF{x}_s\cdot\MBF{u}_s(\MBF{x})=0\mbox{, for }x\in\gamma
%\end{align}
%In other future work, we may also consider the convergence of recent algorithms, e.g. \cite{LRZ}, in which artificial contact forces are introduced to prevent fluid drops or vesicles from coming into close contact. Since the contact forces act over a small region that scales with the grid size, it is expected that methods utilizing these forces converge to solutions of the continuous equations (without contact forces). However, no proof or demonstration of convergence currently exists.
% End move * to conclusion?

%\appendix?(check google, worked!)
\appendix
\section{Proof of Lemmas} \label{sec:proof_lemmas}

\textit{Proof of Lemma \ref{S_h_estimate}.} 
Let $\hat{f}_k^e$ to be the exact Fourier coefficient of $f$. Then \cite{ET:1987}
\begin{equation}\label{Fourier_approx_coeff_defn_in_terms_of_exact}
\hat{f}_k=\hat{f}_k^e+\DPS\sum_{j\neq 0}\hat{f}_{k+Nj}^e \mbox{, for }k=-\frac{N}{2}+1,\cdots,\frac{N}{2}\mbox{,}
\end{equation}
is the computed Fourier coefficient from (\ref{APPFT}), where the sum represents high-wave-number modes  that are aliased to $k\in\bigmpa{-\frac{N}{2}+1,\frac{N}{2}}$.  The sum in (\ref{Fourier_approx_coeff_defn_in_terms_of_exact}) is from $j=-\infty$ to $\infty$, excluding $j=0$.   Introduce the notation $\abs{k}\leq\frac{N'}{2}$ defined as $\bigbra{k:-\frac{N}{2}+1 \leq k\leq\frac{N}{2}}$ and  $\abs{k}>\frac{N'}{2}$, which is defined as $\bigbra{k: \frac{N}{2}<k \mbox{ or } k\leq-\frac{N}{2}}$.
Then we have the estimate
\begin{align}\label{S_h_err_total}
\abs{S_h f(\alpha_i)-f_{\alpha}(\alpha_i)}
%&=\Biggabs{
%\DPS\sum_{\abs{k}\leq\frac{N'}{2}} %\bigpa{\bigpa{\widehat{f_{\alpha}}}_k-\bigpa{\widehat{f_{\alpha}}}_k^e} e^{ik\alpha_i}
%-\DPS\sum_{\abs{k}>\frac{N'}{2}} \big(\widehat{f_{\alpha}}\big)_k^e e^{ik\alpha_i}
%}\nonumber \\
&=\Biggabs{
\DPS\sum_{\abs{k}\leq\frac{N'}{2}} k(\hat{f}_k-\hat{f}_k^e) e^{ik\alpha_i}
-\DPS\sum_{\abs{k}>\frac{N'}{2}} k\hat{f}_k^e e^{ik\alpha_i}
}\nonumber \\
&\leq
\DPS\sum_{\abs{k}\leq\frac{N'}{2}} \abs{k}\abs{\hat{f}_k-\hat{f}_k^e}
+\DPS\sum_{\abs{k}>\frac{N'}{2}} \abs{k}\abs{\hat{f}_k^e}\mbox{.}
\end{align}
The first term on the right hand of (\ref{S_h_err_total}) is the aliasing error, and the second term is the truncation error.
We use (\ref{Fourier_approx_coeff_defn_in_terms_of_exact}) to bound the aliasing error as
% First consider the aliasing error
\begin{align}\label{aliasing_err_est_1}
&\DPS\sum_{\abs{k}\leq\frac{N'}{2}}
\abs{k}\abs{\hat{f}_k-\hat{f}_k^e}
=\DPS\sum_{\abs{k}\leq\frac{N'}{2}}
\abs{k} \Biggabs{\DPS\sum_{j\neq 0} \hat{f}_{k+Nj}^e}\nonumber \\
&\leq\DPS\sum_{\substack{\abs{k}\leq\frac{N'}{2}}} \sum_{j \ne 0}
\abs{k+jN} \abs{\hat{f}_{k+Nj}^e}
\leq\DPS\sum_{\abs{\tilde{k}}\geq\frac{N}{2}}
\abs{\tilde{k}} \abs{\hat{f}_{\tilde{k}}^e}\mbox{,}
\end{align}
where $\tilde{k}=k+jN$, with $j\neq 0$.
The last line of (\ref{aliasing_err_est_1}) follows from
\begin{equation}\label{S_h_estimate_index_eqn}
\abs{\tilde{k}}=\abs{k+jN}
\geq\bigabs{\abs{k}-j\abs{N}}
\geq\Bigabs{\frac{N}{2}-N}
\geq\frac{N}{2}\mbox{.}
\end{equation}
The aliasing error is further bounded by (dropping the tilde)
\begin{align}
\DPS\sum_{\abs{k}\geq\frac{N}{2}}
\abs{k}\abs{\hat{f}_k^e}
&\leq
\DPS\sum_{\abs{k}\geq\frac{N}{2}}
\frac{\abs{k}^{s+1}}{\abs{k}^s}\abs{\hat{f}_k^e}
\nonumber \\
&\leq
\Bigpa{\DPS\sum_{\abs{k}\geq\frac{N}{2}}
\abs{k}^{2(s+1)}\abs{\hat{f}_k^e}^2}^{\frac{1}{2}}
\biggpa{\DPS\sum_{\abs{k}\geq\frac{N}{2}}
\frac{1}{\abs{k}^{2s}}}^{\frac{1}{2}}
\nonumber \\
&\leq c\norm{f}_{{s+1}}
\bigpa{N^{-2s+1}}^{\frac{1}{2}}
\nonumber \\
&\leq c h^{s-\frac{1}{2}}\norm{f}_{s+1}\mbox{, using }h=\frac{2\pi}{N}.
\end{align}
In the third inequality above we have used the bound
\begin{equation}
\DPS\sum_{\abs{k}\geq\frac{N}{2}}
\frac{1}{\abs{k}^{2s}}<c N^{-2s+1}\mbox{.}
\end{equation}
%which follow from the integral test:
%\begin{equation}
%\intdef{\frac{N}{2}}{\infty} x^{-2s}dx
%\DPS\sim \Bigparvt{\frac{x^{-2s+1}}{-2s+1}}{\frac{N}{2}}{\infty}
%=c N^{-2s+1}\mbox{.}
%\end{equation}
% truncation error
The truncation error is bounded as (starting from (\ref{S_h_err_total}))
\begin{align}
\DPS\sum_{\abs{k}>\frac{N'}{2}} \abs{k}\abs{\hat{f}_k^e}
&=\DPS\sum_{\abs{k}>\frac{N'}{2}} \frac{\abs{k}^{s+1}}{\abs{k}^s}\abs{\hat{f}_k^e}
\nonumber \\
&\leq \Bigpa{\DPS\sum_{\abs{k}>\frac{N'}{2}} \abs{k}^{2(s+1)}\abs{\hat{f}_k^e}^2}^{\frac{1}{2}}
\Bigpa{\DPS\sum_{\abs{k}>\frac{N'}{2}} \abs{k}^{-2s}}^{\frac{1}{2}} \nonumber \\
&\leq c h^{s-\frac{1}{2}}\norm{f}_{s+1}\mbox{.}
\end{align}
Combine the estimates of aliasing error and truncation error to obtain
\begin{equation}\label{SHERREST}
\abs{S_h f(\alpha_i)-f_{\alpha}(\alpha_i)}
\leq c h^{s-\frac{1}{2}} \norm{f}_{s+1}\mbox{.}
\end{equation}
The proof of (\ref{SHERREST}) for $D_h$ instead of $S_h$ is similar.

\textit{Proof of Lemma \ref{Int_product_rule_f_S_h_g_lemma}.}
%Define
%\begin{align}
%\hat{f}_k=\frac{1}{2\pi}
%\intdef{-\pi}{\pi} f(\alpha)e^{-ik\alpha}d\alpha\mbox{,}
%\end{align}
%to be  the exact Fourier coefficients of a smooth periodic function $f(\alpha)$ and 
Let
\begin{align}
\hat{F}_k=\DPS\sum_{m=-\infty}^{\infty}
\hat{f}_{k+mN}^e\mbox{ for }
-\frac{N}{2}+1\leq k\leq \frac{N}{2}
\end{align}
denote the discrete Fourier coefficients of a smooth periodic function  $f(\alpha)$ (taking aliasing into account),  where $\hat{f}^e_k$ are the exact Fourier coefficients of $f$.
Thus, $f(\alpha_i)$ and the discrete function $\phi_i$ have the representation
\begin{align}
f(\alpha_i)&=\DPS\sum_{k=-\frac{N}{2}+1}^{\frac{N}{2}}
\hat{F}_k e^{ik\alpha_i}\mbox{,}\label{Fourier_expansion_f}\\
\phi_i&=\DPS\sum_{k=-\frac{N}{2}+1}^{\frac{N}{2}}
\hat{\phi}_k e^{ik\alpha_i}\mbox{.}\label{Fourier_expansion_phi}
\end{align}
Our interest is in obtaining an estimate for
\begin{align}\label{num_approx_Int_prod_f_S_h_phi}
S_h^{-1}(f(\alpha_i)S_h \phi_i)
=\DPS\sum_{\substack{k=-\frac{N}{2}+1\\k\neq 0}}^{\frac{N}{2}}
\frac{1}{ik}\widehat{(f S_h \phi)}_k
e^{ik\alpha_i}\mbox{,}
\end{align}
where $\widehat{(f S_h \phi)}_k$ denotes the discrete Fourier coefficients of the product $f S_h\phi$.

We shall need an expression for the Fourier coefficients of the product of a smooth function with a discrete function.
For a given $k$, define the sets
\begin{align}
I_{n,k}&=\Bigbra{n\in [-N/2+1,N/2]:-\frac{N}{2}+1\leq k-n\leq\frac{N}{2}}\mbox{,}\nonumber\\
J_{n,k}&=\Bigbra{n\in  [-N/2+1,N/2]:-\frac{N}{2}+1\leq k+N-n\leq\frac{N}{2}}\mbox{,}\nonumber\\
K_{n,k}&=\Bigbra{n\in  [-N/2+1,N/2]:-\frac{N}{2}+1\leq k-N-n\leq\frac{N}{2}}\mbox{.}
\end{align}
Using (\ref{Fourier_expansion_f}) and (\ref{Fourier_expansion_phi}), the product $f(\alpha_i) \phi_i$ can be written as
\begin{align}\label{num_approx_prod_f_alpha_phi_decomp}
f(\alpha_i)\phi_i
=\Biggpa{
\DPS\sum_{k=-N+2}^{-\frac{N}{2}}
+\DPS\sum_{k=-\frac{N}{2}+1}^{\frac{N}{2}}
+\DPS\sum_{k=\frac{N}{2}+1}^N
}
\DPS\sum_{n\in I_{n,k}}
\hat{F}_{k-n}\hat{\phi}_n e^{ik\alpha_i}
\mbox{.}
\end{align}
% begin insert 21
%where we use the notation $\DPS\sum_{n\in I_{n,k}}$ to represent $\DPS\sum_{\substack{n=-\frac{N}{2}+1\\n\in I_{n,k}}}^{\frac{N}{2}}$.
%Equivalently, $\hat{\phi}_n$ is set to zero for $n$ outside the range $[-\frac{N}{2}+1,\frac{N}{2}]$ (this is known as 'zero padding').
% end insert 21
The wave numbers in the first and third sums in parenthesis are aliased to $k\in [-\frac{N}{2}+1,\frac{N}{2}]$.
Rewriting these two sums by replacing $k$ with $k-N$ and $k+N$, respectively, we obtain the equivalent representation
\begin{align}\label{num_approx_prod_f_phi}
f(\alpha_i)\phi_i
&=\DPS\sum_{k=-\frac{N}{2}+1}^{\frac{N}{2}}
\Biggpa{
\DPS\sum_{n\in I_{n,k}}
\hat{F}_{k-n}\hat{\phi}_n
+\DPS\sum_{n\in J_{n,k}}
\hat{F}_{k+N-n}\hat{\phi}_n\nonumber\\
+&\DPS\sum_{n\in K_{n,k}}
\hat{F}_{k-N-n}\hat{\phi}_n
}e^{ik\alpha_i}\mbox{,}
\end{align}
where the requirement $n\in J_{n,k}$ in the second double sum of (\ref{num_approx_prod_f_phi}) and $n\in K_{n,k}$ in the third has allowed us to replace $\DPS\sum_{k=-\frac{N}{2}+1}^0$ and $\DPS\sum_{k=2}^{\frac{N}{2}}$, respectively, with $\DPS\sum_{k=-\frac{N}{2}+1}^{\frac{N}{2}}$.
From (\ref{num_approx_prod_f_phi}), we therefore have for $-\frac{N}{2}+1\leq k\leq\frac{N}{2}$
\begin{align}\label{Fourier_coeff_f_S_h_phi}
\widehat{(f S_h \phi)}_k
&=\DPS\sum_{n\in I_{n,k}}
in\hat{F}_{k-n}\hat{\phi}_n
+\DPS\sum_{n\in J_{n,k}}
in\hat{F}_{k+N-n}\hat{\phi}_n\nonumber\\
+&\DPS\sum_{n\in K_{n,k}}
in\hat{F}_{k-N-n}\hat{\phi}_n\mbox{,}
\end{align}
and similarly,
\begin{align}\label{Fourier_coeff_phi_S_h_f}
\widehat{(\phi S_h f)}_k
&=\DPS\sum_{n\in I_{n,k}}
i(k-n)\hat{F}_{k-n}\hat{\phi}_n
+\DPS\sum_{n\in J_{n,k}}
i(k+N-n)\hat{F}_{k+N-n}\hat{\phi}_n\nonumber\\
+&\DPS\sum_{n\in K_{n,k}}
i(k-N-n)\hat{F}_{k-N-n}\hat{\phi}_n\mbox{.}
\end{align}
Combining (\ref{Fourier_coeff_f_S_h_phi}) with the negative of (\ref{Fourier_coeff_phi_S_h_f}) gives for $-\frac{N}{2}+1\leq k\leq\frac{N}{2}$
\begin{align} \label{f_Shphi_hat}
\widehat{(f S_h \phi)}_k&=-\widehat{(\phi S_h f)}_k
+\DPS\sum_{n\in I_{n,k}}
ik\hat{F}_{k-n}\hat{\phi}_n\nonumber\\
+&\DPS\sum_{n\in J_{n,k}}
i(k+N)\hat{F}_{k+N-n}\hat{\phi}_n
+\DPS\sum_{n\in K_{n,k}}
i(k-N)\hat{F}_{k-N-n}\hat{\phi}_n\mbox{.}
\end{align}
We next recognize from (\ref{num_approx_prod_f_phi}) that
\begin{align}\label{Fourier_coeff_f_phi}
\widehat{(f \phi)}_k
&=\DPS\sum_{n\in I_{n,k}}
\hat{F}_{k-n}\hat{\phi}_n
+\DPS\sum_{n\in J_{n,k}}
\hat{F}_{k+N-n}\hat{\phi}_n\nonumber\\
+&\DPS\sum_{n\in K_{n,k}}
\hat{F}_{k-N-n}\hat{\phi}_n\mbox{,}
\end{align}
% TODO: try to combine the eqn lines, to see whether it works or not!!!
and combining this with (\ref{f_Shphi_hat}) shows that the Fourier coefficients in (\ref{num_approx_Int_prod_f_S_h_phi}) can be written as
\begin{align}\label{Fourier_coeff_Int_f_S_h_phi}
\frac{\widehat{(f S_h \phi)}_k}{ik}&=-\frac{\widehat{(\phi S_h f)}_k}{ik}
+\widehat{(f\phi)}_k
\nonumber\\
+&\DPS\sum_{n\in J_{n,k}}
\frac{N}{k}\hat{F}_{k+N-n}\hat{\phi}_n
-\DPS\sum_{n\in K_{n,k}}
\frac{N}{k}\hat{F}_{k-N-n}\hat{\phi}_n\mbox{,}
\end{align}
% TODO: try to combine the eqn lines, to see whether it works or not!!!
for $-\frac{N}{2}+1\leq k\leq\frac{N}{2}$ with $k\neq 0$.
The first three terms in (\ref{Fourier_coeff_Int_f_S_h_phi}) are the Fourier coefficients of the first three terms in (\ref{num_approx_Int_prod_f_S_h_phi_1}).
To finish the derivation of (\ref{num_approx_Int_prod_f_S_h_phi_1}), we simply need to estimate the two sums on the right hand side of (\ref{Fourier_coeff_Int_f_S_h_phi}).
The main difficulty is to overcome the large factor of $N$.

Each of the two sums in (\ref{Fourier_coeff_Int_f_S_h_phi}) is a discrete convolution which represents the $k$-th Fourier coefficient of the product a smooth function with $\phi_i$.
Denoting the smooth functions by $f_1$ and $f_2$, we form the $l^2$-norm of the products with the aid of the discrete Parseval equality (Lemma \ref{IP_prod_lemma}),
\begin{align}\label{norm_f_1_prod_phi}
\norm{f_1(\cdot)\phi}_{l^2}
=\Biggpa{
2\pi\DPS\sum_{k=-\frac{N}{2}+1}^{\frac{N}{2}}
\left|
\DPS\sum_{n\in J_{n,k}}
\frac{N}{k}\hat{F}_{k+N-n}\hat{\phi}_n
\right|^2
}^{\frac{1}{2}}\mbox{,}
\end{align}
and
\begin{align}\label{norm_f_2_prod_phi}
\norm{f_2(\cdot)\phi}_{l^2}
=\Biggpa{
2\pi\DPS\sum_{k=-\frac{N}{2}+1}^{\frac{N}{2}}
\left|
\DPS\sum_{n\in K_{n,k}}
\frac{N}{k}\hat{F}_{k-N-n}\hat{\phi}_n
\right|^2
}^{\frac{1}{2}}\mbox{.}
\end{align}
We estimate these $l^2$-norms by decomposing the wavenumber range into
\begin{align}
\MCAL{\varkappa}_1
=\Bigbra{k: \frac{N}{4}\leq \abs{k}\leq\frac{N}{2}}\mbox{,}
\end{align}
and
\begin{align}\label{var_kappa_2_set_defn}
\MCAL{\varkappa}_2
=\Bigbra{k: 0<\abs{k}<\frac{N}{4}}\mbox{.}
\end{align}
The sum over $k\in\varkappa_1$ is bounded using $\abs{\frac{N}{k}}\leq 4$.
For example,
\begin{align}
\Biggmpa{
2\pi\DPS\sum_{k\in\varkappa_1}
\left|\frac{N}{k}
\DPS\sum_{n\in J_{n,k}}
\hat{F}_{k+N-n}\hat{\phi}_n
\right|^2
}^{\frac{1}{2}}
&\leq4\Biggmpa{
2\pi\DPS\sum_{k\in\varkappa_1}
\left|
\DPS\sum_{n\in J_{n,k}}
\hat{F}_{k+N-n}\hat{\phi}_n
\right|^2
}^{\frac{1}{2}}\nonumber\\
&\leq4\Biggmpa{
2\pi\DPS\sum_{k=-\frac{N}{2}+1}^{\frac{N}{2}}
\left|
\DPS\sum_{n\in J_{n,k}}
\hat{F}_{k+N-n}\hat{\phi}_n
\right|^2
}^{\frac{1}{2}}\mbox{}\label{Decomp_wavenumber_1}\\
&=4\Biggmpa{
2\pi\DPS\sum_{k=\frac{N}{2}+1}^N
\DPS\sum_{n\in I_{n,k}}
\left|
\hat{F}_{k-n}\hat{\phi}_n
\right|^2
}^{\frac{1}{2}}\mbox{,}\label{Decomp_wavenumber_2}
\end{align}
where in the latter equality we have first replaced $\DPS\sum_{k=-\frac{N}{2}+1}^{\frac{N}{2}}$ in (\ref{Decomp_wavenumber_1}) with $\DPS\sum_{k=-\frac{N}{2}+1}^0$
(per the comment following (\ref{num_approx_prod_f_phi}))
and then substituted $k-N$ for $k$.
The expression in (\ref{Decomp_wavenumber_2}) is clearly bounded by a constant times the extended $l^2$ norm of (\ref{num_approx_prod_f_alpha_phi_decomp}),
\begin{align}
\norm{f(\cdot)\phi}_{l^2_{ext}}
\equiv\Biggmpa{
2\pi\DPS\sum_{k=-N+2}^N
\left|
\DPS\sum_{n\in I_{n,k}}
\hat{F}_{k-n}\hat{\phi}_n
\right|^2
}^{\frac{1}{2}}\mbox{,}
\end{align}
for which
\begin{align}
\norm{f(\cdot)\phi}_{l^2_{ext}}
\leq \norm{f}_{\infty}\norm{\phi}_{l^2_{ext}}
=\norm{f}_{\infty}\norm{\phi}_{l^2}\mbox{,}
\end{align}
(where $\norm{\phi}_{l^2_{ext}}$ is defined by zero padding).

Hence, in (\ref{norm_f_1_prod_phi}) (and similarly in (\ref{norm_f_2_prod_phi})), the sum over $k\in\varkappa_1$ is bounded by $c\norm{\phi}_{l^2}$.

The sum over $k\in\varkappa_2$ 
%for (\ref{norm_f_1_prod_phi}) and (\ref{norm_f_2_prod_phi}) 
requires a different estimate.
%For example, consider (\ref{norm_f_1_prod_phi}).
For $k\in\varkappa_2$, we have
\begin{align}\label{wavenumber_condition_10_1}
k+N-n\geq \frac{N}{4}\mbox{ when }n\in\Bigmpa{-\frac{N}{2}+1,\frac{N}{2}}\mbox{.}
\end{align}
% TODO: make shortcuts for \bigleftparightmpa, e.g. (a,b], etc
Moreover, $\hat{f}_k^e$ decay like $\mbox{O}(k^{-s})$, where $s$ is the number of continuous derivatives of $f$, and it is easily seen that when $s>1$, $\hat{F}_k$ also decays like $\mbox{O}(k^{-s})$.
\newline
Hence,
\begin{align}\label{Fourier_coeff_F_k_plus_N_minus_n_bound}
\abs{\hat{F}_{k+N-n}}
\leq c\abs{k+N-n}^{-s}
\leq c\Bigpa{\frac{N}{4}}^{-s}\mbox{,}
\end{align}
per (\ref{wavenumber_condition_10_1}).
Considering (\ref{norm_f_1_prod_phi}), it follows that the sum over $k\in\varkappa_2$ satisfies the bound
\begin{align}
\Biggmpa{
2\pi\DPS\sum_{k\in\varkappa_2}
\left|
\DPS\sum_{n\in J_{n,k}}
\frac{N}{k}
\hat{F}_{k+N-n}\hat{\phi}_n
\right|^2
}^{\frac{1}{2}}
\leq&cN^{-s+1}\Biggmpa{
\DPS\sum_{\abs{k}\leq\frac{N}{4}}
\left|
\DPS\sum_{n\in J_{n,k}}
\hat{\phi}_n
\right|^2
}^{\frac{1}{2}}\nonumber\\
\leq&cN^{-s+2}\Biggmpa{
\DPS\sum_{\abs{k}\leq\frac{N}{4}}
\DPS\sum_{n\in J_{n,k}}
\left| \hat{\phi}_n \right|^2
}^{\frac{1}{2}}\nonumber\\
\leq&cN^{-s+\frac{5}{2}}\norm{\phi}_{l^2}\mbox{.}
\end{align}
Here we have used (\ref{Fourier_coeff_F_k_plus_N_minus_n_bound}) in the first inequality,
\begin{align}
\left| \DPS\sum_{n \in J_{n,k}} \phi_n \right|^2\leq N^2\DPS\sum_{n \in J_{n,k}} \left| \phi_n \right|^2
\end{align}
in the second, and replaced the double sum  $\DPS\sum_{\abs{k}\leq\frac{N}{4}} \sum_{n \in J_{n,k}}$ by $\frac{N}{4} \DPS\sum_{n=-N/2+1}^{N/2}$ in the third.
It follows that when $s\geq \frac{5}{2}$, the sum over $k\in\varkappa_2$ in (\ref{norm_f_1_prod_phi}) and similarly in (\ref{norm_f_2_prod_phi}) are bounded by $c\norm{\phi}_{l^2}$.
%The sum over $k\in\varkappa_2$ in (\ref{norm_f_2_prod_phi}) is bounded similarly.
Thus, both (\ref{norm_f_1_prod_phi}) and (\ref{norm_f_2_prod_phi}) are bounded by $c\norm{\phi}_{l^2}$, which finishes our estimate of the two sums in (\ref{Fourier_coeff_Int_f_S_h_phi}).
This completes the derivation of (\ref{num_approx_Int_prod_f_S_h_phi_1}).

Equation (\ref{num_approx_Int_prod_phi_S_h_f_A}) readily follows by nothing that both $D_h S_h^{-1}(f(\alpha_i)\phi_i)$ and $S_h^{-1}(f(\alpha_i)S_h\phi_i)$ are $A_0(\phi_i)$ operators,  the latter of which  is a consequence of (\ref{num_approx_Int_prod_f_S_h_phi_1}).
%and the estimate
%\begin{align}
%\norm{S_h^{-1}(f(\cdot)\phi)}_{l^2}
%&=\Bigpa{
%\DPS\sum_{k=-\frac{N}{2}+1}^{\frac{N}{2}}
%\frac{\widehat{(f\phi)}_k^2}{\abs{k}^2}
%}^{\frac{1}{2}}\nonumber\\
%&\leq \Bigpa{
%\DPS\sum_{k=-\frac{N}{2}+1}^{\frac{N}{2}}
%\widehat{(f\phi)}_k^2
%}^{\frac{1}{2}}\nonumber\\
%&=\norm{f\phi}_{l^2}\nonumber\\
%&\leq \norm{f}_{\infty}\norm{\phi}_{l^2}\mbox{,}
%\end{align}
%where $\widehat{(f\phi)}_k$ is given by (\ref{Fourier_coeff_f_phi}).
%%%%%%%%%%%%%%%%%%%%%%%%%%%%%%%%%%%%%%%%%%%%
%%%%%%%%%%%%%%%%%%%%%%%%%%%%%%%%%%%%%%%%%%%%

\textit{Proof of Lemma \ref{I_plus_beta_K_matrix_lemma_6_1}.}
First, we provide the result for the corresponding continuous equation, which is denoted by
\begin{align}\label{I_plus_beta_K_c}
(\MBF{I}+\beta\MBF{K_c})\bm{\tilde{\omega}}(\alpha)
=-\beta \MBF{K}_c \MBF{g}(\alpha) \mbox{.}
\end{align}
This system can be solved by the method of successive approximations for sufficiently small $\beta>0$.
More precisely, there exists a value $\beta_0(T)$  such that for $0 \leq \beta<\beta_0(T)$ the continuous operator $\bigpa{I+\beta \MBF{K}_c}$ is invertible and has a bounded inverse \cite{SGM:1964}, i.e.,
\begin{align}
\norm{
\bigpa{I+\beta \MBF{K}_c}^{-1}
}_{L^2}
\leq c
\mbox{.}
\end{align}
We summarize the argument, following the analysis in \cite{BHL:1996}, that the discrete operator $I+\beta \MBF{K}$ is likewise invertible.

For any discrete $l^2$ function ${\omega}_i$, define
\begin{align}\label{Phi_i_eqn_3_1}
\bm{\Phi}_i=\bm{\omega}_i+
\beta \MBF{K}
\bm{\omega}_i
\end{align}
where $\bm{\omega}_i$ is the vector counterpart of $\omega_i$ (see (\ref{g_omega_omegatil_def})).
We show that $\bm{\Phi}=0$ implies $\bm{\omega}=0$, which demonstrates the invertibility of $I+\beta \MBF{K}$.
The strategy is to exactly represent the discrete equation (\ref{Phi_i_eqn_3_1}) as a continuous (integral) equation with a piecewise continuous integrand.
The invertibility result for a continuous equation then will imply the invertibility of the discrete equation.
This construction is essentially the same as in the appendix of \cite{BHL:1996}, and  relies   on the smoothness of the continuous
kernel $\MBF{K}_c(\alpha, \alpha')$ and the consistency of our discretization. We refer the reader to \cite{BHL:1996} for details.

\section{Estimates for the nonlinear terms in the velocity variation}

We first present expressions for the nonlinear terms in the variation of the velocity, i.e.,  $\dot{U}_{l,i}^{NL}$   for $l=1, ..., 3$.
The variation $\dot{U}_{1,i}^{NL}$ is given by
\begin{align}\label{dot_U_1i_NL}
\dot{U}_{1,i}^{NL}&=-\frac{h}{\pi}
\DPS\sum_{\substack{j=-\frac{N}{2}+1\\(j-i)\mbox{ odd}}}^{\frac{N}{2}}
\biggbra{\dot{\omega}_j^p \biggpa{
2\mbox{Re}\Bigpa{
\frac{S_h\tau_j}{\tau_j-\tau_i}
}^{\cdot}
}
+\omega_h^p(\alpha_j) \biggmpa{
2\mbox{Re}\biggpa{S_h \dot{\tau}_j
\Bigpa{\frac{1}{\tau_j-\tau_i}}^{\cdot}
}
}\nonumber \\
&+\omega_h^p(\alpha_j)2\mbox{Re}\biggpa{
\frac{S_h\tau_h(\alpha_j)
\bigpa{\dot{\tau}_j-\dot{\tau}_i}^2}{[\tau_h(\alpha_j)-\tau_h(\alpha_i)]^2
\bigpa{\tau_h(\alpha_j)-\tau_h(\alpha_i)+\dot{\tau}_j-\dot{\tau}_i}}
}
}\mbox{.}
\end{align}
Note that the third term within braces above comes from the nonlinear term in the variation of $\bigpa{\frac{1}{\tau_j-\tau_i}}^{\cdot}$, via (\ref{quotient_eqn_f}).

To compactly represent the other nonlinear terms, introduce the notation
\begin{align}
[f_i,g_i,h_i]^{\cdot}=\dot{f}_i \dot{g}_i h(\alpha_i)
+\dot{f}_i g(\alpha_i) \dot{h}_i
+f(\alpha_i) \dot{g}_i \dot{h}_i
+\dot{f}_i \dot{g}_i \dot{h}_i\mbox{,}
\end{align}
which gives the nonlinear terms in the variation of the product $f_ig_ih_i$.
A similar notation is used for the nonlinear terms in the variation of a product with four or more discrete functions.
Then,
\begin{align}\label{dot_U_2i_NL}
\dot{U}_{2,i}^{NL}&=\frac{h}{\pi}
\DPS\sum_{\substack{j=-\frac{N}{2}+1\\(j-i)\mbox{ odd}}}^{\frac{N}{2}}
\Biggbra{
\Bigmpa{\OVL{\omega}_j^p,S_h\tau_j,\frac{1}{\OVL{\tau_j}-\OVL{\tau_i}}}^{\cdot}
\nonumber \\
&+\frac{\OVL{\omega^p_h}(\alpha_j) S_h\tau_h(\alpha_j)
\bigpa{\OVL{\dot{\tau}_j}-\OVL{\dot{\tau}_i}}^2}{[\OVL{\tau_h}(\alpha_j)
-\OVL{\tau_h}(\alpha_i)]^2
\bigpa{ \OVL{\tau_h}(\alpha_j)-\OVL{\tau_h}(\alpha_i)+\OVL{\dot{\tau_j}}-\OVL{\dot{\tau_i}} }  } 
}\mbox{,}
\end{align}
where for example in $\bigmpa{\OVL{\omega}_j^p,S_h\tau_j,\frac{1}{\OVL{\tau_j}-\OVL{\tau_i}}}^{\cdot}$, if $g_j=S_h\tau_j$, then $g(\alpha_j)=S_h\tau_h(\alpha_j)$.
Continuing,
\begin{align}\label{dot_U_3i_NL}
\dot{U}_{3,i}^{NL}&=-\frac{h}{\pi}
\DPS\sum_{\substack{j=-\frac{N}{2}+1\\(j-i)\mbox{ odd}}}^{\frac{N}{2}}
\Biggbra{
\Bigmpa{\OVL{\omega_j^p},\tau_j-\tau_i,\OVL{S_h\tau_j},
\frac{1}{\OVL{\tau_j}-\OVL{\tau_i}},\frac{1}{\OVL{\tau_j}-\OVL{\tau_i}}
}^{\cdot}
\nonumber \\
&+\frac{\OVL{\omega_h^p}(\alpha_j)\bigmpa{\tau_h(\alpha_j)-\tau_h(\alpha_i)}\OVL{S_h\tau_h}(\alpha_j)
\bigmpa{\OVL{\dot{\tau}_j}-\OVL{\dot{\tau}_i}}^2
}{
[\OVL{\tau_h}(\alpha_j)-\OVL{\tau_h}(\alpha_i)]^2
\bigpa{ \OVL{\tau_h}(\alpha_j)-\OVL{\tau_h}(\alpha_i)+\OVL{\dot{\tau}_j}-\OVL{\dot{\tau}_i} }
}
}\mbox{.}
\end{align}

We now estimate these nonlinear terms.
%(\ref{dot_U_1i_NL}), (\ref{dot_U_2i_NL}) and (\ref{dot_U_3i_NL}) in the velocity variation.
Consider first the expression (\ref{dot_U_1i_NL}) for $\dot{U_1}^{NL}$.
We expand some of the variations in this expression using Lemmas \ref{product_rule_for_dots_lemma} and \ref{quotient_rule_1_over_f_lemma}, for example,
\begin{align}\label{half_ln_dot_appendix_eqn_2_1}
\Bigpa{\frac{S_h\tau_j}{\tau_j-\tau_i}}^{\cdot}
&=-\frac{\tau_{\alpha}(\alpha_j)}{(\tau(\alpha_j)-\tau(\alpha_i))^2}
\bigpa{\dot{\tau}_j-\dot{\tau}_i}
+\frac{S_h\dot{\tau}_j}{\tau(\alpha_j)-\tau(\alpha_i)}\nonumber\\
&+\frac{\tau_{\alpha}(\alpha_j)(\dot{\tau}_j-\dot{\tau}_i)^2}{(\tau(\alpha_j)-\tau(\alpha_i))^2(\tau(\alpha_j)-\tau(\alpha_i)+\dot{\tau}_j-\dot{\tau}_i)}\nonumber\\
&+S_h\dot{\tau}_j \Bigpa{
\frac{1}{\tau_j-\tau_i}
}^{\cdot}
+\mbox{O}(h^s)\mbox{.}
\end{align}
To estimate the terms involving differences of variations in (\ref{half_ln_dot_appendix_eqn_2_1}), we make use of the Fourier series representation
\begin{align}
\frac{\dot{\tau}_j-\dot{\tau}_i}{\alpha_j-\alpha_i}
&=\DPS\sum_{k=-\frac{N}{2}+1}^{\frac{N}{2}}
\hat{\dot{\tau}}_k
\frac{e^{ik\alpha_j}-e^{ik\alpha_i}}{\alpha_j-\alpha_i}\mbox{,}
\end{align}
so that
% Begin insert A0
\begin{align}\label{zeta_over_alpha_dot_appendix_eqn_3_1}
\Bigabs{\frac{\dot{\tau}_j-\dot{\tau}_i}{\alpha_j-\alpha_i}}
&\leq \DPS\sum_{k=-\frac{N}{2}+1}^{\frac{N}{2}}
\abs{\hat{\dot{\tau}}_k}
\biggabs{\frac{e^{ik\alpha_j}-e^{ik\alpha_i}}{\alpha_j-\alpha_i}}
\mbox{,}
\nonumber\\
&\leq \DPS\sum_{k=-\frac{N}{2}+1}^{\frac{N}{2}}
|k|\abs{\hat{\dot{\tau}}_k}
\mbox{,}
\nonumber\\
&\leq \Bigpa{
\DPS\sum_{k=-\frac{N}{2}+1}^{\frac{N}{2}}
1}^{\frac{1}{2}}
\norm{S_h\dot{\tau}}_{l^2}
\mbox{,}
\nonumber\\
&\leq \frac{c}{h^{\frac{1}{2}}}\norm{S_h\dot{\tau}}_{l^2}\mbox{,}
\end{align}
for any $-\frac{N}{2}+1\leq j\leq \frac{N}{2}$ and $-\frac{N}{2}+1\leq i\leq \frac{N}{2}$.
%where in the third inequality we have used the Cauchy Schwartz inequality.
% End insert A0
The first term on the right hand side of (\ref{half_ln_dot_appendix_eqn_2_1}) is estimated by multiplying and dividing by $\bigpa{\alpha_j-\alpha_i}^2$, and using (\ref{zeta_over_alpha_dot_appendix_eqn_3_1}) to obtain,
% Begin insert A1
\begin{align}\label{zeta_alpha_zeta_over_zeta_minus_alpha_appendix_eqn_4_1}
\biggabs{\frac{\tau_{\alpha}(\alpha_j)(\dot{\tau}_j-\dot{\tau}_i)}{(\tau(\alpha_j)-\tau(\alpha_i))^2}}
\leq \frac{c}{h^{\frac{3}{2}}}\norm{S_h\dot{\tau}}_{l^2}\mbox{,}
\end{align}
% End insert A1
The extra factor $\frac{1}{h}$ comes from an extra factor of $\frac{1}{\alpha_j-\alpha_i}$.
% Begin insert A2
It is easy to see that the second, third and fourth terms on the right hand side of (\ref{half_ln_dot_appendix_eqn_2_1}) are also bounded in magnitude by $\frac{c}{h^{\frac{3}{2}}}\norm{S_h\dot{\tau}}_{l^2}$.
% End insert A2

Returning to  the expression for $\dot{U_1}^{NL}$ in (\ref{dot_U_1i_NL}), it follows  that the first sum of the right hand side of (\ref{dot_U_1i_NL}) is bounded by
% Begin insert A3
\begin{align}\label{S_h_zeta_j_over_zeta_j_minus_i_bd_57_b_1}
c h^2\norm{\dot{\omega}^p}_{l^2}
+\mbox{O}(h^s)\mbox{,}
\end{align}
using the bound $\frac{c}{h^{\frac{3}{2}}}\norm{S_h\dot{\tau}}_{l^2}$ on the magnitude of (\ref{half_ln_dot_appendix_eqn_2_1}) along with $\norm{S_h\dot{\tau}}_{l^2}=\mbox{O}(h^{\frac{7}{2}})$(see Remark \ref{S_h_dot_zeta_i_remark}) to obtain the factor of $h^2$ in (\ref{S_h_zeta_j_over_zeta_j_minus_i_bd_57_b_1}).
It is easy to see that the second and third sums in equation (\ref{dot_U_1i_NL}) are bounded by $h^\frac{3}{2} \norm{S_h\dot{\tau}}_{l^2}$, which, in view of (\ref{S_h_dot_zeta_i_A_estimate_eqn_174}), implies that these terms are $h^\frac{3}{2} A_0(\dot{\theta})+A_{-s}(\dot{\sigma})+\mbox{O}(h^s)$.
Putting these estimates together, we find that
\begin{align}
\dot{U}_{1,i}^{NL}=
A_0(\dot{\theta})+A_{-1}(\dot{\alpha}_{0i})+A_{-s}(\dot{\sigma})+\mbox{O}(h^s)\mbox{.}
\end{align}
using the estimate (\ref{dot_omega^p_err_decomp}) for $\dot{\omega}^p$.
% End insert A3
Estimates for $\dot{U_2}^{NL}$ and $\dot{U_3}^{NL}$ are performed similarly to $\dot{U_1}^{NL}$, and verify that
\begin{align}
\dot{u}^{NL}=A_0(\dot{\theta})+A_{-1}(\dot{\alpha}_{0i})+A_{-s}(\dot{\sigma})+\mbox{O}(h^s)\mbox{.}
\end{align}
%Finally, $A_{-2}(\dot{\omega})$ can be replaced using \ref{dot_omega_err_decomp_A_268}, which gives the desired estimate
%\begin{align}
%\dot{u}^{NL}=
%A_0(\dot{\theta})+A_{-1}(\dot{\alpha_0})
%+h^2\dot{\tau_c}
%+\mbox{O}(h^s)\mbox{.}
%\end{align}
% TODO: insert proof of lemma 6-1 here!!!

\bibliography{Ref_siegel}{}
\bibliographystyle{plain}
\end{document}